\documentclass[11pt]{article}
\usepackage{amssymb,amsmath}
\usepackage{latexsym}

\newfont{\bb}{msbm10 at 11pt}
\newfont{\bbsmall}{msbm8 at 8pt}
\usepackage{color}
\usepackage{graphicx}

\def\R{\mathbb{R}}
\def\C{\mathbb{C}}
\def\N{\mathbb{N}}
\def\B{\mathbb{B}}
\def\D{\mathbb{D}}
\def\Z{\mathbb{Z}}

\newcommand{\esf}{\mbox{\bb S}}
\newcommand{\Te}{\mbox{\bb T}}
\newcommand{\Nsmall}{\mbox{\bbsmall N}}

\newcommand{\esfsmall}{\mbox{\bbsmall S}}
\newcommand{\Zsmall}{\mbox{\bbsmall Z}}

\newcommand{\Csmall}{\mbox{\bbsmall C}}

\newcommand{\rth}{\R^3}
\newcommand{\wt}{\widetilde}

\def\a{{\alpha}}

\def\g{{\gamma}}
\def\G{{\Gamma}}
\def\l{{\lambda}}

\def\de{{\delta}}
\def\be{{\beta}}
\def\ve{{\varepsilon}}

\def\centerbmp#1#2#3{\vskip#2\relax\centerline{\hbox to#1{\special
    {bmp:#3 x=#1, y=#2}\hfil}}}

\newtheorem{theorem}{Theorem}[section]
\newtheorem{lemma}[theorem]{Lemma}
\newtheorem{proposition}[theorem]{Proposition}

\newtheorem{remark}[theorem]{Remark}
\newtheorem{corollary}[theorem]{Corollary}
\newtheorem{definition}[theorem]{Definition}

\newtheorem{assertion}[theorem]{Assertion}

\newenvironment{proof}
{\smallskip\noindent{\it Proof.}\hskip \labelsep}
                          {\hfill\penalty10000\raisebox{-.09em}{$\Box$}\par\medskip}

\begin{document}

\begin{title}
{Properly embedded minimal planar domains}
\end{title}
\vskip .5in

\begin{author}
{William H. Meeks III\thanks{This material is based upon
  work for the NSF under Award No. DMS-1309236.
Any opinions, findings, and conclusions or recommendations
  expressed in this publication are those of the authors and do not
  necessarily reflect the views of the NSF.}
  \and Joaqu\'\i n P\' erez$^\dag $
\and Antonio Ros\thanks{Research partially supported by a Spanish MEC-FEDER Grant
no. MTM2007-61775 and MTM2011-22547, and a Regional J. Andaluc\'\i a Grant no.
P06-FQM-01642.}, }
\end{author}

\maketitle

\begin{abstract}
In 1997, Collin~\cite{col1} proved that any properly embedded
minimal surface in $\rth$ with finite topology and more than one end
has finite total Gaussian curvature.  Hence, by an earlier result of
L\'opez and Ros~\cite{lor1},  catenoids are the only non-planar,
non-simply connected, properly embedded, minimal planar domains in
$\rth$ of finite topology.  In 2005, Meeks and Rosenberg~\cite{mr8}
proved that the only simply connected, properly embedded minimal
surfaces in $\rth$ are planes and helicoids. Around 1860, Riemann
defined a one-parameter family of
periodic, infinite topology, properly embedded, minimal planar
domains ${\cal R}_t$ in $\rth$, $t \in (0,\infty)$.
These surfaces are called the {\it Riemann minimal examples},
and the family $\{ {\cal R}_t\} _t$ has natural limits being a vertical catenoid
as $t\to~0$, and a vertical helicoid as $t\to\infty$.
In this paper we complete the classification of properly embedded, minimal planar
domains in $\rth$ by proving that the only connected examples with infinite
topology are the Riemann minimal examples. We also prove that the
limit ends of Riemann minimal examples are model surfaces for
the limit ends of properly embedded minimal surfaces $M\subset\rth$
of finite genus and infinite topology, in the sense that such an $M$
has two limit ends, each of which has a representative which is
naturally asymptotic to a limit end representative of a Riemann
minimal example with the same associated flux vector.

\vspace{.1cm} \noindent{\it Mathematics Subject Classification:}
Primary 53A10,
 Secondary 49Q05, 53C42, 37K10.

\vspace{.1cm} \noindent{\it Key words and phrases:} Minimal surface,
finite total curvature, Jacobi function, Shiffman function,
Korteweg-de Vries equation, KdV hierarchy, algebro-geometric potential,
stability, index of
stability, curvature estimates,
 limit tangent plane at infinity.
\end{abstract}

\pagebreak

\tableofcontents

\section{Introduction.}
The history of minimal surfaces in $\rth$ begins with the discovery of the classical
examples found in the 18-th and 19-th centuries. The first important result in this
direction is due to Euler~\cite{eul}, who proved in 1741 that when a small arc on
the catenary $x_1=\cosh x_3$ is rotated around the $x_3$-axis, then one obtains a
surface which minimizes area among all surfaces of revolution with the same boundary
circles. The entire surface of revolution of $x_1=\cosh x_3$ was initially called the
{\it alysseid} but since Plateau's time, called the catenoid.

 In 1776,
Meusnier~\cite{meu1} observed that the plane, the catenoid and the
helicoid all have zero mean curvature. In this same paper, he proved
that the condition on the mean curvature of the graph $G$ of a
function $u=u(x,y)$ over a domain in the plane to vanish
identically can be expressed by the following quasilinear, second
order, elliptic partial differential equation. This equation was
found in 1762 by
Lagrange\footnote{In reality, Lagrange arrived to a slightly
different formulation, and equation (\ref{eq:ecminsurf}) was derived
five years
 later by Borda~\cite{Bord1}.}~\cite{lag1}, who
 proved that it is equivalent to $G$ being a {\it minimal
 surface},
i.e., a  critical point of the area functional  with respect to
variations fixing the boundary of the graph:
\begin{equation}
\label{eq:ecminsurf} (1+u_x^2)u_{yy}-2u_xu_yu_{xy}+
(1+u^2_y)u_{xx}=0.
\end{equation}
It follows that  the plane, the catenoid and the helicoid are
examples of properly embedded, minimal planar domains in $\rth$ (a
{\it planar domain} is a connected surface which embeds in the plane, or
equivalently which is non-compact, connected and has genus zero).

Around 1860, Riemann discovered  (and posthumously published, Hattendorf and
Riemann~\cite{ri2,ri1}) other interesting examples of properly embedded,
periodic, minimal planar domains in $\rth$. 
These examples, called the {\it Riemann minimal examples}, appear in a
one-parameter family ${\cal
R}_t$, $t \in (0,\infty)$, and satisfy the property that, after
a rotation, each ${\cal R}_t$
intersects every horizontal plane in a circle or in a line.
The ${\cal R}_t$ have natural limits being a vertical catenoid as $t\to
0$ and a vertical helicoid as $t\to \infty$.
The main theorem of this manuscript states that these
beautiful surfaces of Riemann are unique in a particularly simple
way. This result, which was conjectured in~\cite{mpr3}, is motivated
by our earlier proof of it under the additional hypothesis that the
minimal planar domain is periodic~\cite{mpr1} and by partial results
in our previous papers~\cite{mpr8,mpr9,mpr3,mpr4}.

\begin{theorem}
\label{mainthm}
After a homothety and a rigid motion,
any connected, properly embedded, minimal planar domain in $\rth$
with an infinite number of ends is a Riemann minimal example.
\end{theorem}

Understanding properly embedded minimal surfaces in $\R^3$ is
the key for understanding the local structure of embedded minimal
surfaces in any Riemannian three-manifold. Inside this family of surfaces, the case of
genus zero is the most important, both because it is the starting point for
the general theory and because it gives the local picture for any finite genus minimal
surface in a three-manifold away from a finite collection of points where the
genus concentrates. Since we will focus on surfaces with genus zero, the only topological
information comes from the ends (i.e., the ways to go to infinity). For example,
the plane and the helicoid have only one end, the catenoid is topologically a cylinder and thus
has two ends, and the Riemann minimal examples are topologically cylinders with a periodic
set of punctures and thus have infinitely many ends.

Meeks and Rosenberg~\cite{mr8} proved that the only
simply connected, properly embedded, minimal planar domains in $\rth$ are
planes and helicoids. Earlier results of Collin~\cite{col1} and of
L\'opez and Ros~\cite{lor1} demonstrated  that the only non-simply connected,
properly embedded, minimal planar domains in $\rth$ with
a finite number of ends are catenoids.  Consequently,
Theorem~\ref{mainthm} gives the  following final classification
result.

\begin{theorem}[Classification Theorem for Minimal Planar Domains]
\label{classthm}
Up to scaling and rigid motion, any connected,
properly embedded, minimal planar domain in $\rth$ is a plane, a
helicoid, a catenoid or one of the Riemann minimal examples.  In
particular, for every such surface there exists a foliation of $\R^3$
by parallel
planes, each of which intersects the surface transversely in a
connected curve which is a circle or a line.
\end{theorem}

In a series of pioneering papers, Colding and Minicozzi~\cite{cm21,cm22,cm24,cm23}
gave a rather complete local description of properly embedded
minimal disks in a ball, showing that any such surface is either graphical (like the
plane) or contains a double-spiral staircase (like the helicoid).
Building on these results by Colding-Minicozzi,
Meeks and Rosenberg~\cite{mr8} characterized the plane and the helicoid as the only
simply connected, properly embedded, minimal surfaces in $\rth$.
The present paper relies on Colding-Minicozzi theory which will be used
to reduce the proofs of Theorems~\ref{mainthm} and \ref{classthm}
to Assertion~\ref{ass} below. This reduction is based on
our previously published papers~\cite{mpr3} and~\cite{mpr4} where we appealed to certain
results in the series~\cite{cm21,cm22,cm24,cm23} and to further
structure results by
Colding-Minicozzi for genus zero surfaces that appear in the fifth paper of their series~\cite{cm25}.
Besides these crucial applications of results by Colding-Minicozzi in our papers~\cite{mpr3} and~\cite{mpr4},
no other use will be made in the present paper of their results.
We will make clear in Section~\ref{sec3} which results of Colding-Minicozzi theory are needed and where in~\cite{mpr3} and~\cite{mpr4} they are applied.

The results of Collin~\cite{col1}, L\'opez-Ros~\cite{lor1} and
Meeks-Rosenberg~\cite{mr8} not only lead to the classification of all
properly embedded, minimal planar domains in $\rth$ of finite
topology, but these results and work of Bernstein and Breiner~\cite{bb2},
also characterize the asymptotic behavior of the
annular ends of any properly embedded minimal surface in $\rth$;
namely, each such end contains a representative which is asymptotic
to the end of a plane, a catenoid or a helicoid.
We will apply Theorem~\ref{mainthm} to obtain in Theorem~\ref{asympthm}
a similar characterization of the
asymptotic behavior of any properly embedded minimal surface in $\rth$
with finite genus and infinite topology. We remark that Hauswirth
and Pacard~\cite{hauP1} have  found many interesting
examples of such surfaces for any finite positive genus.

We next outline the organization of the paper and at the same time describe
the strategy for proving Theorems~\ref{mainthm} and
~\ref{classthm}. In
Section~\ref{sec2} we give a brief geometric and analytic description of the
Riemann minimal examples and present images of some of these surfaces.
In Section~\ref{sec3} we outline the basic
definitions and theory that reduce the proof of
Theorem~\ref{classthm} to Theorem~\ref{mainthm}, and the proof of
Theorem~\ref{mainthm} to Assertion~\ref{ass} below; see Theorems~\ref{thmcurvestim}
and~\ref{thm1finallimite} for these reductions, which are crucial in the proof
of our main results and which depend on our previously published papers~\cite{mpr3,mpr4};
in particular these reductions depend in an essential manner on Colding-Minicozzi theory.

Before stating Assertion~\ref{ass}, it is worth introducing some notation.
Suppose that $M\subset \R^3$ is a properly embedded minimal planar domain with
two limit ends (limit ends are defined in Section~\ref{sec3.3}),
such as one of the Riemann minimal examples. After a possible rotation of the surface,
any horizontal plane $P$ intersects $M$ in a simple closed curve or in a proper
Jordan arc $\g _P$ (see Theorem~\ref{thmcurvestim} 
for this property).
If we let $\eta $ denote the unitary upward pointing conormal to $M$
along $\g _P$, then the {\it flux vector}
 of $M$ is defined to be
\[
F_M=\int_{\g _P} \eta \, ds
\]
(here $ds$ stands for the length element), and $F_M$ is independent of the
choice of $P$. We proved in~\cite{mpr3} that, after
a rotation around the $x_3$-axis and a homothety, $F_M$ can be assumed to be
$(h,0,1)$ for some $h>0$. We remark that in our definition of the Riemann minimal
examples, $F_{{\cal R}_t} =(t,0,1).$

\begin{assertion}
\label{ass} Let ${\cal M}$ be the space of properly embedded, minimal planar domains
${M\subset\rth}$ with two limit ends, normalized so that every horizontal
plane intersects $M$ in a simple closed curve or a proper arc, and that the
flux vector is $F_M=(h,0,1)$ for some $h=h(M)>0$.
Then, ${\cal M}$ is the set of
Riemann minimal examples $\{{\cal R}_t\}_{t \in (0,\infty)}$.
\end{assertion}

The strategy to prove Assertion~\ref{ass} is by means of the classical
{\it Shiffman function,} which is a Jacobi function that adapts particularly well
to the problem under consideration. Since minimal surfaces can be viewed
as critical points for the area functional $A$, the nullity of the hessian of
$A$ at a minimal surface $M$ contains
valuable information about the geometry of $M$.
Normal variational fields for $M$ can be identified with functions
(we will always consider orientable
surfaces), and the second variation of area tells us that the functions in the
nullity of the hessian of $A$ coincide with the kernel of the Schr\"{o}dinger
operator $\Delta - 2K$ (called the {\it Jacobi operator}), where $\Delta
$ denotes the intrinsic Laplacian on $M$ and $K$ is the Gaussian curvature function on $M$.
Any function $v$ satisfying $\Delta v - 2Kv=0$
on $M$ is called a {\it Jacobi function,} and corresponds to an infinitesimal
deformation of $M$ by minimal surfaces.
A particularly useful Jacobi function in our
proof of Assertion~\ref{ass}
is the Shiffman function $S_M$, defined
for any surface $M\in {\cal M}$. In Section~\ref{sectShiffman}
we will study the Shiffman function, as well as basic properties of
Jacobi functions on a minimal surface which will be applied in our paper.
In the 1950s, Shiffman~\cite{sh1}
defined and applied $S_M$ to detect when a minimal surface
is foliated by circles and straight lines in parallel planes; this remarkable
property was known to characterize the surfaces $\mathcal{R}_t$ since
Riemann's times~\cite{ri2,ri1}. By \cite{ri2,sh1},
$S_M$ vanishes for a surface $M\in \mathcal{M}$ if and only if $M$ is a Riemann minimal example. Thus, a
possible approach to proving Assertion~\ref{ass} would be to
verify that $S_M$ vanishes for any $M\in {\cal M}$,
although we will not prove this fact directly. Instead, we will
demonstrate that $S_M$ is {\it linear} (i.e.,
it is the normal part of a parallel vector field in $\R^3$), a weaker property
which is enough to conclude that $M$ is a Riemann minimal example (Proposition~\ref{propos7.3}).

The desired linearity of $S_M$ for every $M\in \mathcal{M}$
will follow from the fact that $S_M$
can always be integrated in the following sense:
for an arbitrary $M\in {\cal M}$, there exists a one-parameter family
$\{ M_t\} _t\subset {\cal M}$ with $M_0=M$,
such that the normal part of the variational vector field for this variation,
when restricted to each $M_t$, is the Shiffman Jacobi function
$S_{M_{t}}$ multiplied by the unit normal vector field to $M_t$.
Moreover,
in our integration of $S_M$ by $\{ M_t\} _t$, the parameter $t$
can be extended to be a complex parameter, and $t\mapsto M_t$ can be viewed as
the real part of a complex valued {\it holomorphic} curve in a certain complex variety;
we will refer to this integrability property by saying that the Shiffman function can be
{\it holomorphically integrated;} see Definition~\ref{defholomorphicallyintegrated}
and Remark~\ref{remark6.13}.

Assume for the moment that $S_M$ can be
holomorphically integrated for any $M\in \mathcal{M}$, and we will explain why $S_M$ is linear.
The basic idea is to fix a flux vector $F=(h,0,1)$ and then extremize the spacing between the planar ends
among all examples in $\mathcal{M}_F=\{ M\in \mathcal{M}\ | \ F_M=F\} $ (this
requires a compactness result in $\mathcal{M}_F$ that uses the uniform geometric
estimates from Section~\ref{sec3}). Then one considers the
complex deformation $t\mapsto M_t$ around an extremizer $M_0\in \mathcal{M}_F$,
given by holomorphic integration of the Shiffman function $S_{M_0}$ of $M_0$,
and proves that the entire deformation is contained in $\mathcal{M}_F$ and that the
spacing between planar ends depends harmonically on the complex parameter $t$; this
harmonic dependence together with the maximum principle for harmonic functions
applies to give that the spacing remains constant along the deformation $t\mapsto M_t$,
which can be interpreted as the linearity of the Shiffman function of $M_0$.
From here we conclude that
any minimizer and any maximizer of the spacing between planar ends in $\mathcal{M}_F$
is a Riemann minimal example. As there is only one Riemann minimal
example with each flux, then
the maximizer and minimizer are the same, and thus every surface in $\mathcal{M}_F$ is
both a maximizer and minimizer, which implies that $\mathcal{M}_F$ consists of a single surface
that is a Riemann minimal example. The purpose of Section~\ref{per} is to give the
details of the arguments in this paragraph.

To finish our overview of the proof of Assertion~\ref{ass},
it remains to briefly explain how the Shiffman function $S_M$ can be holomorphically
integrated for any $M\in {\cal M}$, which will be the main task of Section~\ref{sec5}.
The approach is through the Korteweg-de Vries (KdV) equation
and its hierarchy. A change of variables transforms the holomorphic integration of
the Shiffman function into solving a Cauchy problem for a meromorphic KdV equation on the cylinder.
The key step for the solvability of this meromorphic KdV Cauchy
problem amounts to proving that the initial data is an {\it algebro-geometric potential} for KdV,
which is a finiteness condition in the hierarchy associated to the KdV equation that
will be established  in Corollary~\ref{sag}. This finiteness condition
depends crucially on the fact that the space ${\cal J}_{\infty }(M)$
of bounded Jacobi functions on any surface
$M\in \mathcal{M}$ is finite dimensional. Finite dimensionality of ${\cal J}_{\infty }(M)$
could be deduced from a paper by Colding, de Lellis and Minicozzi~\cite{cm39},
although we will include a self-contained  proof in Appendix~1.

In Section~\ref{3dim} we will
prove that all functions in ${\cal J}_{\infty }({\cal R})$ are linear
for any Riemann minimal example ${\cal R}$. This result could be seen as the
linearization of our main classification theorem, although does not directly follow
from the uniqueness of the Riemann minimal examples as there might be a bounded Jacobi function
on ${\cal R}$ that does not integrate to an actual variation.
Finally, in Section~\ref{asymp}, we will apply this characterization
of ${\cal J}_{\infty }({\cal R})$ and Theorem~\ref{mainthm} to prove Theorem~\ref{asympthm},
which describes the asymptotic behavior of the limit ends of properly embedded minimal surfaces
in $\rth$ with finite genus.

The authors would like to thank the referee for his valuable comments and suggestions
to improve the exposition.

\section {Analytic definition of the Riemann minimal examples.}
\label{sec2}
In \cite{ri1}, Riemann classified the compact minimal annuli in $\rth$
which are foliated by circles in some family of horizontal planes. He
proved that besides the catenoid and after a
homothety and a translation of $\rth$, each such annulus is contained
 in a properly embedded, minimal planar domain ${\cal R}(t)$ for some
$t\in (0,\infty)$, described as follows: Consider the rectangular torus
$\Te_t=\C / \Lambda_t$, where $\Lambda_t=\{tm + in \mid m, n \in \Z
\}$ and $i=\sqrt{-1}$, and let ${\cal P}_t$ denote the Weierstrass
${\cal P}$-function on $\Te_t$. Let $\sigma_t\colon \C / \langle i\rangle  \to \Te_t$
denote the related $\Z$-cover of $\Te_t$, where $\langle i\rangle =i\Z $.
Let $g_t$ denote the meromorphic function
$a_t{\cal P}_t\circ \sigma_t \colon  \C / \langle i\rangle  \to
\C\cup\{\infty\}$, where $a_t>0$ is chosen so that the branch values
of $g_t$ are $0,\infty $ and one pair of antipodal points on the real axis;
a simple calculation shows that $a_t=(\sqrt{-{\cal P}_t(i/2){\cal P}_t(t/2)})^{-1}$.
Let $dz$ be the holomorphic differential on
$\C / \langle i\rangle$ coming from the coordinate
$z=x+iy$ in $\C$.

In order to complete our description of ${\cal R}(t)$, it is
convenient to use the fact that $\C / \langle i\rangle $ is
isometric to the cylinder $\esf^1\times\R\subset \R^2\times\R$,
where $\esf^1 =i\R/\langle i\rangle$ is considered to be the circle
of circumference one centered at the origin in $\R^2$. Let $Z_t$ be
the set of zeroes of $g_t$, $P_t$ be the set of poles of $g_t$ and
$E_t=Z_t\cup { P}_t$. With this meromorphic data and with
$z_0=0+\frac{1}{2}i\in \C /\langle i\rangle $, the minimal planar
domain ${\cal R}(t)$ is defined analytically to be the image of the
conformal harmonic map $X_t\colon (\esf^1\times\R)-E_t\to \rth$
defined at $z=x+iy\in (\C /\langle i\rangle )-E_t$  by the
Weierstrass formula given in equation~(\ref{eq:Psi}) of Section~\ref{sectShiffman}:
\[
X_t(z)= \Re \int_{z_0}^{z}\left( \frac{1}{2}
(\frac{1}{g_t}-g_t),\frac{i}{2}(\frac{1}{g_t}+g_t),1\right) dz,
\]
where $\Re (w)$ is the real part of a complex vector
$w\in \C ^3$.

When we view ${\cal R}(t)$ as being parameterized by the punctured flat cylinder
$(\esf^1\times\R)-E_t$, then the level set circle $\esf^1 \times \{s\}$ at a height
$s$ different from $nt$ or $(n+\frac{1}{2})t$ for $n\in\Z$, has  image curve in
$\rth$ which is a circle in the horizontal plane at height $s$. Note that
$L_n=X_t[(\esf^1\times\{nt\} )-\{(0,nt)\}]$ is a line parallel to the $x_2$-axis,
 placed at height $nt$ (here we identify $\esf^1$ with $i\R /\langle i\rangle $). Similarly
$L_n^{\frac{1}{2}}=X_t[(\esf^1\times\{(n+\frac{1}{2})t
\})-\{(\frac{i}{2},(n+\frac{1}{2})t)\}]$ is a line parallel to the
$x_2$-axis and placed at height $(n+\frac{1}{2})t$. The isometry
group of ${\cal R}(t)$ is generated by the reflection in the
$(x_1,x_3)$-plane and the rotations of angle $\pi$ around the lines
$L_n$, $L'_n$, where $L'_n$ is the line parallel to $L_n$ which is
equidistant to $L_n$ and $L_n^{\frac{1}{2}}$ ($L_n'$ intersects
${\cal R}(t)$ orthogonally  at two points), $n\in\Z$. In
particular, the surface ${\cal R}(t)$ is periodic under the orientation preserving
translation $v_t$  given by the composition of the rotation by angle $\pi$
around $L_0$ with rotation by angle $\pi$ around $L_0^{\frac{1}{2}}$.
Thus, $v_t$ lies in the $(x_1,x_3)$-plane and has vertical component
$t$, see Figure~\ref{figure1}.
\begin{figure}
\begin{center}
 \includegraphics[width=15.1cm]{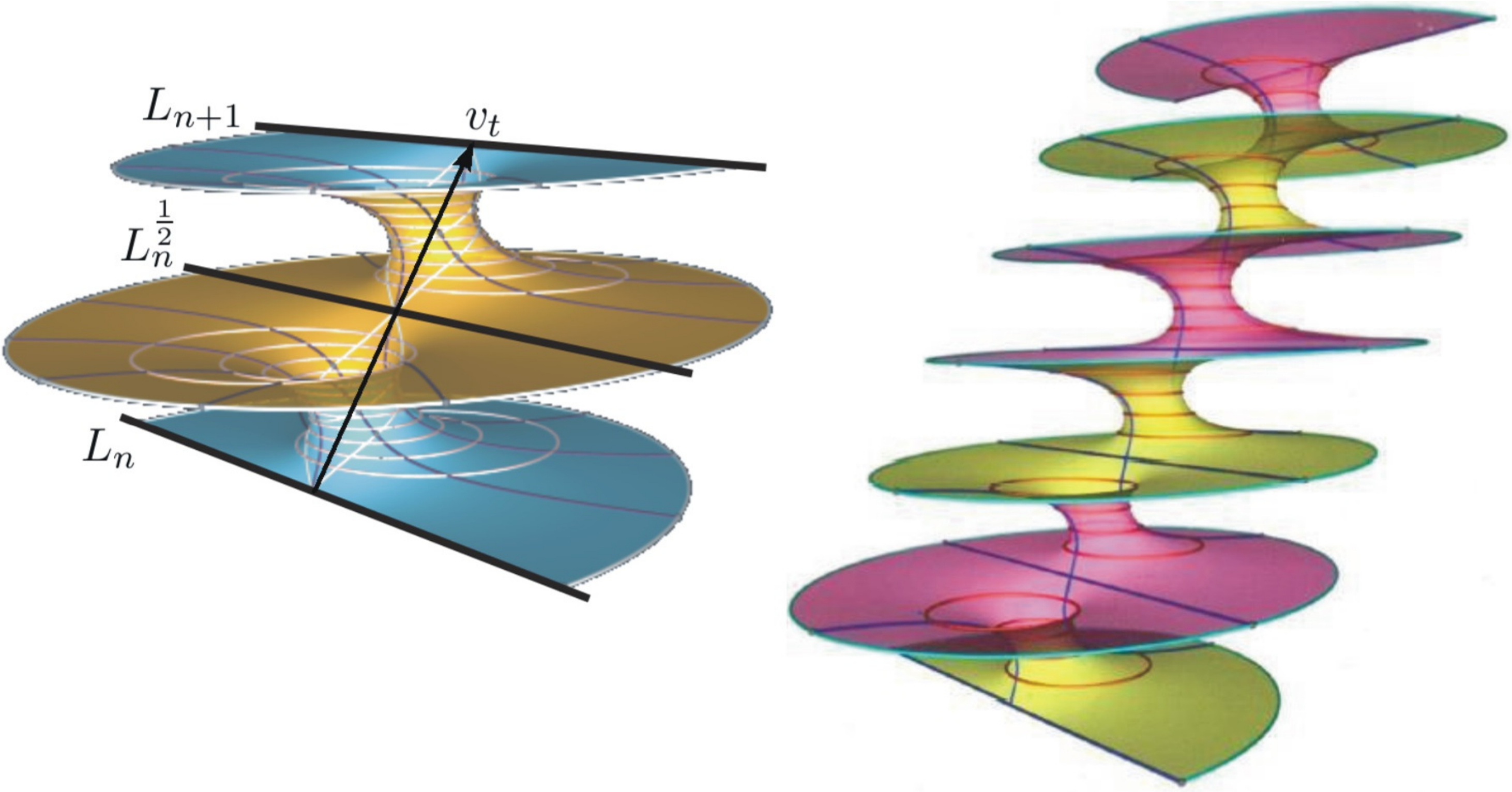}
  \caption{Two Riemann minimal examples (courtesy of M. Weber).}
\label{figure1}
\end{center}
\end{figure}

With respect to the above parametrization $X_t\colon (\C /\langle i\rangle ) -
E_t\to {\cal R}(t)\subset\rth$, the points in $E_t$ correspond to those ends
of ${\cal R}(t)$ which can be represented by annuli asymptotic to horizontal planes
 at heights in $H_t=\{nt,(n+\frac{1}{2})t \mid n\in\Z\}$.
Since $\C /\langle i\rangle $ is naturally conformally $\C-\{0\}\subset
\C\cup\{\infty\}=\esf^2$, we see that ${\cal R}(t)$ is conformally diffeomorphic to
$\esf^2 -{\cal E}({\cal R}(t))$, where ${\cal E}({\cal R}(t))$ is the union of the set of
planar  ends $E_t$ of ${\cal R}(t)$ with  $e_{-\infty}$ (resp. $e_{\infty}$), corresponding to
the bottom (resp. top) end of the cylinder $\C /\langle i\rangle =\esf^1\times\R$ viewed as being the south
(resp. north) pole of $\esf^2$. The set of points ${\cal E}({\cal R}(t)) \subset \esf^2$ with the subspace
topology can be naturally identified with the space of ends of the surface ${\cal
R}(t)$. By the topological classification of non-compact genus-zero surfaces,
${\cal R}(t)$ is the unique (up to homeomorphism) planar domain with two limit
ends. Finally, we remark that the holomorphic function
$g_t|_{(\Csmall /\langle i\rangle )-E_t}$ can be identified with the stereographic
projection of the Gauss map of ${\cal R}(t)$ when we view the Gauss map as being
defined on the parameter space $(\C /\langle i\rangle )- E_t$.

We now are in position to define a normalization of the Riemann minimal
examples that we will use in the sequel. For any $s$, let $\eta$ be the upward pointing,
unitary conormal vector along the boundary curve $\g_s$ of ${\cal R}(t)\cap \{x_3\leq s\}$.
The {\it flux} of ${\cal R}(t)$ is
\[
F_{{\cal R}(t)}=\int_{\g_s}\eta\, ds
\]
and has the form $(h(t),0,1)$ for some positive $h(t)$. It turns out
that $h(t)$ determines the Riemann minimal example ${\cal R}(t)$ and
moreover, $h(t)\rightarrow \infty $ as $t\rightarrow 0$ and $h(t)\to
0$ as $t\to \infty$. Define ${\cal R}_t={\cal R}(h^{-1}(t))$ for
each $t>0$; thus the flux of ${\cal R}_t$ is $(t,0,1)$. Then we
obtain the normalization $\{ {\cal R}_t\} _{t>0}$ of the family of
Riemann minimal examples to which we will we refer throughout this
paper. With this notation, the limit of suitable translations of the
${\cal R}_t$ as $t\to 0$ is a vertical catenoid, and the limit of suitable
translations and homotheties of the
${\cal R}_t$  as $t\to \infty $ is a vertical helicoid, see
Figure~\ref{figure2}.
\begin{figure}
\begin{center}
 \includegraphics[width=15cm]{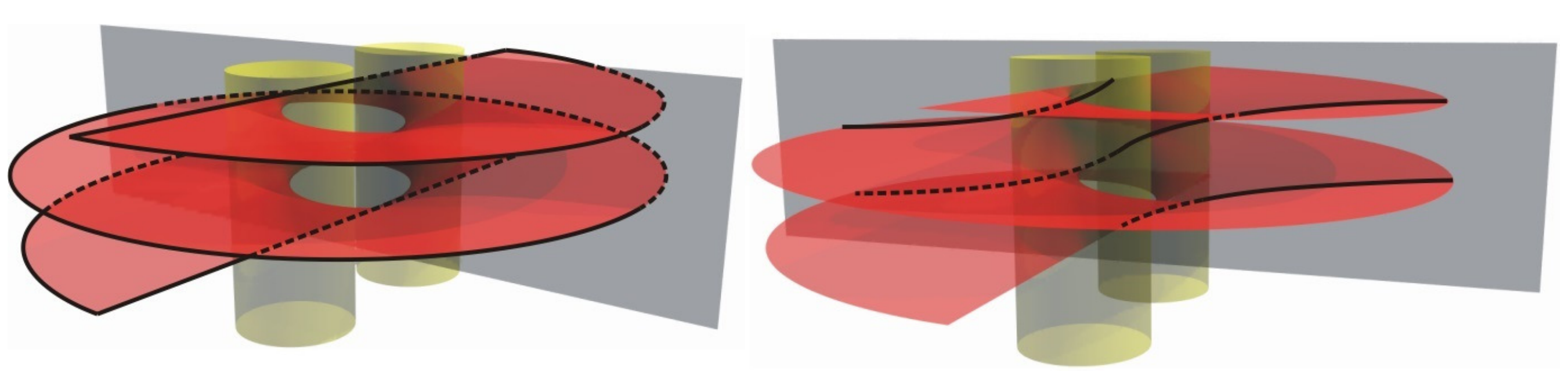}
  \caption{Two views of a
Riemann minimal example close to the helicoidal limit. Two vertical
helicoids are forming at opposite sides of the vertical plane of
symmetry (for the reader's convenience, we have also represented
vertical cylinders containing significant parts of the forming
helicoids).}
\label{figure2}
\end{center}
\end{figure}

\section{Reduction of the proof of Theorem~\ref{classthm}
to the case of two limit ends.}
\label{sec3}
Before proceeding with a discussion of the theoretical
results that reduce the proof of Theorem~\ref{classthm} to the case
of two-limit-end examples of genus zero, we make a few comments
that can suggest to the reader a visual idea of what is going on.
The most natural motivation for understanding this theorem and other
results presented in this section is to try to answer the following general
question: {\it What are the possible shapes of a complete embedded
surface $M\subset \R^3$ which satisfies a variational principle and has
a given topology?} In our case, the variational equation expresses the critical points
of the area functional. We will describe the situation according to the topology
of $M$.

\subsection{Properly embedded minimal surfaces with finite topology and
one end.}
If the requested topology for $M$ is the simplest
one of a disk, then the classification theorem of Meeks and Rosenberg
\cite{mr8} states that the possible shapes for complete examples are
the trivial one given by a plane and (after a rotation) an infinite
double spiral staircase, which is a visual description of a vertical
helicoid. A more precise description of the double spiral staircase
nature of a vertical helicoid is that this surface is the union of
two infinitely sheeted multigraphs, which are glued along a vertical
axis. Crucial in the proof of this classification result are the
results of Colding and Minicozzi~\cite{cm21,cm22,cm24,cm23} which describe
both the local structure of compact, embedded minimal disks
as essentially being modeled by either graphs or pairs of finitely sheeted multigraphs
glued along an ``axis'', and global properties of limits of these shapes.

More generally, if we allow our complete minimal surface $M\subset\rth$ to be
topologically a disk with a finite positive number of handles, then
it turns  out that $M$ is conformally a closed Riemann surface
$\overline{M}$ of positive genus punctured in a single point, see Bernstein
and Breiner~\cite{bb2} and also see Meeks and P\'erez~\cite{mpe3},
where they prove that $M$ is asymptotic to a helicoid and that
it can be defined analytically in terms of meromorphic data on
$\overline{M}$. Using different approaches, Hoffman,
Weber and Wolf~\cite{hweb1} and Hoffman and White~\cite{hofw1} proved the existence of
such a genus-one helicoid; also see Hoffman, Traizet and White~\cite{htw1}
where they construct properly embedded minimal surfaces
with arbitrary positive genus $g\in \N$ and one end.
Meeks and Rosenberg \cite{mr8} have conjectured
that there exists a unique genus $g$ helicoid for each positive
finite $g$.

\subsection{Properly embedded minimal surfaces with finite topology and
more than one end.}
 We now describe the special geometry of any properly
embedded minimal surface $M\subset\rth$ which  has finite topology
and more than one end. In this situation we find many beautiful
examples and even large dimensional families of examples, and so it
is not possible to obtain a general classification result for all of
these surfaces. However, the asymptotic behavior of the ends of $M$
is well-understood by Collin's Theorem~\cite{col1}, which states that
each end of $M$ is asymptotic to the end of a plane or catenoid.
In particular, we find that $M$ is conformally a compact Riemann
surface $\overline{M}$ punctured in a finite number of points and,
by a theorem of Osserman \cite{os3}, $M$ can be  defined
analytically in terms of meromorphic data on $\overline{M}$. For
example, one sees by a simple application of Picard's theorem that
the stereographic projection of the Gauss map $g\colon
M\rightarrow\C\cup\{\infty\}$ extends to a meromorphic function
$G\colon \overline{M}\rightarrow\C\cup\{\infty\}$.
The Gauss-Bonnet formula then implies that the degree of $G$ is equal
to the absolute total curvature of $M$ divided by $4\pi$.

A fundamental classification theorem of L\'opez and Ros \cite{lor1}
states that the plane and the catenoid are the only complete,
embedded, minimal planar domains in $\rth$ with finite total
curvature. Thus, Collin's theorem and the L\'opez-Ros theorem
together imply that the catenoid is the only connected, properly
embedded, minimal planar domain of finite topology in $\rth$ with
more than one end. Summarizing, {\it the plane, the helicoid and the
catenoid are the only properly embedded, minimal planar domains in
$\rth$ with finite topology.}

In order to better understand the asymptotic behavior of general finite topology examples
with more than one end (i.e., finite genus not necessarily zero),
it is helpful to consider the following question:
{\it  What is the visual picture for a
connected, properly embedded minimal surface $M\subset \rth$ with
finite topology and at least two ends?} By Collin's theorem, each end of
$M$ is asymptotic to the end of a catenoid or to a plane. It then
follows from the embeddedness of $M$ that after a fixed rotation of
$M$ and for some large $R_M>0$, $M\cap\{ (x_1,x_2,x_3)\ | \ x^2_1+x^2_2\geq R_M^2\}$
consists of a finite number $E_1,E_2,\ldots ,E_n$ of graphs over the annulus
$A(R_M)=(\R^2 \times \{ 0\} )-\{ (x_1,x_2,x_3)\ | \ x^2_1+x^2_2 < R_M^2\}$. These graphs have
logarithmic growths $\l_1 \leq \l_2
\leq\ldots \leq \l_n$, which are linearly ordered by the relative heights of the graphs
over $A(R_M)$ and the collection $\{E_1,E_2,\ldots ,E_n\}$ represents the
(annular) ends of $M$. Note that  when $\lambda_i=0$, then $E_i$ is
asymptotic to a horizontal plane. The Half-space Theorem by Hoffman and
Meeks~\cite{hm10}
implies that $M$ cannot be contained in a half-space of $\rth$ and so,
$\lambda_1<0$ and $0<\lambda_n$. After this rotation to have graphical
ends on the exterior of a disk on the $(x_1,x_2)$-plane,
 $M$ is said to have {\it horizontal limit tangent
plane at infinity.}

As in the case of one-ended minimal surfaces of finite topology,
there is a precise conjecture on the topological types allowed in  the class of properly
embedded minimal surfaces with finite topology and more than one end.
In 1982, Hoffman and Meeks conjectured that a necessary and sufficient
condition for a surface with finite genus $g$ and a finite number
$k>2$ of ends to admit a proper minimal embedding into $\R^3$ is that
$g+2\geq k$. The case $k=1$ reduces to the plane and Schoen~\cite{sc1} proved that if $k=2$ then the surface is a catenoid (hence $g=0$). L\'opez-Ros~\cite{lor1} proved that  $g=0$ implies that the surface is a catenoid or a plane, hence the last inequality also holds in this case. In the remaining cases, this conjecture  is supported by
existence theorems of Traizet~\cite{tra1} and of Weber and Wolf~\cite{ww1}.
Along these lines, the authors proved that for each fixed genus $g$, there is
an upper bound $k(g)$ for the number of ends of a properly embedded minimal surface
with genus $g$ and finitely many ends~\cite{mpr8}.

\subsection{Properly embedded minimal surfaces with finite genus and
an infinite number of ends.}
\label{sec3.3}
Any properly embedded minimal surface
$M\subset \rth$ with more than one end has an associated plane
passing through the origin which is called {\it the limit tangent
plane at infinity} of~$M$, defined as follows.
Firstly, it can be shown that $\rth - M$ contains the end $E$ of a plane
or catenoid. Such an end $E$ has a limiting normal vector
$v_E$ at infinity which turns out to not depend on the choice of $E$
 in $\rth - M$. The plane passing through
the origin and orthogonal to $v_E$ is the limit tangent plane at
infinity to $M$ (see \cite{chm3} for further details). We will
generally assume that the limit tangent plane at infinity to $M$ is
horizontal, or equivalently, it is the $(x_1,x_2)$-plane. A
fundamental aid in discussing the asymptotic geometry of $M$ is the
Ordering Theorem of Frohman and Meeks~\cite{fme2}, which
states that the space of ends ${\cal E}(M)$ of $M$ has a natural
linear ordering by their relative heights over the
$(x_1,x_2)$-plane, similar to the way in which the ends of a finite topology
minimal surface with more than on end can be linearly ordered.

With this linear ordering on ${\cal E}(M)$ in mind and using the
fact that ${\cal E}(M)$ has a natural topology induced by an order
preserving embedding as a
compact, totally disconnected subspace of the unit interval $[0,1]$
(see Section 2.7 in \cite{mpe1}), we find that there
exist unique elements $e_T,e_B\in {\cal E}(M)$ which are maximal and minimal
elements in the linear ordering on ${\cal E}(M)$, respectively. $e_T$ is called the
{\it top end} and $e_B$ the {\it bottom end} of $M$.
The other ends in ${\cal E}(M)-\{e_B,e_T\}$ are called the
{\it middle ends} of $M$. By a result of Collin, Kusner, Meeks
and Rosenberg~\cite{ckmr1}, the only possible limit ends of
$M$ (limit points of ${\cal E}(M)$ in its natural topology) are
$e_B$ or $e_T$.

The above discussion implies that the classification of the properly
embedded minimal planar domains in $\R^3$ reduces to the classification of examples
with two limit ends and to ruling out the case of one limit end.
We will start by describing the geometry of any surface in the two limit end
case.
Note that Theorem~\ref{thmcurvestim} below uses the notation in Assertion~\ref{ass}.
\begin{theorem}
\label{thmcurvestim}
Given any $M\in {\cal M}$, we have:
\begin{enumerate}
\item $M$ can be conformally parameterized by the cylinder
$\C /\langle i\rangle $ punctured at
an infinite discrete set of points $\{ p_j,q_j\} _{j\in \Zsmall }$.

\item The stereographically projected Gauss map of $M$, considered to be a
meromorphic function $g$ on $\C /\langle i\rangle $ after attaching the planar ends of $M$,
has order-two zeros at the points $p_j$ and order-two poles at the $q_j$.

\item The height differential of $M$ is $dh=dz$ and so, its height function is $x_3(z)=\Re (z)$.
In particular, the middle ends $p_j,q_j$ of $M$ are planar, and they are naturally ordered by heights
by $\Re (p_j)<\Re (q_j)<\Re (p_{j+1})$ for all $j\in \Z$, with $\Re (p_j)\to \infty $
{\rm (}resp. $\Re (p_j) \to -\infty ${\rm )} when $j\to \infty $ {\rm (}resp. $j\to -\infty ${\rm ).}

\item Every horizontal plane intersects $M$
in a simple closed curve when its height is not in $H=\{\Re (p_j),
\Re (q_j)\ | \ j\in \Z \} $ and in a single properly embedded arc
when its height is in $H$; in particular, the principal divisor of
$g$ is $(g)=\prod _{j\in \Zsmall }p_j^2q_j^{-2}$.

\item $M$ has bounded Gaussian curvature, and this bound only depends
on an upper bound of $h$ {\rm (}recall that
the flux of $M$ along a compact horizontal section is $F_M=(h,0,1)$ with $h>0${\rm ).}

\item The vertical spacings between consecutive ends are bounded from above and below
by positive constants that only depend on $h$. Also, $M$ admits an embedded
regular neighborhood of fixed radius $r=r(h)>0$.

\item For every divergent sequence $\{ z_k\} _k\subset \C /\langle i\rangle $,
there exists a subsequence of the meromorphic functions $g_k(z)=g(z+z_k)$ which converges uniformly on compact
subsets of $\C /\langle i\rangle $ to a non-constant meromorphic function
$g_{\infty }\colon \C /\langle i\rangle \to \C \cup \{ \infty \} $. In fact, $g_{\infty }$ corresponds to the
Gauss map of a surface $M_{\infty }\in {\cal M}$, which is the limit
 of a related subsequence of translations of $M$ by vectors whose
$x_3$-components are $\Re (z_k)$.
\end{enumerate}
\end{theorem}
\begin{proof}
Most of the arguments in this proof can be found in our previous paper~\cite{mpr3}; for the sake of completeness and
also in order to clarify the dependence of the results in~\cite{mpr3} on Colding-Minicozzi theory,
we will include some details about this proof. By Theorem~1.1 in~\cite{ckmr1},
the middle ends of $M$ are not limit ends. As $M$ has genus zero, then these middle ends can be represented by annuli.
By Collin's theorem~\cite{col1}, every annular end of $M$ is asymptotic to the end of a plane or catenoid.
By Theorem~3.5 in~\cite{ckmr1}, there exists a sequence of horizontal planes $\{ P_i\} _{i\in \N }$ with
increasing heights such that $M$ intersects each plane $P_i$ transversely in a compact set,
every middle end of $M$ has an end representative which is the closure of the intersection of $M$ with
the slab bounded by $P_i\cup P_{i+1}$, and every such slab contains exactly one of these middle end representatives.
By the Halfspace Theorem~\cite{hm10},
the restriction of the harmonic third coordinate function $x_3$ to
the portion $M(+)$ of $M$ above $P_0$ is not bounded from above and extends smoothly across the middle ends.
By Theorem~3.1 in~\cite{ckmr1}, $M(+)$ has a parabolic
conformal structure. After compactification of $M(+)$ by adding its middle ends and their limit point
$p_{\infty }$ corresponding to the top end in $M(+)$, we obtain a conformal parameterization
of this compactification defined on the unit disk $\D =\{ |z|\leq 1\} $, so that $p_{\infty }=0$, the middle ends
in $M(+)$ correspond to a sequence of points $p_i\in \D -\{ 0\} $ converging to zero and $x_3|_{M(+)}(z)=-\l \ln|z|+c$
for some $\l ,c\in \R $, $\l>0$.
Also note that different planar ends cannot have the same height above $P_0$ (since they lie in different slabs
bounded by the planes $P_i$). In particular, $M(+)$ intersects every plane $P'$ above $P_0$ in a simple closed curve
if the height of $P'$ does not correspond to the height of any middle end, while $P'$ intersects $M(+)$ is a
Jordan arc when the height of $P'$ equals the height of a middle end. This implies that the zeros and poles
of the stereographically projected Gauss map $g$ of $M$ at the middle ends of $M(+)$ have order two.
Since the behavior of $M(-)=M-[M(+)\cup P_0]$ can be described analogously, then items 1, 2, 3 and 4 of
the theorem are proved.

The fact that the Gaussian curvature $K_M$ of $M$ is bounded with the bound depending
 only on an upper bound of the horizontal part of the flux vector $F_M$ (item 5 of the theorem),
 was proven in Theorem~5
of~\cite{mpr3} in the more general case of a sequence $\{ M(i)\} _i\subset \mathcal{M}$ such that
$\{ h(i)\} _i$ is bounded, where $F_{M(i)}=(h(i),0,1)$ for each $i\in \N$. In this setting, the conclusion of
Theorem~5 of~\cite{mpr3} is that the sequence of Gaussian curvature functions
$\{ K_{M(i)}\} _i$ of the $M(i)$ is
uniformly bounded. Since we will use later this stronger version of the curvature estimates
(namely, in Proposition~\ref{propos7.6}), we now
sketch its proof. The argument is by contradiction, so assume that $\{ K_{M(i)}\} _i$ is not
uniformly bounded.

The first step consists of finding points $p(i)\in M(i)$ and positive numbers $\l (i)\to \infty $
such that after passing to a subsequence, the surfaces $M'(i)=\l (i)(M(i)-p(i))$
converge uniformly on compact subsets of $\R^3$ with
multiplicity 1 to a vertical helicoid $H$ passing through the origin $\vec{0}$, with $|K_H|\leq 1$ and
$|K_H|(\vec{0})=1$. This is done in Lemma~5 of~\cite{mpr3}, whose proof uses a standard blow-up argument
on the scale of curvature; this blow-up process creates a subsequential limit of the $M'(i)$, which is simply
 connected by a flux argument (non-zero fluxes of the limit surface must come from non-zero arbitrarily small
fluxes on the $M(i)$, which is impossible). This limit of the $M'(i)$ is a helicoid $H$
by the classification by Meeks and Rosenberg of the simply connected,
properly embedded minimal surfaces~\cite{mr8} (we remind the reader that this classification depends
on Colding-Minicozzi results contained in~\cite{cm21,cm22,cm24,cm23}). The axis of
$H$ is vertical as the Gauss maps of the $M(i)$ omit the vertical directions by the already proven item 3 of the
theorem, and the normalization of
$K_H$ follows directly from construction. This finishes the first step.

The second step consists of
renormalizing the surfaces $M'(i)$ by rescaling and rotation around the $x_3$-axis, so that
\begin{description}
  \item[(P1)] $\vec{0}\in M'(i)$.
 \item[(P2)] The horizontal section of $M'(i)$ at height zero is a simple closed curve.
  \item[(P3)] $\{ M'(i)\} _i$ converges on compact subsets of $\R^3$ to a vertical helicoid with
  axis passing through $\vec{0}$.
\end{description}
Property {\bf (P3)} above insures that one can find an open arc $\a '(i)\subset M'(i)\cap \{ x_3=0\} $
centered at $\vec{0}$ so that the Gauss map of $M'(i)$ takes values at the end points of $\a '(i)$
in different hemispheres determined by the horizontal equator. By continuity, this allows us
to find a point $q'(i)\in [M'(i)\cap \{ x_3=0\} ]-\a '(i)$ closest to the origin where the
Gauss map of $M'(i)$ is horizontal, and then renormalize the $M'(i)$ by rescaling and rotation around
the $x_3$-axis to define new surfaces $\wt{M}(i)$ so that $q'(i)$ is independent of $i$
and has the form $\wt{q}=(0,6\tau ,0)\in \wt{M}(i)$, where
$\tau >\tau _0$ is fixed but arbitrary, and $\tau _0>1$ is defined by the following
auxiliary property (this is Lemma~4 in~\cite{mpr3},
whose proof only uses the curvature estimates for stable minimal surfaces by Schoen~\cite{sc3}):

\begin{description}
\item[(P4)] There exists $\tau _0>1$ such that given a properly embedded, non-compact,
orientable, stable minimal surface $\Delta $ contained in a horizontal slab of width not greater
than 1, and such that the boundary $\partial \Delta $ lies inside a vertical cylinder of radius
$1$, the portion of $\Delta $ at distance greater than $\tau _0$ from the axis of the cylinder
consists of a finite number of graphs over the complement of a disk of radius $\tau _0$ in the
$(x_1,x_2)$-plane.
\end{description}

The third step consists of finding embedded
closed curves $\de (\tau ,i)\subset \wt{M}(i)$ so that the flux of $\wt{M}(i)$ along $\de (\tau ,i)$ decomposes as
\begin{equation}
\label{eq:teorKestim1}
\mbox{Flux}\left( \wt{M}(i),\de (\tau ,i)\right) =V(\tau ,i)+W(\tau ,i),
\end{equation}
where $V(\tau ,i),W(\tau ,i)\in \R^3$ are vectors such that $\lim
_{i\rightarrow \infty }V(\tau, i)=(12\tau ,0,0)$
and $\| W(r,i)\| $ is bounded by a constant independent of $i,\tau $.
Assuming (\ref{eq:teorKestim1}), the desired contradiction that will give item 5 of the theorem comes from the fact that
the angle between the flux vector $F_{M(i)}$ of the $M(i)$ and its horizontal projection
$h(i)$ is invariant under translations, homotheties and rotations around the $x_3$-axis. But
(\ref{eq:teorKestim1}) implies that the corresponding angles for the flux vectors of the
surfaces $\wt{M}(i)$ tend to zero as $i\rightarrow \infty $ and $\tau \rightarrow \infty $,
which contradicts that the $h(i)$ were assumed to be bounded. To finish this sketch of the proof
of item~5 of Theorem~\ref{thmcurvestim} we will give details on how to construct the connection loops $\de (\tau ,i)$.
Each of these connection loops consists of four consecutive arcs $\a _1(\tau ,i)$, $L(\tau ,i)$,
$\a _2(\tau ,i)$, $\wt{L}(\tau ,i)$ contained in $\wt{M}(i)$ with the following properties:

\begin{description}
\item[(P5)] $\a _1(\tau ,i)$ is a short curve 
close to the axis of the
highly sheeted vertical helicoid which is forming nearby $\vec{0}$ by property {\bf (P3)} above,
which goes down exactly one level in the double staircase structure occurring around
$\vec{0}$.
\item [(P6)] $L(\tau, i)$, $\wt{L}(\tau ,i)$ are approximations of a horizontal segment from $\vec{0}$ to $\wt{q}$,
whose extrema near $\vec{0}$ coincide with the extrema of $\a _1(\tau ,i)$,
    and such that $\wt{L}(\tau ,i)$ lies directly above $L(\tau ,i)$. These almost segments $L(\tau ,i)$,
    $\wt{L}(\tau ,i)$ are defined in pages 23, 24 of~\cite{mpr3} and their existence is insured by
    Lemma~8 in~\cite{mpr3}, which is a delicate application of Colding-Minicozzi theory; this lemma
    asserts that the sequence $\{ \wt{M}(i)\} _i$ is {\it uniformly simply connected in $\R^3$} (ULSC),
    which means that there exists $r>0$ such that every component of the intersection of $\wt{M}(i)$
    with any ball in $\R^3$ of radius $r$ is a disk. This property allows one to use Theorem~0.9
    in Colding and Minicozzi~\cite{cm25}  to conclude that
    after passing to a subsequence, the $\wt{M}(i)$ converges to the foliation $\mathcal{L}$ of $\R^3$
    by horizontal planes with singular set of convergence $S(\mathcal{L})=\G \cup \G'$ being the
    $x_3$-axis $\G $ and the vertical straight line $\G'$ passing through $\wt{q}$. Once
    this limit foliation result is established, the almost-straight-line, almost-horizontal
    segments $L(\tau, i)$, $\wt{L}(\tau ,i)$
    satisfying {\bf (P6)}, as well as the fourth arc $\a _2(\tau ,i)$ in $\de (\tau ,i)$
    satisfying the following property, are easy to construct; for similar constructions, see the Lamination Metric
    Theorem by Meeks (Theorem~2 in~\cite{me30}):
\end{description}
\begin{description}
\item[(P7)] $\a _2(\tau ,i)$ is an embedded arc connecting the end points of $L(\tau, i)$, $\wt{L}(\tau ,i)$ nearby $\wt{q}$,
with length bounded from above by a constant that does not depend either on $i$ or on $\tau $.
\end{description}

Assuming the connection loops $\de (\tau ,i)=\a _1(\tau, i)\cup L(\tau ,i)\cup \a _2(\tau ,i)\cup \wt{L}(\tau ,i)$
are constructed verifying {\bf (P5)}, {\bf (P6)}, {\bf (P7)}, then the decomposition in (\ref{eq:teorKestim1})
reduces to defining $V(\tau ,i)$ as the flux of $\wt{M}(i)$ along $L(\tau ,i)\cup \wt{L}(\tau ,i)$, and
$W(\tau ,i)$ as the flux of $\wt{M}(i)$ along $\a_1(\tau ,i)\cup \a _2(\tau ,i)$. In summary, to
conclude the proof of item~5 of Theorem~\ref{thmcurvestim}
one needs to check that the sequence $\{ \wt{M}(i)\} _i$ satisfies the hypotheses of
Theorem~0.9 in~\cite{cm25} (i.e., that $\{ \wt{M}(i)\} _i$ is
ULSC on every compact subset of $\R^3$; see equation~(0.1) in~\cite{cm25}
for this notion). This proof of this property
starts by demonstrating the following statement (Lemma~6 in~\cite{mpr3}):
\begin{description}
\item[(P8)] There exist $a(i)<0<b(i)$ such that
for every extrinsic ball $B$ of radius 1,
the intersection of $\wt{M}(i)$ with the portion of $B$
inside a horizontal slab $S(a(i),b(i))=\{ (x_1,x_2,x_3)\ | \ a(i)<x_3< b(i)\} $ is
simply connected.
\end{description}
It is worth explaining how the numbers $a(i),b(i)$ in {\bf (P8)} are chosen,
to understand why property {\bf (P8)} holds. We denote by
$\overline{\B }(x_0,r)=\{ x\in \R^3\ | \ \| x-x_0\| \leq r\} $ the closed Euclidean ball
centered at $x_0\in \R^3$ with radius $r>0$. We choose $a(i)<0<b(i)$ so that
$\wt{M}(i)\cap \overline{\B}(\vec{0},1)$
contains an open arc $\be (i)$ passing through $\vec{0}$ connecting $\{ x_3=a(i)\} \cap \partial \B (\vec{0},1)$
to $\{ x_3=b(i)\} \cap \partial \B (\vec{0},1)$, and $\wt{M}(i)\cap \overline{\B}(\wt{q},1)$
contains another open arc passing through $\wt{q}$ connecting $\{ x_3=a(i)\} \cap \partial \B(\wt{q},1)$
to $\{ x_3=b(i)\} \cap \partial \B(\wt{q},1)$, and $S(a(i),b(i))$ is maximal with this property.
Statement {\bf (P8)} is proven by contradiction: the existence of a homotopically non-trivial
 curve $\g(i)\subset \wt{M}(i)\cap S(a(i),b(i))\cap B$ for some extrinsic ball $B$ of radius 1
and a standard area-minimization construction using $\wt{M}(i)$ as a barrier allows us to find
a properly embedded, non-compact, orientable stable minimal surface $\Delta (i)$ in
$S(a(i),b(i))-\wt{M}(i)$ with $\partial \Delta (i)=\g (i)$. As the distance from
$\vec{0}$ to $\wt{q}$ is $6\tau >6\tau _0$ (this $\tau _0$ was defined in Property~{\bf (P4)}),
then the distance from $B$ to at least one of the vertical cylinders $C(\vec{0},1)$, $C(\wt{q},1)$ of radius
1 with axes passing through $\vec{0},\wt{q}$ respectively, is larger than $\tau _0$
(we can assume dist$(B,C(\vec{0},1))>\tau _0$ as the other case can be solved similarly).
By Property~{\bf (P4)}, $\Delta (i)\cap C(\vec{0},1)$ is a union of horizontal graphs, all contained
in $S(a(i),b(i))$. These graphs cross the arc $\be(i)$, which is a contradiction that proves
Property {\bf (P8)}.

With Property~{\bf (P8)} in hand, the next step shows that there exists some $\ve >0$ independent of $i$ so that
$S(-\ve ,\ve )=\{ (x_1,x_2,x_3)\ : \ |x_3|<\ve \} $ is contained in $S(a(i),b(i))$ (Lemma~7 in~\cite{mpr3});
this is essentially a consequence of the 1-sided curvature estimates in~\cite{cm23}. From here we conclude:
\begin{description}
\item[(P9)] The origin is a {\it singular point} for the sequence $\{ \wt{M}(i)\} _i$ (i.e., the Gaussian curvatures of
the $\wt{M}(i)$ near $\vec{0}$ blow-up) and that there exist constants $r,\de > 0$ so that for every extrinsic
ball $B$ of radius $r$ whose center is closer than $\de $ from $\vec{0}$, the intersection of $\wt{M}(i)$
with $B$ consists of compact disks with boundary in $\partial B$. With the notation in Colding-Minicozzi~\cite{cm25}, this can be abbreviated by saying that $\vec{0}\in \mathcal{S}_{ulsc}$.
\end{description}
In this situation one can apply the Non-Mixing Theorem~0.4 in~\cite{cm25}
to conclude that every singular point for the sequence $\{ \wt{M}(i)\} _i$ is in $\mathcal{S}_{ulsc}$, which
in turn implies by Theorem~0.9 in~\cite{cm25} the desired convergence of the $\wt{M}(i)$ to the foliation
$\mathcal{L}$ of $\R^3$ by horizontal planes with singular set consisting of $\G \cup \G '$.

We remark that in order to apply Theorem~0.9 in~\cite{cm25}, one needs to check that some component
of the intersection of $\wt{M}(i)$ with a ball of fixed size centered at $\vec{0}$ is not a disk;
this property holds since otherwise, Theorem~0.1 in \cite{cm23} would lead to the convergence
of (a subsequence of) the $\wt{M}(i)$ to the same horizontal foliation $\mathcal{L}$, but with singular set
consisting solely of $\G $. This contradicts that the tangent plane of the $\wt{M}(i)$ at $\wt{q}$ is
vertical for every $i$.

We also remark that the above argument does not really need that $\{ \wt{M}(i)\} _i$ is ULSC in $\R^3$,
but only that $\{ \wt{M}(i)\} _i$ is ULSC {\it on every compact subset} of $\R^3$,
 as defined in equation~(0.1) in~\cite{cm25}; our argument in~\cite{mpr3}
 to prove that $\{ \wt{M}(i)\} _i$ is ULSC in $\R^3$
is different than the one presented here, as it does not use the Non-Mixing Theorem~0.4 in~\cite{cm25}
but instead, uses a blow-up argument on the scale of topology to produce a new limit object of
a blow-up sequence of the $\wt{M}(i)$ and then finds a contradiction in all possible such limits
(proof of Assertion~2 of~\cite{mpr3}). This finishes our sketch of the proof of
item~5 of Theorem~\ref{thmcurvestim}.

As for the proof of item~6 of the theorem, the fact that the Gaussian curvature function
$K_M$ of a surface $M\in \mathcal{M}$ is bounded implies that $M$ admits an embedded regular neighborhood of
radius $1/\sup \sqrt{|K_M|}$ (see Meeks and Rosenberg~\cite{mr1}). This clearly gives
that the vertical spacing between consecutive ends is bounded from below.
To see why the spacing is bounded from above,
one first checks that for every two consecutive ends of $M$ asymptotic to horizontal planes
$\Pi_n$,$\Pi_{n+1}$, there exists a point $p_n\in M\cap \{ x_3(\Pi _n)<x_3<x_3(\Pi _{n+1})\} $
such that the tangent plane to $M$ at $p_n$ is vertical. Next one shows that given
$\ve >0$ fixed and sufficiently small, there exists a point $q_n\in M$ at intrinsic distance
less than 2 from $p_n$ such that $|K_M(q_n)|>\ve $ (this is a flux argument, since the tangent plane
at $p_n$ is vertical and the vertical component of the flux of $M$ is normalized to be 1).
As $M$ has bounded Gaussian curvature and admits an embedded regular neighborhood of fixed
radius, then the translated surfaces $M-q_n$ converge (after passing to a subsequence) with multiplicity one
to a connected, non-flat, properly embedded minimal planar domain $M_{\infty }$, whose (non-constant)
Gauss map omits the vertical directions. If the spacing between consecutive middle ends of $M$
is unbounded, then one can produce such a limit surface $M_{\infty }$ with a top or bottom end which is
of catenoidal type with vertical limiting normal vector;
this implies that $M_{\infty }$ has vertical flux,
and in this situation one can
 use a variation of the L\'opez-Ros deformation argument~\cite{lor1}
on $M$ to find a contradiction. For details, see page 36 of~\cite{mpr3}.
Similar reasoning shows that the spacing between consecutive ends for a sequence of surfaces
$\{ M_n\} _n\subset \mathcal{M}$ can be bounded from above and below by positive constants
that only depend on upper and non-zero lower bounds of the horizontal component of the flux
vector of the $M_n$. Now item~6 of the theorem is proved.

Finally, we explain the proof of item~7 of the theorem.
Take a divergent sequence $\{ z_k\} _k\subset \C /\langle i\rangle $ and call $g_k(z)=g(z+z_k)$, where
$g$ is the stereographically projected extension of the Gauss map of a surface $M\in \mathcal{M}$.
Recall that the family of functions $\{ g_k\} _k$ is normal if and only if on every compact set $C$ of $\C /\langle
i\rangle $, the sequence of numbers $\{ S_k(C)\} _k$ is bounded from  above, where
\[
S_k(C)=\sup \left\{ \frac{|g_k'(z)|}{1+|g_k(z)|^2} \ | \ z\in C\right\} .
\]
As the height differential of $M$ is $dz$, then the spherical derivative $\frac{|g'(z)|}{1+|g(z)|^2}$ of $g$
is, up to a constant, the square root of the Gaussian curvature of $M$ at the point corresponding to $z$, which is bounded by item~5 of the theorem. Thus, there exists a meromorphic function $g_{\infty }\colon
\C /\langle i\rangle \to \C \cup \{ \infty \} $
so that after passing to a subsequence, the $g_k$ converge uniformly on compact subsets of
$\C /\langle i\rangle $ to $g_{\infty }$. Note that $g_{\infty }$ cannot be constant
since the $z_k$ is at bounded distance in $\C /\langle i\rangle$ from consecutive
ends of $M$ by item~6, where $g$ has zeros and poles. As $g_{\infty }$ is not constant, then
$g_{\infty }$ has only second order zeros and poles by Hurwitz's theorem. The fact that $g_{\infty }$ corresponds to the Gauss map of a surface $M_{\infty }\in \mathcal{M}$
is straightforward; in fact, $M_{\infty }$ is a limit of an appropriately chosen sequence of translations of
$M$. This completes our discussion of the proof of Theorem~\ref{thmcurvestim}.
\end{proof}

Coming back to our discussion of minimal planar domains in $\R^3$, we must rule out the case of one limit end.
This was done in Theorem~1 of~\cite{mpr4}, that we state below.
\begin{theorem}
\label{thm1finallimite}
If $M$ is a connected, properly embedded minimal
  surface  in $\R^3$ with finite genus, then one of the
  following possibilities holds:
\begin{enumerate}
\item $M$ is a plane;
\item $M$ has one end and is asymptotic to the end of a helicoid;
\item $M$ has a finite number of ends greater than one, has finite
 total curvature and each end of $M$ is asymptotic to a plane or to the end
 of a catenoid;
\item $M$ has two limit ends.
\end{enumerate}
Furthermore, $M$ has bounded Gaussian curvature and is conformally
diffeomorphic to a compact Riemann surface punctured in a countable
closed subset which has exactly two limit points if the subset is infinite.
\end{theorem}

The proof of Theorem~\ref{thm1finallimite} depends on the previous Theorem~\ref{thmcurvestim}, as well
as on the Limit Lamination Theorem~0.9 of Colding and Minicozzi~\cite{cm25}.

The discussion in this section completes the reduction of the proof of the main Theorem~\ref{classthm}
to that of Assertion~\ref{ass}.

\section{Jacobi functions on a minimal surface.}
\label{sectShiffman}
The first variation of area allows one to consider
a minimal surface $M\subset \R^3$ to be a critical point for the area
functional acting on compactly supported (normal) variations. The
second variation of area is governed by the {\it stability} or {\it
Jacobi operator} $L=\Delta +|\sigma |^2=\Delta -2K$ where $\Delta $
denotes the intrinsic Laplacian on $M$, $|\sigma |^2$ is the square
of the norm of its second fundamental form and $K$ is its Gaussian
curvature function. $L$ is a linear Schr\"{o}dinger operator whose
potential $|\sigma |^2=|\nabla N|^2$ is associated to the Gauss map
$N\colon M\to \esf^2$ of $M$. The holomorphicity of $N$ is crucial
in understanding the functions in the kernel ${\cal J}(M)$ of $L$
(so called {\it Jacobi functions}), which correspond to normal parts
of infinitesimal deformations of $M$ through minimal surfaces. In
terms of the stereographically projected Gauss map $g$ of $M$, the
Jacobi equation $Lv=0$ can be written as
\begin{equation}
\label{eq:Jacobi}
v_{z\overline{z}}+2\frac{|g'|^2}{(1+|g|^2)^2}v=0,
\end{equation}
where $z$ is a local conformal coordinate on $M$. Note that since
$N\colon M\to \esf^2$ is harmonic, then $\Delta N + |\nabla N|^2N=0$ and thus, ${\cal J}(M)$
always contains the space $L(N)$ of linear functions of the components of $N$
(which we will refer to as {\it linear Jacobi functions}):
\[
L(N)=\{ \langle N,a\rangle \ | \ a\in \R^3\} \subset {\cal J}(M).
\]

Next we briefly recall some well-known facts about the Weierstrass representation,
see e.g. Osserman~\cite{os3,os1}. Besides the
(meromorphic) stereographic projection $g$ of the Gauss map,
we can associate to every minimal surface $M\subset \R^3$
 a holomorphic differential $dh$ (not necessarily exact)
so that $M$ can be parameterized as $X\colon M\to \R^3$, $X(z)=\Re \int ^z\Psi $, where
\begin{equation}
\label{eq:Psi}
\Psi =\left( \frac{1}{2}(\frac{1}{g}-g),\frac{i}{2}(\frac{1}{g}+g),1\right) dh;
\end{equation}
we call $(g,dh)$ the {\it Weierstrass pair} of $M$.
The so called {\it period problem} for $(g,dh)$
amounts to checking that
$\Re \int _{\G }\Psi =0$ for each closed curve $\G \subset M$.
A simple algebraic calculation demonstrates that this vanishing period
condition is
equivalent to the following one for all closed curves $\G \subset M$:
\begin{equation}
\label{eq:periods}
\overline{\int _{\G }\frac{dh}{g}}=\int _{\G }g\, dh,\qquad \Re \int _{\G }dh=0.
\end{equation}

Suppose that $t\mapsto M_t$ is a (smooth) deformation of $M_0=M$ by minimal surfaces.
Away from the set $B(N)$ of branch points of
the Gauss map of $M$, we can use $g$ as a local conformal coordinate for $M_t$, which gives a Weierstrass
pair $(g,dh(t))$ with $dh(0)=dh$. Since the set of meromorphic differentials on $M-B(N)$ is a linear space and
the $\C^3$-valued differential form $\Psi $ in (\ref{eq:Psi}) depends linearly on $dh$, a formal derivation in
(\ref{eq:Psi}) with respect to $t$ at $t=0$ gives rise to a Weierstrass pair
\[
\left( g,\stackrel{\dot{\frown}}{dh}=\left. \frac{d}{dt}\right| _0dh(t)\right) .
\]
The pair $(g,\stackrel{\dot{\frown}}{dh})$
turns out to solve the period problem, defining a branched minimal immersion $X_v$
(possibly constant) with the same Gauss map as $M$. After identification of the space
of infinitesimal deformations of $M$ by minimal surfaces with
the space ${\cal J}(M)$ of Jacobi functions, we have a correspondence
\begin{equation}
\label{eq:Xv}
v\in {\cal J}(M)\mapsto
X_v=\Re \int ^z\left( \frac{1}{2}(\frac{1}{g}-g),\frac{i}{2}(\frac{1}{g}+g),1\right)
\stackrel{\dot{\frown}}{dh}.
\end{equation}
This correspondence was studied by  Montiel and Ros~\cite{mro1}
(see also Ejiri and Kotani~\cite{ek2}), who wrote down explicitly
$X_v,\stackrel{\dot{\frown}}{dh}$ in terms of $v$ as
\begin{equation}
\label{eq:MontielRos}
\left. \begin{array}{l}
{\displaystyle
X_v=vN+\frac{1}{|N_z|^2}\{ v_zN_{\overline{z}}+v_{\overline{z}}N_z\}\colon
M-B(N)\rightarrow \R^3}
\\
{\displaystyle
\stackrel{\dot{\frown}}{dh}=\frac{g}{g'}\left( v_{zz}+\left( \frac{2\overline{g}g'}{1+|g|^2}
-\frac{g''}{g'}\right) v_z\right) dz}
\end{array}\right\}
\end{equation}
where $z$ is any local conformal coordinate on $M$. The Gauss map of $X_v$ is $N$ and its support function
$\langle X_v,N\rangle $ is $v$. Linear Jacobi functions $v\in L(N)$ produce constant maps $X_v$
(and vice versa), and the correspondence $v\mapsto X_v$ is a
linear isomorphism from the linear space of Jacobi functions on $M$
modulo $L(N)$ onto the linear space of all branched
minimal immersions $X\colon M-B(N)\rightarrow \R^3$ with Gauss map $N$
modulo the constant maps. 

A direct consequence of the derivation in $t=0$ of (\ref{eq:Psi}) is that
the map $v\in {\cal J}(M)\mapsto X_v$ behaves well with respect to fluxes,
as stated in the following lemma. Note that since the flux of a minimal surface
is a homological invariant, we can take the curve $\G $ described below away
from the branch points of the  Gauss map.
\begin{lemma}
\label{lema4.1}
In the above setting, let $\{ \psi _t\colon M\rightarrow \R^3\}
_{|t|<\varepsilon }$ be a smooth deformation of $M$ by minimal
surfaces, and denote by $v\in {\cal J}(M)$ the normal part of its variational field.
For any fixed closed curve $\G \subset M$, we have:
\begin{equation}
\label{eq:derivativeflux}
\left. \frac{d}{dt}\right| _{t=0}
\mbox{{\rm Flux}}(\psi _t,\G )=\mbox{{\rm Flux}}(X_v,\G ).
\end{equation}
\end{lemma}

\begin{definition}
{\rm
Given a minimal surface $M$, the {\it conjugate Jacobi function} $v^*$ of a
Jacobi function $v\in {\cal J}(M)$ is defined (locally) as the support function
$\langle (X_v)^*,N\rangle $
of the conjugate minimal immersion $(X_v)^*$. Recall that such a conjugate minimal immersion
is an isometric minimal immersion of the underlying Riemannian
surface, whose coordinate functions are the harmonic conjugates to the ones of $X_v$.
Note that $v^*$ is defined up to additive constants, and $v^*$ is globally well-defined
precisely when $(X_v)^*$ is globally well-defined. We define
\[
{\cal J}_{\Csmall }(M)=\{ v+iv^*\ | \ v\in {\cal J}(M) \mbox{ and $v^*$ is globally defined}\} .
\]
}
\end{definition}
In other words, ${\cal J}_{\Csmall }(M)$ is the space of support
functions $\langle X,N\rangle $ of holomorphic maps $X\colon
M-B(N)\to \C^3$ whose real and imaginary parts are orthogonal to $N$
(now $\langle X,N\rangle $ denotes the usual bilinear complex
product on $\C^3$). Since $(X_v)^{**}=-X_v$, we deduce that the
conjugate Jacobi of $v^*$ is $-v$, which endows ${\cal J}_{\Csmall
}(M)$ with a structure of complex vector space. A simple observation
is that if $v\in L(N)$, then $X_v$ is constant, which means that
$(X_v)^*$ is also
 constant and so $v^*$ is a (globally defined) function in $L(N)$. In other words,
\begin{equation}
\label{eq:L(N)}
L_{\Csmall }(N):=\{ \langle N,a\rangle \ | \ a\in \C^3\} \subseteq {\cal J}_{\Csmall }(M).
\end{equation}
For a general $v\in {\cal J}(M)$, the map $(X_v)^*$ is globally well-defined provided
that all its period vectors along closed curves on $M$
vanish (equivalently, when Flux$(X_v,\G )=0$ for all closed curves $\G \subset M$).
As a direct consequence of Lemma~\ref{lema4.1}, we have the following statement.
\begin{lemma}
\label{lema4.3}
Given a minimal surface $M\subset \R^3$ and $v\in {\cal J}(M)$, the conjugate Jacobi function
$v^*$ of $v$ is globally defined on $M$ if and only if $v$ preserves infinitesimally
the flux vector along every closed curve on $M$.
\end{lemma}
\begin{remark}
\label{rem4.4}
{\rm
As we will see in Sections~\ref{subsecShiffman} and~\ref{sec4.2}, if $M$ is a minimal surface satisfying
the hypotheses of Assertion~\ref{ass} and $M$ is not a Riemann minimal example,
then it admits a non-zero Jacobi function called its {\it Shiffman function} $S_M$,
whose Jacobi conjugate $S_M^*$ is globally defined.
In particular, if there exists a smooth deformation $M_t$ of $M$ by minimal surfaces such that
for every $t$, $M_t$ also admits a Shiffman function $S_{M_t}$ and the normal part of
$\frac{d}{dt}M_t$ equals the Shiffman function $S_{M_t}$, then the flux of $M_t$ along
any closed curve will be independent of~$t$.
}
\end{remark}


\subsection{The Shiffman Jacobi function.}
\label{subsecShiffman}
Next we recall the definition and some basic properties of the Shiffman function.
In 1956, Shiffman~\cite{sh1} introduced a Jacobi function that incorporates the
curvature variation of the parallel sections of a minimal surface. This function
can be defined locally: assume that $(g(z),dh=dz)$ is the Weierstrass data of a
minimal surface $M\subset \R^3$, where $z$ is a local conformal coordinate in $M$
(in particular, $g$ has no zeros or poles and any minimal surface
admits such a local representation around a point with non-vertical normal vector).
The induced metric $ds^2$ by the inner product of
$\R^3$ is $ds^2=\Lambda ^2|dz|^2$, where $\Lambda
=\frac{1}{2}(|g|+|g|^{-1})$. The horizontal level curves $x_3=c$
correspond to $z_c(y)=c+iy$ in the $z$-plane (here $z=x+iy$ with
$x,y\in \R $) and the planar curvature of this level
curve is
\begin{equation}
\label{eq:curvplanarsect} \kappa _c(y)=\left. \left[
\frac{|g|}{1+|g|^2}\Re \left( \frac{g'}{g}\right) \right] \right|
_{z=z_c(y)},
\end{equation}
where the prime stands for derivative with respect to $z$.
\begin{definition}
\label{defShiffman}
{\rm We define the {\it Shiffman function} of $M$ as
\begin{equation}
\label{eq:uShiffman}
S_M=\Lambda \frac{\partial \kappa
 _c}{\partial y}=\Im \left[ \frac{3}{2}\left( \frac{g'}{g}\right) ^2-\frac{g''}{g}
 -\frac{1}{1+|g|^2}\left( \frac{g'}{g}\right) ^2\right],
 \end{equation}
 where $\Im $ stands for imaginary part.
 }
\end{definition}

Since $\Lambda $ is a positive function, the zeros of $S_M$ coincide
with the critical points of $\kappa _c(y)$. Thus, $S_M$ vanishes
identically if and only if $M$ is foliated by pieces of circles and
straight lines in horizontal planes. In a posthumously published
paper, Riemann~\cite{ri2,ri1} classified all minimal surfaces
with such a foliation property: they reduce to the plane, catenoid,
helicoid and the one-parameter family of surfaces defined in Section~\ref{sec2}.
A crucial property is that
$\Delta S_M+|\sigma |^2S_M=0$,
i.e., $S_M$ is a Jacobi function on $M$. Shiffman himself exploited this
property when he proved that if a minimal annulus $M$ is bounded by two circles
in parallel planes, then $M$ is foliated by circles in the intermediate planes;
see also Fang~\cite{fan4} for other applications of the Shiffman function.

Coming back to our properly embedded minimal surface $M\subset \R^3$ in
the family ${\cal M}$ described in Assertion~\ref{ass}, we deduce from
Theorem~\ref{thmcurvestim} (with the notation in that theorem)
that $M$ intersects transversally every horizontal plane and hence,
its Shiffman function $S_M$ can be defined globally on $M=(\C /\langle i\rangle )-\{ p_j,q_j\} _j$.
Expressing $g$ locally around a zero $p_j$, it is straightforward to check that
$S_M$ is bounded around $p_j$, with continuous extension $S_M(p_j)=-\frac{1}{6}\Im \left(
\frac{g^{(5)}(p_j)}{g''(p_j)}\right) $, and a similar result holds for poles of $g$. Hence
$S_M$ can be viewed as a continuous function on the cylinder $\C /\langle i\rangle $. Since
$v=S_M$ solves the Jacobi equation (\ref{eq:Jacobi}) and when
expressed around $p_j$ or $q_j$ the Jacobi equation has the form
$\Delta v+qv=0$ for $q$ smooth (here $\Delta $ refers to the Laplacian in the flat metric
on $\C /\langle i\rangle $),
elliptic regularity implies that $S_M$ extends smoothly to $\C /\langle i\rangle $.
In fact, Corollary~\ref{corol4.10} below implies that $S_M$ is bounded on $\C /\langle i\rangle $.

\subsection{The space of allowed Gauss maps and their infinitesimal deformations.}
\label{sec4.2}
Our method to prove that every $M\in {\cal M}$ is a Riemann minimal example is based on the fact that
the Shiffman function can be integrated at any $M\in {\cal M}$, in a similar manner as
 a vector field on a manifold
admits integral curves passing through any point. To construct a
framework in which this last sentence makes sense, we need
some definitions. Since surfaces $M\in {\cal M}$ have Weierstrass data
$(g,dz)$ on $\C /\langle i\rangle $,
all the information we need for understanding $M$ is contained
in its Gauss map $g$.
We start by defining the appropriate space of functions where these Gauss maps naturally reside.
\begin{definition}
{\rm
A meromorphic function $g\colon \C/\langle i\rangle \to \C \cup \{ \infty \} $ will be called {\it quasiperiodic}
if it satisfies the following two conditions.
\begin{enumerate}
\item There exists a constant $C>0$ such that the distance between
any two distinct points in $g^{-1}(\{ 0,\infty \} )\subset \C / \langle i\rangle $ is at least $C$
and given any $p\in g^{-1}(\{ 0,\infty \} )$,
there exists at least one point in $g^{-1}(\{ 0,\infty \} )-\{ p\} $ of distance less than
$1/C$ from $p$.
\item For every divergent sequence $\{ z_k\} _k\subset
\C /\langle i\rangle $, there exists a subsequence of the meromorphic functions
$g_k(z)=g(z+z_k)$ which converges uniformly on compact subsets of $\C /\langle i\rangle $ to a non-constant
meromorphic function $g_{\infty }\colon \C/\langle i\rangle  \to \C \cup \{ \infty \} $ (in particular,
$g_{\infty }$ is quasiperiodic as well).
\end{enumerate}
}
\end{definition}

\begin{remark}
\label{remark4.6}
{\rm A direct consequence of the last definition is that if  $g:\C /\langle i\rangle \to \C \cup \{ \infty \} $
is a quasiperiodic meromorphic function, then there is a bound on the order of the zeros and a bound on the order of
the poles of $g$, as well as uniform bounds away from zero and from above for the coefficients of $z^k$ (resp. $z^{-k}$)
in the series expansion of $g$ and its derivatives around any zero (resp. pole) of order $k$ of $g$.
}
\end{remark}
\par
\noindent
We consider the space of meromorphic functions
\[
{\cal W}=\left\{ g\colon \C/\langle i\rangle  \to \C \cup \{ \infty \} \ \mbox{quasiperiodic\ } : \
(g)=\prod _{j\in \Zsmall }p_j^2 q_j^{-2}
\right\} ,
\]
where $(g)$ denotes the divisor of zeros and poles of $g$ on $\C /\langle i\rangle $.
Statement~7 of
Theorem~\ref{thmcurvestim} implies that the Gauss map of every $M\in {\cal M}$ lies in ${\cal W}$.
We endow ${\cal W}$ with the topology of uniform convergence on compact sets of $\C /
\langle i \rangle $.
Given $g\in {\cal W}$, it follows from Remark~\ref {remark4.6} that any limit $g_{\infty }$ of (a subsequence of)
$g_k(z)=g(z+z_k)$ with $\{ z_k\} _k\subset \C /\langle i\rangle $ being a divergent sequence, satisfies that
$g_{\infty }$ lies in ${\cal W}$.
If $g\in {\cal W}$ has divisor of zeros $Z=\prod _jp_j^2$, then the
 set $\{ p_j\} _j$ is quasiperiodic in the sense that for every
divergent sequence $\{ z_k\} _k\subset \C /\langle i\rangle $,
there exists a subsequence of $\{ Z+z_k\} _k$ which converges in the Hausdorff distance
on compact subsets of $\C /\langle i\rangle $
to a divisor $Z_{\infty }$ in $\C /\langle i\rangle $ (analogously for poles).

Reciprocally, two disjoint quasiperiodic divisors $Z=\prod _jp_j^2$,
$P=\prod _jq_j^2$ in $\C /\langle i\rangle $ 
define a unique quasiperiodic
meromorphic function $g$ (up to multiplicative non-zero constants)
whose principal divisor is $(g)=Z/P$: existence follows from Douady
and Douady~\cite{Douady1}:
\[
g(z)=\prod _{n\in \Z }c (n)\frac{\cosh \frac{2\pi z-p_j}{2}\cosh \frac{2\pi z-p_{j+1}}{2}}
{\sinh ^2\frac{2\pi z-q_j}{2}},
\]
where $c(n)$ is a non-zero complex number such that the above infinite product converges,
while uniqueness can be shown as follows.
Suppose $g_1,g_2\in {\cal W}$ have $(g_1)=(g_2)$. Then the function
$f=g_1/g_2$ is holomorphic and has no zeros in $\C /\langle i \rangle $.
If $f$ is unbounded on $\C /\langle i\rangle $, then there exists
$\{ z_k\} _k\subset \C /\langle i\rangle $
 such that $f(z_k)$ diverges. Furthermore, $\{ z_k\} _k$ is a divergent sequence
since $f$ has no poles. By quasiperiodicity of $g_1$ and $g_2$,
after extracting a subsequence we can assume that $f_k(z)=f(z+z_k)$ converges uniformly on compact
subsets of $\C /\langle i\rangle$ to a meromorphic function $f_{\infty }
\colon \C/\langle i \rangle \to \C \cup \{ \infty \}$ which is not constant infinity. Then $f_{\infty }$ has no
poles by Hurwitz's theorem, but $f_k(0)=f(z_k)\to \infty $ as $k\to \infty $, which is a contradiction.
Thus $f$ must be bounded, and so $f$ is constant
by Liouville's theorem.

To each $g\in {\cal W}$ we associate the quasiperiodic set
of its zeros $p_j$ and poles
$q_j$ in $\C /\langle i\rangle $
(we choose an ordering for this set of zeros and poles),
together with the value of $g$ at a prescribed point $z_0\in (\C /\langle i\rangle )-g^{-1}(\{ 0,\infty \} )$.
The bijective correspondence
\begin{equation}
\label{eq:topol}
g\mapsto (p_j,q_j,g(z_0))\in \left[ \Pi_{ j\in\mathbb{Z}}\, (\mathbb{C}/\langle i \rangle)\right]
\times (\C -\{ 0\} )
\end{equation}
allows us to identify ${\cal W}$ (endowed with the uniform topology on compact sets on
$\C /\langle i\rangle $) the space $\left[ \Pi_{ j\in\mathbb{Z}}\, (\mathbb{C}/\langle i \rangle)\right]
\times (\C -\{ 0\} )$ (endowed with its metrizable product topology).

Given $\ve >0$, we denote by  $\D (\ve )=\{ t\in \C \ | \ |t|<\ve \}$.
We say that a curve $t\in \D (\ve )\to g_t\in {\cal W}$ with $g_0=g$ is {\it holomorphic} if
the corresponding functions $p_j(t),q_j(t),g_t(z_0)$ depend holomorphically on~$t$.
In this case, the function $\dot{g}\colon \C /\langle i\rangle \to \C \cup \{ \infty \} $ given by
$z\in \C /\langle i\rangle \mapsto \left. \frac{d}{dt}\right| _{t=0}g_t(z)$
is meromorphic on $\C /\langle i\rangle $. We will call $\dot{g}$ the {\it infinitesimal deformation of $g$
associated to the curve $t\mapsto g_t$}.

If $\dot{g}$ is the infinitesimal deformation of $g=g_0\in {\cal W}$ associated to the curve $t\mapsto g_t$
and $g$ has principal divisor $(g)=\prod _jp_j^2q_j^{-2}$, then the principal divisor of $\dot{g}$ clearly satisfies
\begin{equation}
\label{eq:tgW}
(\dot{g})\geq \prod _jp_jq_j^{-3}.
\end{equation}
In particular, if $\dot{g}$ is constant, then $\dot{g}=0$.
Reciprocally, if $f$ is a meromorphic function on $\C /\langle i\rangle $
and its principal divisor verifies $(f)\geq \prod _{j}p_jq_j^{-3}$, then $f$ is the infinitesimal deformation
of $g$ associated to a holomorphic curve $t\mapsto g_t\in {\cal W}$ with $g_0=g$
 (construct $g_t$ up to a multiplicative constant
$a(t)\in \C -\{ 0\} $ from its quasiperiodic principal divisor
$(g_t)=\prod _jp_j(t)^2q_j(t)^{-2}$ where $p_j(t),q_j(t)$ are holomorphic curves in $\C /\langle i\rangle $
such that
$p_j(0)=p_j$, $q_j(0)=q_j$ and the order of $t\mapsto p_j(t)$ at $p_j$ is chosen according to the order of $f$
at $p_j$ for each $j$; then choose the constant $a(t)$ depending holomorphically on $t$ such that $a(0)=
g(z_0)$).

We will denote by $T_g{\cal W}=\left\{ f\colon \C/\langle i\rangle \to \C \cup \{ \infty\} \mbox{ meromorphic }\ | \
(f)\geq \prod _jp_jq_j^{-3} \right\} $
the set of infinitesimal deformations of $g$ associated to holomorphic curves.
By the above arguments, $T_g{\cal W}$ is a complex linear space.

\begin{remark}
\label{remarkg'}
{\rm
Note that $g,g'\in T_g{\cal W}$ are respectively the
infinitesimal deformations at $t=0$ associated to the holomorphic curves
$t\mapsto (t+1)g(z)$, $t\mapsto g(z+t)$ (from now on, we will denote by prime $'$ the
derivation with respect to the conformal coordinate $z$).
}
\end{remark}

Let $\g =\{ it\ | \ t\in [0,1]\}$ be the generator of the homology of the cylinder $\C /\langle i\rangle $.
Given $g\in {\cal W}$, the pair
$(g,dh=dz)$ is the Weierstrass data of a complete, immersed minimal surface in $\R^3$ with embedded
horizontal planar ends (considered each one separately) at the zeros and poles of
$g$ if and only if the corresponding period problem~(\ref{eq:periods}) can be solved. In our setting, this is
equivalent to solving the following equations.
\begin{equation}
\label{eq:periods1}
\overline{\int _{\g }\frac{dz}{g}}=\int _{\g }g\, dz,\qquad \mbox{Res}_{p_j}\left( \frac{dz}{g}\right) =
\mbox{Res}_{q_j}(g\, dz)=0\quad \forall j\in \Z .
\end{equation}
The above equalities suggest defining the {\it period map} $\mbox{Per}\colon {\cal W}\to \C ^2\times \C^{\Zsmall}\times
\C ^{\Zsmall }$ by
\begin{equation}
\label{eq:periodmap}
\mbox{Per}(g)=\left( \int _{\g }\frac{dz}{g}, \int _{\g }g\, dz, \{ \mbox{Res}_{p_j}\left( \frac{dz}{g}\right) \} _j,
\{ \mbox{Res}_{q_j}(g\, dz)\} _j\right) .
\end{equation}
Inside ${\cal W}$ we have the space of {\it immersed} minimal surfaces, i.e., those $g\in {\cal W}$ such that
$(g,dz)$ solves the period problem:
\begin{equation}
\label{eq:Mimm}
{\cal M}_{\mbox{\footnotesize imm}}=\mbox{Per}^{-1}\{ (a,\overline{a},0,0)\ | \ a\in \C \} .
\end{equation}
\begin{definition}
{\rm
A {\it quasiperiodic, immersed minimal surface of Riemann type} is a minimal surface $M\subset \R^3$ which
admits a Weierstrass pair of the form $(g,dz)$ on $(\C /\langle i\rangle )-g^{-1}(\{ 0,\infty \} )$ where $g$ lies in
${\cal M}_{\mbox{\rm \footnotesize imm}}$.
}
\end{definition}
\begin{remark}
{\rm
Since Residue$_{p_j}(\frac{dz}{g})=-\frac{2}{3}\frac{g'''(p_j)}{g''(p_j)^2}$ and Residue$_{q_j}(g\, dz)=
-\frac{2}{3}\frac{(1/g)'''(q_j)}{(1/g)''(q_j)^2}$, the fact that the pair $(g,dz)$ closes the period
at a zero $p_j$ (resp. at a pole $q_j$) of $g$, can be stated equivalently as $g'''(p_j)=0$
(resp. $(1/g)'''(q_j)=0$).
}
\end{remark}

\begin{definition} \label{def4.12}
{\rm
A {\it Jacobi function} associated to an element $g\in {\cal W}$ is a map
$v\colon (\C /\langle i\rangle )-g^{-1}(\{ 0,\infty \} )\to \R $ which satisfies equation (\ref{eq:Jacobi})
on $(\C /\langle i\rangle )-g^{-1}(\{ 0,\infty \} )$.
The linear space of real-valued Jacobi functions associated to $g$ will be denoted by
${\cal J}(g)$. By equations (\ref{eq:Xv}) and (\ref{eq:MontielRos}), every $v\in {\cal J}(g)$ gives rise to
a branched minimal immersion $X_v\colon (\C /\langle i\rangle ) -B(g)\to \R^3$
with (complex) Gauss map $g$, where $B(g)$ is the branch locus of $g$. The {\it conjugate Jacobi function}
$v^*$ of $v\in {\cal J}(g)$ is the (locally defined) support function
of the conjugate minimal immersion $(X_v)^*$ of $X_v$.
}
\end{definition}

We consider the complex linear space
\[
{\cal J}_{\Csmall }(g)=\{ v+iv^*\ | \ v\in {\cal J}(g) \mbox{ and $v^*$ is globally defined}\} .
\]
${\cal J}_{\Csmall }(g)$ is the space of support functions $\langle X,N\rangle $ of
holomorphic maps $X\colon (\C /\langle i\rangle )-B(g)\to \C^3$ whose real and
imaginary parts are orthogonal to $N=\left(
\frac{2g}{|g|^2+1},\frac{|g|^2-1}{|g|^2+1}\right) \in \C \times \R \equiv \R^3$. The linear
functions of $g$ form a complex linear subspace of ${\cal J}_{\Csmall }(g)$:
\[
L_{\Csmall }(g):=\{ \langle N,a\rangle \ | \ a\in \C^3\} \subseteq {\cal J}_{\Csmall }(g).
\]
For later purposes, it is useful to recognize a basis of $L_{\Csmall }(g)$.
Writing $a=(a_1,a_2,a_3)$ with $a_i\in \C $ and using that $g$ is the stereographic projection
of $N$ from the north pole, we have
\begin{equation}
\label{eq:basislinearg}
\langle N,a\rangle =\frac{2}{1+|g|^2}\left[ a_1\Re (g)+a_2\Im (g)
\right] +a_3\frac{|g|^2-1}{|g|^2+1}=\frac{1}{|g|^2+1}\left( Ag+B\overline{g}
\right) +a_3\frac{|g|^2-1}{|g|^2+1},
\end{equation}
where $A,B\in \C $ are determined by the equations $2a_1=A+B$, $2a_2=i(A-B)$. In particular,
$\frac{g}{|g|^2+1},\frac{\overline{g}}{|g|^2+1},\frac{|g|^2-1}{|g|^2+1}$ is a basis of
$L_{\Csmall }(g)$. We will use this fact in the proof of Corollary~\ref{corol4.10} below.

\begin{definition} \label{def4.12quas}
{\rm
A Jacobi function $v\in {\cal J}(g)$ (resp.
${\cal J}_{\Csmall }(g)$) is said to be {\it quasiperiodic} if for every divergent sequence
$\{ z_k\} _k\subset \C /\langle i\rangle $, there exists a subsequence of the functions
$v_k(z)=v(z+z_k)$ which converges uniformly on compact subsets of
$(\C /\langle i\rangle )-g_{\infty }^{-1}(\{ 0,\infty \} )$
 to a function $v_{\infty }$, where $g_{\infty }\in {\cal W}$ is the limit of (a subsequence of)
$\{ g_k(z)=g(z+z_k)\} _k$, which exists since $g$ is quasiperiodic.
Note that $v_{\infty }\in {\cal J}(g_{\infty })$ (resp.
${\cal J}_{\Csmall }(g_{\infty })$), and that if $v_{\infty }$ is constant, then
$v_{\infty }=0$.
}
\end{definition}

Next we give a condition for a Jacobi function to have a globally defined conjugate Jacobi function.

\begin{proposition}
\label{propos4.8}
Given $g\in {\cal M}_{\mbox{\rm \footnotesize imm}}$, we have:
\begin{enumerate}
\item Let $h\colon \C /\langle i\rangle \to \C \cup \{ \infty \} $ be a meromorphic function
which is a rational expression of $g$ and its derivatives with respect to $z$ up to some order, such that
\begin{equation}
\label{eq:gpuntodeh}
\dot{g}(h)=\left(\frac{g^3h'}{2g'}\right) '
\end{equation}
belongs to $T_g{\cal W}$.
Then, the map
\begin{equation}
\label{eq:Jacobideh}
f(h)=\frac{g^2h'}{g'}+\frac{2gh}{1+|g|^2}
\end{equation}
lies in ${\cal J}_{\Csmall }(g)$, is quasiperiodic and bounded on $\C /\langle i\rangle $.
Furthermore, for every closed curve $\G \subset \C /\langle i\rangle $,
\begin{equation}
\label{eq:periodconstant}
\int _{\G }\frac{\dot{g}(h)}{g^2}dz=\int _{\G }\dot{g}(h)\, dz=0.
\end{equation}
\item Reciprocally, if $\dot{g}\in T_g{\cal W}$ satisfies (\ref{eq:periodconstant}),
then there exists a meromorphic function $h$ on $\C /\langle i\rangle $
such that (\ref{eq:gpuntodeh}) holds.
\end{enumerate}
\end{proposition}
\begin{proof}
We first demonstrate item 1.
Since $\dot{g}(h)\in T_g{\cal W}$, there exists a holomorphic curve $t\mapsto g_t\in {\cal W}$
such that $g_0=g$ and $\left. \frac{d}{dt}
\right| _{t=0}g_t=\dot{g}(h)$. Therefore
$\langle \left. \frac{d}{dt}\right| _0\int ^z\Psi _t,N\rangle
\in {\cal J}_{\Csmall }(g)$, where
$\Psi _t=\left( \frac{1}{2}(\frac{1}{g_t}-g_t),\frac{i}{2}(\frac{1}{g_t}+g_t),1\right) dz$ and
$N=\left( \frac{2\Re (g)}{|g|^2+1},\frac{2\Im (g)}{|g|^2+1},\frac{|g|^2-1}{|g|^2+1}\right)$.

A simple calculation gives
\begin{equation}
\label{eq:primitivas}
\int ^z\frac{\dot{g}(h)}{g^2}dz=\frac{gh'}{2g'}+h,\qquad
\int ^z\dot{g}(h)\, dz=\frac{g^3h'}{2g'}
\end{equation}
up to additive complex numbers, and then for some $a\in \C^3$ we have
\[
\langle \left. \frac{d}{dt}\right| _0\int ^z\Psi _t,N\rangle =
\langle \int ^z
{\textstyle \left( \frac{1}{2}(-\frac{\dot{g}(h)}{g^2}-\dot{g}(h)),
\frac{i}{2}(-\frac{\dot{g}(h)}{g^2}+\dot{g}(h)),0\right) }dz,N\rangle
\]
\[
=\langle {\textstyle \left( \frac{1}{2}(-\frac{gh'}{2g'}-h-\frac{g^3h'}{2g'}),\frac{i}{2}(-\frac{gh'}{2g'}-h+
\frac{g^3h'}{2g'}),0\right) },N\rangle
+\langle a,N\rangle
\]
\[
=\frac{1}{|g|^2+1}\langle {\textstyle \left( \frac{gh'}{2g'}+h\right) (-1,-i,0)
+\frac{g^3h'}{2g'} (-1,i,0),(\Re(g),\Im(g),0)}\rangle
+\langle a,N\rangle
\]
\[
=\frac{1}{|g|^2+1}\left[ -\left( \frac{gh'}{2g'}+h\right) g-\frac{g^3h'}{2g'}\overline{g}\right]
+\langle a,N\rangle
=-\frac{1}{2}f(h)+\langle a,N\rangle .
\]
In summary,
\begin{equation}
\label{eq:propos4.13new}
\langle \left. \frac{d}{dt}\right| _0\int ^z\Psi _t,N\rangle =
-\frac{1}{2}f(h)+\langle a,N\rangle .
\end{equation}

From (\ref{eq:propos4.13new}) we deduce that $f$ is a Jacobi function and lies in ${\cal J}_{\C }(g)$.
Quasiperiodicity of $f(h)$ follows directly from the quasiperiodicity of $g$, since $h$ is a rational
function of $g$ and its derivatives. In order to prove that $f(h)$ is bounded on
$\C /\langle i\rangle $, we first check that $f(h)$ is bounded around every zero
and pole of $g$ and around every zero of $g'$.
\begin{description}
\item[(A)] \underline{Suppose $z=0$ is a zero of $g$.} It suffices to prove that $\frac{g^2h'}{g'}+2gh$ is bounded around
$z=0$. Since $\frac{g^2h'}{g'}+2gh=\frac{(g^2h)'}{g'}$, we must check that $(g^2h)'$ has a zero at $z=0$. From
equation~(\ref{eq:gpuntodeh}) we have  $h'=\frac{2g'}{g^3}\dot{G}$, where $\dot{G}$ is a primitive of $\dot{g}(h)$
defined in a neighborhood of $z=0$. Thus, $h=\int ^z\frac{2g'}{g^3}\dot{G}\, dz=-\frac{1}{g^2}\dot{G}+\int ^z
\frac{\dot{g}(h)}{g^2}$ and $(g^2h)'=-\dot{g}(h)+\left( g^2\int ^z\frac{\dot{g}(h)}{g^2}dz\right) '$.
As $\dot{g}(h)$ vanishes at $z=0$ (here we are using that $\dot{g}(h)\in T_g{\cal W}$ and formula~(\ref{eq:tgW})), it suffices to show that
$\left( g^2\int ^z\frac{\dot{g}(h)}{g^2}dz\right) '$ vanishes at $z=0$, which clearly follows from the fact that $g$ has an order-two zero and $\dot{g}(h)$
vanishes at $z=0$ (also note that this property does not depend on the constant of integration since $g(0)=0$).
\item[(B)] \underline{Suppose $z=0$ is a pole of $g$.} Since $\dot{g}(h)\in T_g{\cal W}$, then
(\ref{eq:tgW}) implies that  $\left(\frac{g^3h'}{2g'}\right) '$ has at most a
order-three pole at $z=0$. Thus, $\frac{g^2h'}{g'}$ is bounded at $z=0$ and so, $h'$ vanishes at $z=0$. Now
we deduce that $\frac{2gh}{1+|g|^2}$ also vanishes at $z=0$.
\item[(C)] \underline{Suppose $g(0)\neq 0$ and $g'(0)=0$.} Then (\ref{eq:tgW}) implies that
$\left(\frac{g^3h'}{2g'}\right) '$ is holomorphic
at $z=0$. Hence the branching order of $h$ at $z=0$ is not less than the branching order of $g$ at the same point.
Then trivially both $\frac{g^2h'}{g'}$, $\frac{2gh}{1+|g|^2}$ are bounded at $z=0$.
\end{description}
The discussion in  items {\bf (A)}, {\bf (B)}, {\bf (C)} shows that if we consider  the discrete set
$A=g^{-1}(\{ 0,\infty \} )\cup (g')^{-1}(0)$, then for every $Q_j\in A$ there exists a disk $D(Q_j)
\subset \C /\langle i\rangle $ and a positive number $C_j$ such that $|f(h)|\leq C_j$ in $D(Q_j)$.
The quasiperiodicity of $g$ and Remark~\ref{remark4.6} insure that both $C_j$ and the radius of $D(Q_j)$ can be
taken independently of $j$. Hence to deduce that $f(h)$ is bounded on $\C /\langle i\rangle $, it suffices to prove that
$f(h)$ is bounded in $(\C /\langle i \rangle)-\cup _jD(Q_j)$. This last property holds because
$g$ is quasiperiodic, $h$ is a rational expression of $g$ and its derivatives, and $f$ is given in $(\C /\langle i \rangle)-\cup _jD(Q_j)$ in terms of $g,h$ by the formula (\ref{eq:Jacobideh}). Hence, $f(h)$
is bounded on $\C /\langle i\rangle $. Finally, (\ref{eq:periodconstant}) is a direct consequence of
(\ref{eq:primitivas}), and item 1 of the proposition is proved.

Concerning item 2, equation (\ref{eq:primitivas}) together with the hypothesis (\ref{eq:periodconstant})
allow us to find a meromorphic function $h$ on $\C /\langle i \rangle $ such that (\ref{eq:gpuntodeh}) holds. This finishes the proof.
\end{proof}
\begin{remark}
{\rm \mbox{}\newline \par
\vspace{-.5cm}
\begin{enumerate}
\item Equation (\ref{eq:periodconstant}) could be interpreted as the fact that $\dot{g}$ lies in the kernel of the differential $d\, \mbox{Per}_g$ of the period map at $g\in {\cal M}_{\mbox{\rm \footnotesize imm}}$, defined as in (\ref{eq:periodmap}).
\item If one takes $h=c_1+\frac{c_2}{g^2}$ in (\ref{eq:gpuntodeh}) with $c_1,c_2\in \C $,
then $\dot{g}(h)=0$ (and vice versa). Furthermore, $f(h)$ is a complex
linear combination of $\frac{g}{1+|g|^2},\frac{\overline{g}}{1+|g|^2}$, which can be viewed as a horizontal linear function of the ``Gauss map'' $g$. Taking $h=\frac{1}{g}$ in (\ref{eq:gpuntodeh}), then $\dot{g}(h)=-\frac{1}{2}g'$ and
$f(h)=\frac{1-|g|^2}{1+|g|^2}$, which is a vertical linear function of $g$.
\end{enumerate}
}
\end{remark}
\begin{corollary}
\label{corol4.10}
Let $M$ be a quasiperiodic, immersed minimal surface of Riemann type.
Then, its Shiffman function $S_M$ given by (\ref{eq:uShiffman}) admits
a globally defined conjugate Jacobi $S_M^*$, and
$S_M+iS_M^*=f$ is given by equation (\ref{eq:Jacobideh})
for
\begin{equation}
\label{Shiffmandeh}
h=h_S=\frac{i}{2}\frac{(g')^2}{g^3}.
\end{equation}
In particular:
\begin{enumerate}
\item Both $S_M,S_M^*$ are bounded on the cylinder $M\cup g^{-1}(\{ 0,\infty \} )$.
\item The corresponding infinitesimal deformation
$\dot{g}_S=\dot{g}(h_S)\in T_g{\cal W}$ is given by
\begin{equation}
\label{gpuntodeShiffman}
\dot{g}_S=\frac{i}{2}\left( g'''-3\frac{g'g''}{g}+\frac{3}{2}\frac{(g')^3}{g^2}\right) .
\end{equation}
\item If $\dot{g}_S=0$ on $M$, then both $S_M,S_M^*$ are linear.
\end{enumerate}
\end{corollary}
\begin{proof}
Note that $h$ defined by (\ref{Shiffmandeh}) is a rational expression of $g$ and $g'$.
A direct computation gives that plugging (\ref{Shiffmandeh}) into (\ref{eq:gpuntodeh}) we obtain
(\ref{gpuntodeShiffman}), and that this last expression has the correct behavior expressed in (\ref{eq:tgW}).
In particular, $\dot{g}_S$ is the tangent vector associated to a holomorphic curve in ${\cal W}$ passing
through $g$ at $t=0$, i.e., $\dot{g}_S\in T_g{\cal W}$.
Using Proposition~\ref{propos4.8}, we deduce item 1 of this corollary. It only remains
to check item~3: If $\dot{g}(h_S)=0$, then (\ref{eq:gpuntodeh}) gives $h_S=b-\frac{c}{g^2}$ for
$b,c\in \C $. After substitution in (\ref{eq:Jacobideh}), we obtain
$S_M+iS_M^*=2c\frac{\overline{g}}{1+|g|^2}+2b\frac{g}{1+|g|^2}$. Hence,
both $S_M,S_M^*$ are linear.
\end{proof}

\section{Holomorphic integration of the Shiffman function.}
\label{sec5}
In this section we prove that the Shiffman function $S_M$ of a quasiperiodic,
immersed minimal surface $M$ of Riemann type can be holomorphically integrated
(Theorem~\ref{thmintegrShiffman} below), in the sense that $M$ can be deformed
by a complex family $t\mapsto M_t$ where $t$ moves in a disk $\D (\ve )\subset \C $
centered at the origin, $M_0=M$, such that each $M_t$ is a quasiperiodic, immersed minimal surface
of Riemann type and at any $t\in \D (\ve )\cap \R $, the normal component of the variational
field of $t\mapsto M_t$ is the Shiffman function of $M_t$. This property will be crucial
in our proof of Assertion~\ref{ass}.

The approach to prove the holomorphic integration of $S_M$
is by means of  meromorphic KdV theory, as we next briefly explain.
By Corollary~\ref{corol4.10}, we can associate to $S_M$ an infinitesimal
deformation $\dot{g}_S$ given by equation~(\ref{gpuntodeShiffman}),
which can be considered to be an evolution equation in complex time $t$
involving certain quasiperiodic meromorphic
functions in the cylinder $\C /\langle i\rangle $ (namely, elements in the
space $\mathcal{W}$). The change of variables
\begin{equation}
\label{eq:unew}
u=-\frac{3(g')^2}{4g^2}+\frac{g''}{2g}
\end{equation}
transforms (\ref{gpuntodeShiffman}) into the meromorphic {\it KdV equation}
\begin{equation}
\label{eq:KdVnew}
\dot{u}=-u'''-6uu'.
\end{equation}

In Remark~\ref{rem5.6} we will motivate the reason  for the change of
variables~(\ref{eq:unew}), which could seem to be mysterious at first sight. Therefore,
we are interested in finding a solution $u(z,t)$ of (\ref{eq:KdVnew})
with initial data $u(z,0)=u(z)$ given by (\ref{eq:unew}).
KdV theory insures that this Cauchy problem admits a unique
solution if the initial data $u(z)$ is an {\it algebro-geometric potential} for KdV
(see the paragraph just before Theorem~\ref{sw} for this notion). Although this
integrability result appears to be rather standard in KdV theory, the  reader
interested in minimal surface theory might be unfamiliar with it. Since  to our knowledge, this
is the first time that this theory is applied to minimal surfaces,  we will include a
self-contained proof of the integrability of the Cauchy problem for KdV with
algebro-geometric initial data (Proposition~\ref{teorintegraru}).

Suppose for the moment that the meromorphic function $u(z)$ given by (\ref{eq:unew})
for $g\in {\cal M}_{\mbox{\rm \footnotesize imm}}$ is algebro-geometric, and so
the solution $u(z,t)$ of (\ref{eq:KdVnew}) with $u(z,0)=u(z)$ exists. In order to construct
the desired complex family $M_t$ of quasiperiodic, immersed minimal surfaces of Riemann type,
or equivalently, their Gauss maps $g_t\in {\cal M}_{\mbox{\rm \footnotesize imm}}$,
we argue as follows. First note that $g=1/y^2$ defines a meromorphic function $y(z)$
on $\C $, and that~(\ref{eq:unew}) implies that the following
Schr\"{o}dinger equation in the variable $z$ is satisfied:
\begin{equation}
\label{eq:Sch}
y''+uy=0.
\end{equation}
Now replace $u(z)$ by $u(z,t)$ in~(\ref{eq:Sch}), with the unknown
$y(z,t)$. We will couple this Schr\"{o}dinger equation with an evolution equation in $y$, in such
a way that the integrability condition of the corresponding system of PDEs is that $u(z,t)$
satisfies~(\ref{eq:KdVnew}). Thus, there exists a solution $y(z,t)$ of this coupled system
of PDEs with $y(z,0)=y(z)$. It turns out that letting $g_t(z)=1/y(z,t)^2$,
then $g_t$ solves (\ref{gpuntodeShiffman}). Of course, there are many technical
aspects of this construction which must be taken into account in order for
$g_t$ to define an element in ${\cal M}_{\mbox{\rm \footnotesize imm}}$.

We finish this summary of the results in this section by indicating why
$u=u(z)$ given by (\ref{eq:unew}) for $g\in {\cal M}_{\mbox{\rm \footnotesize imm}}$
is an algebro-geometric potential for KdV. Equation~(\ref{eq:KdVnew}) is just the second term
in a sequence of infinitesimal flows, called the {\it KdV hierarchy}.
By definition, $u(z)$ is algebro-geometric for KdV if this hierarchy stops at some level,
in the sense that the $n$-th flow in the hierarchy is a linear combination of the preceding flows.
The idea here is to associate to each flow in the KdV hierarchy a bounded Jacobi
function on the initial surface $M$ associated to $g\in {\cal M}_{\mbox{\rm \footnotesize imm}}$ (Theorem~\ref{thm9.2}). Then, the
fact that $u(z)$ is algebro-geometric for KdV will follow from the finite dimensionality
of the linear space of bounded Jacobi functions on $M$, a result that will be proven in Appendix~1
(this finite dimensionality also follows from the more general results in~\cite{cm39}).

\subsection{Algebro-geometric potentials of the KdV equation.}
\label{secKdV}
We first introduce some background properties of the
Korteweg-de Vries equation~KdV.  A presentation of the KdV theory
close to the viewpoint we will need here can be found in Gesztesy
and Weikard~\cite{gewe1} and Joshi~\cite{joshi1}. For a quick
introduction one can read Goldstein and Petrich~\cite{gope1},
where the related mKdV equation ({\it modified} Korteweg-de Vries)
is interpreted as a flow of the curvature of a planar curve; for other
applications of the KdV equation in geometry, see Chern and Peng~\cite{chpe1}.
In the literature one can find different normalizations of the KdV equation
(given by different coefficients for $u''', uu'$ in equation
(\ref{kdv}) below); all of them are equivalent up to a change of variables.
We will follow here the normalization that appears in \cite{joshi1}.

Given a meromorphic function $u(z)$, where $z$ belongs to an open set
$O\subset \C $, we consider the {\it KdV infinitesimal flow},
which is the infinitesimal deformation
\begin{equation}
\label{kdv}
\frac{\partial u}{\partial t} = -u'''-6 u u',
\end{equation}
where as usual, $u'$, $u''$, $u'''$, $u^{(4)},\ldots$ denote the successive derivatives
of $u$ with respect to $z$. Associated to (\ref{kdv}) we have the {\it KdV equation},
an evolution equation where we look for a meromorphic function $u(z,t)$, with $z\in O$ and
$t\in \D (\ve )=\{ t\in \C \ | \ |t|<\varepsilon \} $, satisfying~(\ref{kdv}).
The {\it Cauchy problem} for the KdV equation consists of finding a solution $u(z,t)$ of
(\ref{kdv}) with prescribed initial condition $u(z,0)=u(z)$. In fact, the KdV
infinitesimal flow is one
of the terms in a sequence of infinitesimal flows of $u$, called the {{\it KdV hierarchy}:
\begin{equation}
\label{kdvn}
\left\{ \frac{\partial u}{\partial t_n} = -\partial_z{\cal P}_{n+1}(u)\right\} _{n\geq 0},
\end{equation}
where ${\cal P}_{n+1}(u)$ is a differential operator given by a polynomial expression of $u$
and its derivatives up to order $2n$. These operators are defined by the recurrence law
\begin{eqnarray}
\label{law}
\left\{ \begin{array}{l}
\partial_z {\cal P}_{n+1}(u) = (\partial_{zzz} + 4u\,\partial_z+2u'){\cal P}_{n}(u), \\
\rule{0cm}{.5cm}{\cal P}_{0}(u)=\frac{1}{2}.
\end{array}\right.
\end{eqnarray}
In particular, the first operators and infinitesimal flows of the KdV hierarchy are given by
\begin{equation}
\label{eq:KdVhie}
\mbox{}\hspace{-.9cm}
\left.
\begin{array}{l}
{\cal P}_{1}(u)=u\\
\rule{0cm}{.5cm}{\cal P}_{2}(u)=u''+3u^2\qquad \mbox{(KdV)}\\
\rule{0cm}{.5cm}{\cal P}_{3}(u)=u^{(4)}+10 u\, u''+5(u')^2+10u^3\\
\mbox{}\hspace{.7cm}\vdots
\end{array}
\right|
\begin{array}{l}
\frac{\partial u}{\partial t_0} = -u'\\
\rule{0cm}{.5cm}\frac{\partial u}{\partial t_1} = -u'''- 6u\, u'\qquad \mbox{(KdV)}\\
\rule{0cm}{.5cm}\frac{\partial u}{\partial t_2} = -u^{(5)}- 10 u\, u''' -20 u'u''-30 u^2u'\\
\mbox{}\hspace{.7cm}\vdots
\end{array}
\end{equation}
The Cauchy problem for the $n$-th equation of the KdV hierarchy consists
of finding a solution $u(z,t)$ of $\frac{\partial u}{\partial t_n}=
-\partial _z{\cal P}_{n+1}(u)$ with prescribed initial condition $u(z,0)=u(z)$.

A function $u(z)$ is said to be an {\it algebro-geometric potential of the KdV equation}
(or simply {\it algebro-geometric}) if there exists an infinitesimal
flow $\frac{\partial u}{\partial t_n}$
which is a linear combination of the lower order infinitesimal flows:
\begin{equation}
\label{eq:algcj}
\frac{\partial u}{\partial t_n} =  c_0 \frac{\partial u}{\partial t_0}+ \ldots + c_{n-1}\frac{\partial u}{\partial t_{n-1}},
\end{equation}
with $c_0,\ldots,c_{n-1}\in \C $. The next statement collects some important
properties of algebro-geometric potentials.
\begin{theorem}
\label{sw}
Let $u(z)$ be an algebro-geometric potential. Then:
\begin{enumerate}
\item $u$ extends to a meromorphic function $u\colon \C \to \C
\cup \{\infty\}$.
\item If $u$ has a pole at $z=z_0$, then its Laurent expansion around $z_0$ is given by
\[
u(z)=\frac{-k(k+1)}{(z-z_0)^2} + \mbox{\rm holomorphic}(z),
\]
for a suitable positive integer $k$.
\item All the solutions of the linear Schr\"{o}dinger equation $y''+u\, y =0$
are meromorphic functions $y\colon \C \to \C \cup \{\infty\}$.
\end{enumerate}
\end{theorem}
Item $1$ is due to Segal and Wilson \cite{SeWi} and can be found also in
Weikard \cite{weik1} and Gesztesy and Weikard \cite{gewe1}. Items $2$ and
$3$ are proved in~\cite{gewe1} and~\cite{weik1}.

Another fundamental property of algebro-geometric potentials is that the
Cauchy problem for any infinitesimal flow of the KdV hierarchy is uniquely solvable
in the class of algebro-geometric potentials (with fixed coefficients $c_j$ in
equation~(\ref{eq:algcj})). This integrability follows from the commutativity of any
two infinitesimal flows of the KdV hierarchy. We give now a direct proof of this well-known
fact in the particular case of the KdV infinitesimal flow
(\ref{kdv}), which we will use later.

The infinitesimal flow $\frac{\partial u}{\partial t_n}$  defines naturally
 a differential operator  $\frac{\partial }{\partial t_n}$
which acts on differential expressions of $u$ and its derivatives. For instance, we have
\[
\frac{\partial }{\partial t_n}\left( u'\right) = \left(\frac{\partial u}{\partial t_n}
\right)' \qquad {\rm and} \qquad \frac{\partial }{\partial t_n}\left( \frac{u''}{u}+u^2\right)
= \frac{1}{u}\left(\frac{\partial u}{\partial t_n}\right)''+\left(
2u- \frac{u''}{u^2}\right) \frac{\partial u}{\partial t_n}.
\]
\begin{lemma}
\label{lemacomm}
The KdV infinitesimal flow $\frac{\partial }{\partial t}=\frac{\partial }{\partial t_1}$ commutes
with any other infinitesimal flow $\frac{\partial }{\partial t_n}$ in the KdV
hierarchy:
\begin{equation}
\label{conmutan}
\frac{\partial}{\partial t}\frac{\partial u}{\partial t_n} =
\frac{\partial}{\partial t_n}\frac{\partial u}{\partial t}.
\end{equation}
\end{lemma}
{\it Sketch of the proof.}
The proof is by induction on $n$. The lemma clearly holds for $n=1$.
Assuming that $\frac{\partial }{\partial t}$ commutes with $\frac{\partial }
{\partial t_{n-1}}$ we want to prove (\ref{conmutan}). We will write simply
${\cal P}_n$ instead of ${\cal P}_n(u)$. It follows from
equations~(\ref{kdvn}) and~(\ref{eq:KdVhie}) that
\begin{equation}
\label{ut}
\frac{\partial u }{\partial t} =  -(u''+3u^2)'.
\end{equation}
Thus, the induction hypothesis implies that the two expressions below have the same value:
\[
\begin{array}{l}
\frac{\partial}{\partial t_{n-1}} \frac{\partial u }{\partial t} =
-\frac{\partial}{\partial t_{n-1}} (u''+3u^2)' =
-\left( \left(\frac{\partial u}{\partial t_{n-1}}\right)''
+ 6\, u\frac{\partial u}{\partial t_{n-1}} \right)',
\\
\rule{0cm}{.5cm}
\frac{\partial}{\partial t} \frac{\partial u }{\partial t_{n-1}} =
-\frac{\partial}{\partial t}{\cal P}_n'=
- \left( \frac{\partial}{\partial t}{\cal P}_n\right)'.
\end{array}
\]
Therefore, there exists $c\in \C $ such that
$\frac{\partial}{\partial t}{\cal P}_n=
(\frac{\partial u}{\partial t_{n-1}})'' +
6\, u\frac{\partial u}{\partial t_{n-1}} + c$.
We claim that $c=0$: since ${\cal P}_n$ is a polynomial expression of
$u$ and its derivatives, if we differentiate ${\cal P}_n$ with respect to $t$
and we make the substitution $\frac{\partial u }{\partial t} = -u'''-6u'$,
we will obtain another polynomial expression in the variables $u$
and its derivatives without independent term,
which gives our claim. Therefore
$
\frac{\partial}{\partial t}{\cal P}_n=
\left(\frac{\partial u}{\partial t_{n-1}}\right)''
+ 6\, u\frac{\partial u}{\partial t_{n-1}}
$
which, using that $\frac{\partial u}{\partial t_{n-1}}=-{\cal P}_{n}'$,
transforms into
\begin{equation}
\label{induccion}
\frac{\partial}{\partial t}{\cal P}_n=
-{\cal P}_{n}''' - 6\, u{\cal P}_{n}'.
\end{equation}

We are now ready to prove the commutativity at the $n$-th level: Using
equations (\ref{kdvn}) and (\ref{law}) we have
\[
\frac{\partial}{\partial t} \frac{\partial u }{\partial t_{n}} =
-\frac{\partial}{\partial t}\left[
\left(\partial_{zzz} + 4 u \partial_z + 2u'\right){\cal P}_n\right] =
-\left(\partial_{zzz} + 4 u \partial_z + 2u'\right)\frac{\partial
{\cal P}_n}{\partial t}
-4\frac{\partial u}{\partial t}{\cal P}_n'
-2\left(\frac{\partial u}{\partial t}\right)' {\cal P}_n.
\]
Substituting (\ref{ut}) and (\ref{induccion}) in the last expression,
we find a polynomial expression $E_1$ of $u$, ${\cal P}_n$ and their
derivatives with respect to $z$, for
$\frac{\partial}{\partial t} \frac{\partial u }{\partial t_{n}}$.
On the other hand, (\ref{ut}) gives
\[
\frac{\partial}{\partial t_{n}} \frac{\partial u }{\partial t} =
-\frac{\partial}{\partial t_{n}}\left[ (u''+3u^2)'\right] =
-\left( (\frac{\partial u}{\partial t_{n}})'' +
6\, u\frac{\partial u}{\partial t_{n}} \right)',
\]
which combined with the recurrence law $ \frac{\partial u}{\partial t_{n}} =
-(\partial_{zzz} +4u\partial_z+2u'){\cal P}_n $ gives a
polynomial expression $E_2$ in the variables $u$, ${\cal P}_n$ and
its derivatives, too. Comparing both expressions $E_1,E_2$, a
lengthy but direct computation shows that
\[
\frac{\partial}{\partial t}\frac{\partial u}{\partial t_n}
- \frac{\partial}{\partial t_n}\frac{\partial u}{\partial t}=0,
\]
which proves the lemma.
 {\hfill\penalty10000\raisebox{-.09em}{$\Box$}
\par
\vspace{.2cm}
Next we prove the integrability of the KdV infinitesimal flow for an
algebro-geometric initial condition.

\begin{proposition}
\label{teorintegraru}
Let $u=u(z):\C \to \C \cup\{\infty\}$ be
an algebro-geometric potential of the KdV equation so that
$\frac{\partial u}{\partial t_n} =  c_0 \frac{\partial u}{\partial t_0}+
\ldots + c_{n-1}\frac{\partial u}{\partial t_{n-1}}$. Then, there exist
$\varepsilon >0$ and a unique map $u=u(z,t)\colon \C\times
\D(\varepsilon)\to \C\cup\{\infty\}$ such that the following properties hold.
\begin{enumerate}
\item $u(z,0) = u(z)$ and $u_t(z)=u(z,t)$ is algebro-geometric, for each $t\in \D (\ve )$.
\item $u(z,t)$ is holomorphic in $\{(z,t)\in \C \times
\D (\ve )\ : \ |u(z,t)|<\infty\}$ and is a solution of the system of
partial differential equations
\begin{equation}
\label{eq:(AG)(KdV)}
\left. \begin{array}{l}
{\mbox{\rm (A-G)\hspace{1cm}}\displaystyle \frac{\partial u}{\partial t_n} - c_0 \frac{\partial u}{\partial t_0}
- \cdots - c_{n-1}\frac{\partial u}{\partial t_{n-1}}=0
}
 \\
\rule{0cm}{.7cm}{\mbox{\rm (KdV)\hspace{0.9cm}}\displaystyle \frac{\partial u}{\partial t} =
-u'''-6uu'}
\end{array}\right\}
\end{equation}
%
%
where as usual, prime denotes derivative with respect to $z$.
\item If there exists $\omega \in \C $ such that $u(z+\omega)=u(z)$
for all $z\in \C $, then $u(z+\omega,t)=u(z,t)$ for all $z\in \C $.
\item If the jet
\[
J(z_0)=\left( u(z_0),u'(z_0),\ldots, u^{(2n)}(z_0)\right) \in \C^{2n+1}
\]
is bounded by a constant $C>0$, then there exist
 $\delta>0$ and $C_1=C_1(\delta,C)>0$ such that $u(z,t)$,
$u'(z,t)$ and $\frac{\partial u}{\partial
t}(z,t)$ are holomorphic functions bounded by $C_1$  in
$\{ z\in \C \ : \ |z-z_0|<\de \} \times \D(\delta)$.
\end{enumerate}
\end{proposition}
\begin{proof}
We will use the notation $\frac{\partial }{\partial s}=\frac{\partial
}{\partial t_n} -  c_0 \frac{\partial }{\partial t_0}- \ldots -
c_{n-1}\frac{\partial }{\partial t_{n-1}}$. Hence the
system~(\ref{eq:(AG)(KdV)}) can be equivalently written as
\begin{eqnarray}
\label{eq:A-G,KdV}
\left\{ \begin{array}{l}
{\displaystyle \frac{\partial u}{\partial s}=0,}
 \\
\rule{0cm}{.7cm}{\displaystyle
\frac{\partial u}{\partial t} = -u''' - 6 u u'.}
\end{array}\right.
\end{eqnarray}
According to the Frobenius Theorem, the integrability condition of
(\ref{eq:A-G,KdV}) is given by the commutativity $\frac{\partial
}{\partial s}\frac{\partial u}{\partial t}=\frac{\partial }{\partial
t} \frac{\partial u}{\partial s}$ which follows from
Lemma~\ref{lemacomm}. Therefore given any $z_0\in \C$ which is not a
pole of $u(z)$, there exists a positive number $\de $ and a unique solution
$u(z,t)$, $(z,t)\in \{ |z-z_0|<\delta)\} \times \D(\delta)$,
of the system~(\ref{eq:(AG)(KdV)}) with initial conditions
\begin{equation}
\label{jet}
\frac{\partial^k u}{\partial z^k }(z_0,0)= u^{(k)}(z_0), \qquad
k=0,\ldots,2n.
\end{equation}
(Note that the operator $\frac{\partial }{\partial s}$ involves derivatives
with respect to $z$ up to order $2n+1$). As $u_t$ satisfies (A-G), then
it is algebro-geometric. Thus part~1 of Theorem \ref{sw} insures that
$u_t$ extends meromorphically to the whole plane $\C $.
As equation (A-G) is an ODE in the variable $z$ and $\frac{\partial u}
{\partial t_n}$ involves derivatives with respect to $z$ up to order
$2n+1$, it follows from the initial condition (\ref{jet}) that
$u(z,0)=u(z)$. This proves items 1, 2 of
Proposition~\ref{teorintegraru}.

Item 3 of the proposition follows easily from the uniqueness part, and
the local estimate in item~4 is the standard dependence of the solution
of an initial value problem on the initial data.
\end{proof}

Our next result describes the evolution in time of the poles of
a solution of the Cauchy problem for the KdV equation, for a special case
which we will find when applying this machinery to a quasiperiodic, properly
immersed minimal surface of Riemann type.

\begin{theorem}
\label{integrar2}
Let $u=u(z)\colon \C /\langle i\rangle \to \C \cup\{\infty\}$ be a quasiperiodic
algebro-geometric potential on the cylinder, whose Laurent expansion around
each pole $z_0$ of $u$ is given by
\begin{equation}
\label{eq:thm6.4*}
u(z)=\frac{-2}{(z-z_0)^2}+\mbox{\rm holomorphic}(z).
\end{equation}
Let $u=u(z,t)\colon \C /\langle i\rangle \times \D(\varepsilon)\to  \C \cup\{\infty\}$
be the solution of the system~{\rm (\ref{eq:(AG)(KdV)})}
with initial data $u(z)$. Then, the following properties hold.
\begin{enumerate}
\item $u(z,t)$ is meromorphic (as a function of two variables) and
$u_t(z)=u(z,t)$ is quasiperiodic for each $t$.
\item Given a pole $z_0$ of $u(z)$, there exists a holomorphic curve
$t\in \D(\ve )\mapsto z_0(t)$ with $z_0(0)=z_0$, such that in
a neighborhood of $(z_0,0)$ we have
\begin{equation}
\label{eq:thm6.4A}
u(z,t)=\frac{-2}{(z-z_0(t))^2}+\mbox{\rm holomorphic}(z,t).
\end{equation}
Moreover, all the poles of $u_t(z)$ are obtained in this way.
\end{enumerate}
\end{theorem}
\begin{proof}
Since $u(z)$ is algebro-geometric,
Proposition~\ref{teorintegraru} gives a unique solution of (\ref{eq:(AG)(KdV)})
with $u(z,0)=u(z)$, $z\in \C $. Furthermore, $u(z,t)$
is holomorphic in $\{(z,t)\in \C \times
\D (\ve )\ : \ |u(z,t)|<\infty\}$
and $u_t(z)=u(z,t)$ descends to
$\C /\langle i\rangle $.

We next prove that every pole $z_0$ of
$u(z)$ propagates holomorphically in $t$ to a curve of poles $z_0(t)$ of
$u_t(z)$ with the desired Laurent expansion. Let $D$ be a
closed disk centered at $z_0$, such that $u(z)$ does not vanish in
$D-\{ z_0\} $. By continuity, $u_t(z)$ has no
zeros in $\partial D$ for $|t|$ sufficiently small. Recall that
$u_t$ is meromorphic in $\C $ since it is algebro-geometric.
By the argument principle,
\begin{equation}
\label{eq:8.4A}
\# (u_t^{-1}(\infty )\cap D)-\# (u_t^{-1}(0)\cap D)=
\# (u^{-1}(\infty )\cap D)-\# (u^{-1}(\infty )\cap D)=1.
\end{equation}
Let $a_1,\ldots ,a_{m}$ the poles of $u_t$ in $D$ (both $m$ and the $a_j$
may depend on $t$). As $u_t$ is algebro-geometric, part~2 of Theorem~\ref{sw}
insures that there exist positive integers $k_1,\ldots ,k_m$ such that
\[
u_t(z)=\frac{-k_j(k_j+1)}{(z-a_j)^2}+\mbox{\rm holomorphic}(z,t)
\]
in a neighborhood of $a_j$. Since the residue of $u_t$ at $a_j$ is zero for all $j$,
there exists a meromorphic function $v_t(z)$ defined on
$D$ such that $v_t'=u_t$ in $D$. Moreover $v_t$ is unique up to an additive constant,
which we choose so that $v_t$ has value $1$ at some point $p_0\in \partial D$ and
$v_t|_{\partial D}$ has no zeros. The Laurent expansion of $v_t$ around its poles is
\[
v_t(z)=\frac{k_j(k_j+1)}{z-a_j}+\mbox{\rm holomorphic}(z,t).
\]
Since $S(t)=\sum _jk_j(1+k_j)=\frac{1}{2\pi i}\int _{\partial D}v_t(z)ds$ is continuous
and integer-valued, $S(t)$ is constant. As $S(0)=2$, we conclude that $v_t$ has just one
pole which is simple, i.e., $m=1$ and $k_1=1$. Thus, $u_t$ has just one pole $z_0(t)$
in $D$ (which has order two) with the coefficient $-2$ for the term in $(z-z_0(t))^{-2}$
in its Laurent expansion, i.e., equation~(\ref{eq:thm6.4A}) holds. By (\ref{eq:8.4A}),
$u_t$ has no zeros in $D$ for $|t|<\ve $.

Next we prove that the curve $t\mapsto z_0(t)$ is holomorphic (note that we
cannot use the implicit function theorem since the function $w$ below
is not known to be holomorphic as a function of two variables around $(z_0,0)$).
Consider the function
\[
w(z,t)=\frac{1}{v_t(z)},\quad (z,t)\in D \times \D(\ve ).
\]
Note that $w(z,t)$ is holomorphic in $z$ since $v_t$ does not vanish in $D$
(this follows because $A(t)=\# (v_t^{-1}(\infty ))-\# (v_t^{-1}(0))$ is constant,
$A(0)=1$ and $v_t$ has a unique pole in $D$ counted with multiplicities).
For $t$ fixed, $w(\cdot ,t)$ has a unique zero $z_0(t)$ in $D$, and $z_0(0)=z_0$.
As a function of two variables, $w(z,t)$ is holomorphic outside
$\{ (z_0(t),t)\ | \ t\in \D (\ve )\} $. Now the holomorphicity of $t\mapsto z_0(t)$
is a consequence of the following observation. The function $w_t(z)=w(z,t)$
is holomorphic in the closure $\overline{D}$ of $D$, has a simple
zero at $z_0(t)\in D$ and no more zeros in $\overline{D}$. Hence,
$\frac{z w_t'(z)}{w_t(z)}dz$ is a meromorphic differential in
$\overline{D}$ with a simple pole at $z=z_0(t)$ and thus,
\[
\int_{\partial D} \frac{z w_t'(z)}{w_t(z)}dz = 2\pi i
\mbox{ Res}_{z_0(t)}\left( \frac{z w_t'(z)}{w_t(z)}dz\right) =
2\pi i z_0(t).
\]
Since the integrand in the left-hand-side of the above formula depends holomorphically
on $t$, the same holds for $z_0(t)$. This argument proves the following
property that we state separately for future reference.
\begin{assertion}
\label{ass6.5}
Let $h(z,t)$ be a holomorphic function in $\{(z,t)\in \D(\ve ) \times
\D (\ve )\ : \ |h(z,t)|<\infty\}$,
such that $z\mapsto h_t(z)=h(z,t)$ has exactly one zero in $\D (\ve )$ counting multiplicities, and
$h(0,0)=0$. Then, there exists a holomorphic curve $\alpha (t)$,
$|t|<\varepsilon$, such that the zeros of
$h$ in a neighborhood of $(0,0)$ are given by the trace of $\alpha$.
\end{assertion}
We now return to the proof of Theorem~\ref{integrar2}.
To prove the first part of item~1, we only need to check
that $u$ is meromorphic around the points in
$\Gamma=\{(z_0(t),t)\ |\ |t|<\varepsilon\}$ where
$u=\infty$. This follows from equation (\ref{eq:thm6.4A}):
 as $u(z,t)+2(z-z_0(t))^{-2}$ is holomorphic and bounded outside
of the analytic subset $\Gamma$, it extends holomorphically through
$\Gamma$.

It remains to check that $u_t(z)$ is quasiperiodic for $|t|$ sufficiently small.
This fact will hold if we prove the following inequality for the spherical gradient of $u_t(z)$:
\begin{equation}
\label{eq:8.4B}
\frac{|u'_t(z)|}{1+|u_t(z)|^2}\leq C
\end{equation}
for all $z\in \C$, where $C>0$ is independent of $z$.
Repeating the arguments above at every pole $z_{0,j}$ of $u$,
we obtain a sequence of pairwise disjoint closed disks $\{ D_j\} _j$ such that
each $D_j$ is centered at $z_{0,j}$, for $|t|$ small (independently of $j$)
$u_t(z)$ has a unique pole at $z_{0,j}(t)\in D_j$,
and the curve $t\mapsto z_{0,j}(t)$ is holomorphic in $t$.
Note that since $u(z)$ is quasiperiodic, the radii of the $D_j$ can be taken independently
of $j$. Since the jet $J(z_1)=\left( u(z_1),u'(z_1),\ldots ,u^{(2n)}(z_1)\right) $ is uniformly
bounded in $\C^{2n+1}$ for $z_1\in (\C /\langle i\rangle )-\cup _jD_j$
(because $u$ is quasiperiodic), part~4 of Proposition~\ref{teorintegraru}
implies that both $u(z,t),u'(z,t)$ are uniformly bounded for
$(z,t)\in \left[ (\C /\langle i\rangle )-\cup _jD_j\right] \times \D(\ve )$ for $\ve $
sufficiently small. Therefore, (\ref{eq:8.4B}) holds outside $\cup _jD_j$ with $C$ uniform in
$t$. Now consider one of the disks $D_j$. For $t$ fixed, $u_t|_{D_j}$ omits a neighborhood
of zero which is independent of $t$ (this property needs estimates for $u_t$ in a
slightly bigger disk, which we may assume). By Montel's theorem, $\{ u_t|_{D_j}\} _t$
form a normal family, which implies that (\ref{eq:8.4B}) holds for $z\in D$ uniformly in $t$.
Now the proof is complete.
\end{proof}

\subsection{The Shiffman hierarchy associated to a Riemann type minimal surface.}
\label{secKdV2}
Let $M\in {\cal M}$ be a quasiperiodic, immersed minimal surface
of Riemann type, with Gauss map $g\in {\cal M}_{\mbox{\footnotesize imm}}$.
In this section we will associate to $g$ a sequence of
infinitesimal deformations $\frac{\partial g}{\partial t_n}$ which generalizes
the tangent vector $\dot{g}_S\in T_g{\cal W}$ associated
to the complex Shiffman function, which was given in equation (\ref{gpuntodeShiffman}).
For this reason, we call this sequence the {\it Shiffman hierarchy.}
In order to define the Shiffman hierarchy, we will first define a related
hierarchy associated to a linear Schr\"{o}dinger equation.

Consider meromorphic functions $y,u,g\colon \C \to \C \cup\{\infty\}$ related by the equations
\[
y''+u \, y=0\qquad {\rm and} \qquad g=\frac{1}{y^2}.
\]
From these relations we obtain
\begin{equation}
\label{u}
u =-\frac{3(g')^2}{4g^2}+\frac{g''}{2g}.
\end{equation}
\begin{remark}
\label{rem5.6}
{\rm The reader may wonder why the KdV equation appears in
connection to the Shiffman function. The change of variables $x=g'/g$
transforms the expression (\ref{gpuntodeShiffman}) for $\dot{g}_S$
into an equation of mKdV type, namely $\dot{x}=\frac{i}{2}(x'''-
\frac{3}{2}x^2x')$. It is well known that  mKdV equations in $x$ can
be transformed into KdV equations in $u$ through the so called {\it
Miura transformations,} $x\mapsto u=ax'+bx^2$ with $a,b$ suitable
constants (see for example~\cite{gewe1} page~273).
Equation~(\ref{u}) is nothing but the composition of $g\mapsto x$
and a Miura transformation. Since the KdV theory is more standard
than the mKdV theory, we have opted to deal only with the KdV
equation and avoid dealing with the mKdV equation. }
\end{remark}
The {\it Schr\"{o}dinger hierarchy} is defined as a sequence of infinitesimal
flows of $y$ given by
\begin{equation}
\label{eq:SchH}
\left\{ \frac{\partial y}{\partial t_n} =
{\cal P}_n(u)' y - 2 {\cal P}_n(u)y'\right\} _{n\geq 0},
\end{equation}
where ${\cal P}_n(u)$ is the polynomial expression of $u$ and its derivatives given by
equation (\ref{law}). The connection between the Schr\"{o}dinger and
the KdV hierarchies comes from the fact that the integrability
conditions for the system of partial differential equations
\begin{eqnarray}
\label{PDEy}
\left\{ \begin{array}{l}
y''+uy=0 \\
\rule{0cm}{.5cm}\frac{\partial y}{\partial t_n}
= {\cal P}_n(u)' y - 2 {\cal P}_n(u)y'
\end{array}\right.
\end{eqnarray}
are precisely that $u(z,t)$ satisfies the $n$-th equation of the KdV hierarchy,
see Joshi \cite{joshi1}. Both hierarchies are related by
\begin{equation}
\label{relacion}
\frac{\partial u}{\partial t_n} =
-\frac{\partial }{\partial t_n}\left( \frac{y''}{y}\right) .
\end{equation}
If we rewrite these infinitesimal flows in terms of $g$, we obtain the following sequence
of infinitesimal flows of $g$, which we call the {\it Shiffman hierarchy:}
\[
\frac{\partial g}{\partial t_n} =
\frac{\partial }{\partial t_n}\left( \frac{1}{y^2}\right) =
-\frac{2}{y^3}\frac{\partial y}{\partial t_n} =
-2\frac{{\cal P}_n(u)' y - 2 {\cal P}_n(u)y'}{y^3}=
-2\,\partial_z\left(\frac{{\cal P}_n(u)}{y^2}\right) .
\]
By construction, we have the following statement.
\begin{lemma}
The Shiffman hierarchy is given by
\begin{equation}
\label{sh}
\left\{ \frac{\partial g}{\partial t_n} =
-2\,\partial_z (g{\cal P}_n(u))\right\} _{n\geq 0}.
\end{equation}
\end{lemma}
If we compute explicitly these infinitesimal flows solely in terms of $g$
(by substitution of (\ref{law}) and (\ref{u}) in (\ref{sh})),
each right-hand-side is a rational expression in
$g$ and its derivatives. The first infinitesimal flow in this hierarchy is
the infinitesimal deformation
in ${\cal W}$ given by translations in the parameter domain, and the second one is,
up to a multiplicative constant, the infinitesimal deformation $\dot{g}_S$ given by
(\ref{gpuntodeShiffman}),
which corresponds to the Shiffman function (recall from Remark~\ref{remarkg'} and
Corollary~\ref{corol4.10} that both infinitesimal deformations lie in $T_g{\cal W}$).
We also provide the expression for the third infinitesimal flow:
\[
\begin{array}{l}
\frac{\partial g}{\partial t_0} =-g' \\
\rule{0cm}{.5cm}\frac{\partial g}{\partial t_1} =
-g'''+3\frac{g'g''}{g}-\frac{3}{2}\frac{(g')^3}{2g^2} \\
\rule{0cm}{.5cm}\frac{\partial g}{\partial t_2} =
-g^{(5)}+5\frac{g'g^{(4)}}{g}+10\frac{g''g'''}{g}
-\frac{35}{2}\frac{(g')^2g'''}{g^2}-\frac{55}{2}\frac{g'(g'')^2}{g^2}
+\frac{95}{2}\frac{(g')^3g''}{g^3}-\frac{135}{8}\frac{(g')^5}{g^4}\\
\hspace{0.7cm}
\vdots
\end{array}
\]

Another key reason why we are interested in the KdV equation is that
the associated Shiffman hierarchy provides a sequence of bounded
Jacobi functions on any quasiperiodic, immersed minimal surface of
Riemann type, as we now explain. If we let $g\in
{\cal M}_{\mbox{\footnotesize imm}}$ be the complex Gauss map of such a surface,
observe that $g$ has order-two zeroes and order-two poles without residues.
Thus there exists a meromorphic function $y\colon \C \to  \C \cup\{\infty\}$
such that $g=1/y^2$. Moreover $y$ is either periodic, $y(z+i)=y(z)$, or
antiperiodic $y(z+i)=-y(z)$. The later one is the case for the
Riemann minimal examples ${\cal R}_h$, $h>0$ (this follows since the
Gauss map $g$ of ${\cal R}_h$ restricts to each compact horizontal
section with degree one).

\begin{theorem}
\label{thm9.2}
If $g\in {\cal M}_{\mbox{\rm \footnotesize imm}}$, then each of the infinitesimal
flows $\frac{\partial g}{\partial t_n}$ in the Shiffman hierarchy produces a
Jacobi function $f(h_n)\in {\cal J}_{\Csmall }(g)$, which is bounded and quasiperiodic on
$\C /\langle i \rangle $.
\end{theorem}
\begin{proof}
First observe that equation (\ref{u}) and the fact that ${\cal
P}_n(u)$ is a polynomial expression of $u$ and its derivatives
implies that $\frac{\partial g}{\partial t_n}$ is meromorphic and
quasiperiodic, with poles only at (some of) the zeroes and poles of
$g$. To prove the theorem we will use Proposition~\ref{propos4.8}; hence it
suffices to demonstrate the following statement
\par
\noindent
\begin{assertion}
\label{ass9.3}
Under the hypotheses of Theorem~\ref{thm9.2}, for any $n$ there exists
a meromorphic function $h_n$ on $\C /\langle i\rangle $ which is a rational expression
of $g$ and its derivatives up to some order (depending on $n$),
such that $\frac{\partial g}{\partial t_n}=
\partial_z\left( \frac{g^3h_n'}{2g'}\right) $ and
$\frac{\partial g}{\partial t_n}\in T_g{\cal W}$.
\end{assertion}
\par
\noindent
{\it Proof of Assertion~\ref{ass9.3}.} Viewing the equation
$-{2}\partial_z (g{\cal P}_n(u))=\partial_z\left( \frac{g^3h_n'}{2g'}\right) $
as an ODE for the unknown $h_n$ and substituting $g=1/y^2$, we have
\[
\frac{1}{4}h_n'=(y^2)'{\cal P}_n(u)+c(y^4)'=
(y^2{\cal P}_n(u))'-y^2{\cal P}_n(u)'+c(y^4)',
\]
where $c\in \C$ is a constant of integration. Therefore, the existence of the desired $h_n$
will follow if we see that $y^2{\cal P}_n(u)'$ has a global primitive on
$\C /\langle i\rangle $ which is 
meromorphic. By the recurrence law (\ref{law}) for the operators ${\cal P}_n$,
rewritten using the function $y$ instead of $u$, we have
\[
{\cal P}_{n}'={\cal P}_{n-1}'''-\frac{4y''}{y}{\cal P}_{n-1}'
+2\frac{y'y''-yy'''}{y^2}{\cal P}_{n-1}.
\]
Hence, by direct computation,
\[
\begin{array}{l}
y^2{\cal P}_n'=
y^2{\cal P}_{n-1}'''-{4yy''}{}{\cal P}_{n-1}'+2(y'y''-yy'''){\cal P}_{n-1}  \\
\rule{0cm}{.5cm}\hspace{0.8cm}=\partial_z\left(y^2{\cal P}_{n-1}''
-2yy'{\cal P}_{n-1}'+2((y')^2-yy''){\cal P}_{n-1}\right),
\end{array}
\]
from where we conclude the existence of a global primitive of
$y^2{\cal P}_n(u)'$ in $\C /\langle i\rangle $. Furthermore, such a global primitive
is meromorphic on $\C /\langle i\rangle $ by the same property which holds for $y,u$ and ${\cal P}_{n-1}(u)$.
Therefore, the existence of $h_n$ is proved. Since $u,{\cal P}_n(u)$ are rational expressions of
$g$ and its derivatives, then the same holds for $h_n$
(in particular, $h$ is meromorphic on $\C /\langle i\rangle $).

In order to see that the meromorphic function
$\frac{\partial g}{\partial t_n}$ lies in $T_g{\cal W}$, we must check
that its principal divisor $D$ satisfies $D\geq \prod _jp_jq_j^{-3}$,
where $(g)=\prod _jp_j^2q_j^{-2}$ is the principal divisor of $g$ in
$\C /\langle i\rangle $. Since $\frac{\partial g}{\partial t_n}$ is
holomorphic outside from zeros and poles of~$g$, we only need to study
the behavior of $\frac{\partial g}{\partial t_n}$ around the points $p_j,q_j$.
\vspace{.2cm}
\par
\noindent
{\it Claim 1: $\frac{\partial g}{\partial t_n}$ has a zero at every zero $p_j$ of $g$.}
\vspace{.2cm}
\par
\noindent
{\it Proof of Claim 1.} We may assume $p_j=0$. Since
$g\in {\cal M}_{\mbox{\footnotesize imm}}$, the Weierstrass pair $(g,dz)$ closes periods
at $p_j$. Thus $g'''(0)=0$, which gives a series expansion
$g(z)=a z^2 + z^4 f_1(z)$, where $a, b, \ldots$ will denote complex
numbers and $f_1, f_2,\ldots$ will represent holomorphic functions
around $z=0$, during this proof and that of Claim 2 below. Using
(\ref{u}), we obtain
\begin{equation}
\label{uu}
u(z) =  -\frac{2}{z^2} + b + z^2 f_2(z).
\end{equation}
Assume we have proved that ${\cal P}_{n}(u)$ has an order-two pole without
residue at each zero of $g$, i.e.,
\begin{equation}
\label{jn}
{\cal P}_{n}(u)=\frac{c}{z^2}+f_3(z).
\end{equation}
Then we conclude that
\[
g(z){\cal P}_{n}(u)= ac +z^2f_4(z),
\]
which implies that $\frac{\partial g}{\partial
t_n}=-2\partial_z(g{\cal P}_{n}(u))$ has a zero at the origin, as we
wanted. It remains to check that ${\cal P}_{n}(u)$ has an order-two pole
without residue at the origin, which will be proved by induction.
Since ${\cal P}_{1}(u)=u$, the case $n=1$ follows from equation
(\ref{uu}). Assuming (\ref{jn}) we next study the Laurent expansion
for ${\cal P}_{n+1}(u)$ around $z=0$. The recurrence law
(\ref{law}), equations (\ref{uu}) and (\ref{jn}) and a direct
computation give
\[
\partial_z {\cal P}_{n+1}(u) = (
\partial_{zzz} + 4u\,\partial_z+2u'){\cal P}_{n}(u)=
\frac{d}{z^3}+\frac{e}{z}+f_5(z).
\]
As the left-hand-side has a well-defined primitive, we obtain $e=0$ and thus,
${\cal P}_{n+1}(u)$ has the correct behavior at the origin. Now Claim 1 is proved.
\vspace{.2cm}
\par
\noindent
{\it Claim 2: $\frac{\partial g}{\partial t_n}$ has at most an
order-three pole at every pole $q_j$ of $g$.}
\vspace{.2cm}
\par
\noindent
{\it Proof of Claim 2.} Again we can suppose $q_j=0$. First observe that,
as $g$ has a pole at $z=0$ without residue, then $u$ is holomorphic at $z=0$
(direct computation). Since ${\cal P}_1(u)=u$ and
$\partial_z {\cal P}_{n}(u) = (\partial_{zzz} + 4u\,\partial_z+2u'){\cal P}_{n-1}(u)$,
we deduce that $\partial_z {\cal P}_{n}(u) $ is holomorphic at $z=0$. It follows that
${\cal P}_n(u)$ is holomorphic at $z=0$ for all $n$. As $g$ has an order-two pole at $z=0$, we deduce that
$\frac{\partial g}{\partial t_n}=-2\partial _z(g{\cal P}_n(u))$ has at
most an order-three pole at $z=0$.
This completes the proofs of Claim 2 and of Theorem~\ref{thm9.2}.
\end{proof}

\begin{corollary}
\label{sag}
For every $g\in {\cal M}_{\mbox{\rm \footnotesize imm}}$, the function
$u =-\frac{3(g')^2}{4g^2}+\frac{g''}{2g}$
is an algebro-geometric potential of the KdV equation.
\end{corollary}
\begin{proof}
Using Theorems~\ref{thm9.2} and~\ref{bounded} in Appendix~1,
we deduce that there exists $n\in \N $
such that
the Jacobi function $f(h_n)\in {\cal J}_{\C }(g)$ associated to the
infinitesimal flow $\frac{\partial g}{\partial t_n}$ is a
linear combination of the Jacobi functions $f(h_0),\ldots ,f(h_{n-1})$ associated to
$\frac{\partial g}{\partial t_0},\ldots,\frac{\partial g}{\partial t_{n-1}}$, respectively.
Note that the linear map $h\mapsto f(h)$ given by equation~(\ref{eq:Jacobideh}) is injective. Therefore,
$h_n$ is a linear combination of $h_0,\ldots ,h_{n-1}$, and (\ref{eq:gpuntodeh}) implies that
the $n$-th infinitesimal flow $\frac{\partial g}{\partial t_n}$ of the
Shiffman hierarchy is a linear combination of
$\frac{\partial g}{\partial t_0},\ldots,\frac{\partial g}{\partial t_{n-1}}$.
By equations (\ref{relacion}) and (\ref{sh}), each of the infinitesimal flows
$\frac{\partial u}{\partial t_n}$ of the KdV hierarchy can be
expressed in terms of the ones of the Shiffman hierarchy as
\[
\frac{\partial u}{\partial t_n} =
\frac{\partial }{\partial t_n}\left(-\frac{3(g')^2}{4g^2}+\frac{g''}{2g} \right) ,
\]
from where we conclude that $\frac{\partial u}{\partial t_n}$ depends linearly
on the lower order infinitesimal flows in the KdV hierarchy.
\end{proof}

\begin{lemma}
\label{lemma6.11}
Let $u\colon \C\to \C \cup\{\infty\}$ be a meromorphic function
with Laurent expansion given by {\rm (\ref{eq:thm6.4*})} around
any of its poles, and let $y\colon \C\to \C \cup\{\infty\}$ be a meromorphic
solution of the equation $y''+uy=0$. Then, the following properties hold.
\begin{enumerate}
\item Outside of the poles of $u$, the function $y$ is holomorphic and its zeros
are simple.
\item At a pole of $u$, the function $y$ has either a simple pole or an order-two zero.
\end{enumerate}
\end{lemma}
\begin{proof}
First suppose that $y$ has a pole of order $k\geq 1$ at $z=0$. Then
locally $y(z)=z^{-k}f(z)$ with $f(0)\neq 0$, from where we conclude
that
\[
u(z)=-\frac{y''(z)}{y(z)}=
-\frac{k(k+1)}{z^2}+\frac{2kf'}{f}\frac{1}{z}+\mbox{ holomorphic}(z).
\]
This implies that every pole of $y$ is also a pole of $u$, which
is the first part of item~1. By equation~{\rm (\ref{eq:thm6.4*})}
we have $k(k+1)=2$; thus $k=1$ and so, all poles of $y$ are simple.

We now deal with the zeros of $y$. If $y$ has a zero at a point $a$
where $u$ is finite, then it must be a simple zero of $y$ (because
the solutions of $y''+uy=0$ are locally determined by
$(y(a),y'(a))$). Thus it suffices to study the behavior of $y$ at a
pole $z_0$ of $u$ such that $y$ is holomorphic around $z_0$. In this
case, we can write locally $y(z)=(z-z_0)^kf(z)$ for some non-negative
integer $k$ and some holomorphic function $f$ with $f(z_0)\neq 0$.
Then,
\[
u(z)=-\frac{y''(z)}{y(z)}=-\frac{k(k-1)}{(z-z_0)^2}
-\frac{2kf'(z)}{(z-z_0)f(z)}+\mbox{holomorphic}(z).
\]
Again equation~{\rm (\ref{eq:thm6.4*})} implies
$k(k-1)=2$; hence $k=2$ and the lemma is proved.
\end{proof}

\begin{definition}
\label{defholomorphicallyintegrated}
{\rm
Let $M$ be a quasiperiodic, immersed minimal surface of Riemann type,
with Weierstrass pair $(g,dz)$ on $(\C /\langle i\rangle )-g^{-1}(\{ 0,\infty \} )$.
Let $(g)=\prod _{j\in \Zsmall }p_j^2q_j^{-2}$ be the principal divisor of $g$ and
let $z_0\in \C /\langle i\rangle $ be a point different from $p_j$ and $q_j$ for all $j$.
The Shiffman function $S_M$ of $M$ is said to be {\it
holomorphically integrated} if there exist $\varepsilon >0$ and families
$\{ p_j(t)\} _j$, $\{q_j(t)\} _j\subset \C /\langle i\rangle $, $a(t)\in \C -\{ 0\} $  such that
\begin{description}
\item[{\it i)}] For each $j\in \Z $, the functions $t\in \D (\ve )\mapsto p_j(t)$,
$t\mapsto q_j(t)\in \C /\langle i\rangle $ are holomorphic with $p_j(0)=p_j$,
$q_j(0)=q_j$. Also, the function $t\mapsto a(t)$ is holomorphic as well.
\item[{\it ii)}] For any $t\in \D (\varepsilon )$, the divisor
$\prod _jp_j(t)^2q_j(t)^{-2}$ defines an element $g_t\in {\cal M}_{\mbox{\rm \footnotesize imm}}$
with $g_0=g$ and $g_t(z_0)=a(t)$. Let $M_t$ be the quasiperiodic, immersed minimal surface of Riemann
type with Weierstrass pair  $(g_t,dz)$.
\item[{\it iii)}] For $t\in \D (\ve )$, the derivative of $t\mapsto g_t$ with respect to $t$ equals
\[
\frac{d}{dt}g_t=\frac{i}{2}\left( g_t'''-3\frac{g_t'g_t''}{g_t}+
\frac{3}{2}\frac{(g_t')^3}{g_t^2}\right)  \quad \mbox{on }\C /\langle i\rangle .
\]
\end{description}
}
\end{definition}

\begin{remark}
\label{remark6.13}
{\rm
Definition~\ref{defholomorphicallyintegrated} is motivated by the following fact. With the notation in
that definition, suppose that the Shiffman function $S_M$ of $M$
can be holomorphically integrated. Consider the Weierstrass pair $(g_t,dz)$, $t\in \C $, $|t|<\ve $.
Fix a point $z_0\in (\C /\langle i\rangle )-g^{-1}(\{ 0,\infty \} )$.  Applying equation~(\ref{eq:propos4.13new}) to
$\Psi _t=\left( \frac{1}{2}(\frac{1}{g_t}-g_t),\frac{i}{2}(\frac{1}{g_t}+g_t),1\right) dz$
at every $t$, we obtain
\begin{equation}
\label{eq:SMholintegr}
\langle \left. \frac{d}{dt}\right| _t\int _{z_0}^z\Psi _t,N_t\rangle =
-\frac{1}{2}f(h_t)+\langle a(t),N_t\rangle ,
\end{equation}
where $f(h_t)=S_{M_t}+iS^*_{M_t}$ is the (complex valued) Jacobi function of $M_t$. Writing $t=t_1+t_2i$ with
$t_1,t_2\in \R $ and taking real parts in the last displayed equation at $t=t_1+0i$, we obtain
\[
\langle \left. \frac{\partial }{\partial t_1}\right| _{t_1}\Re \int _{z_0}^z\Psi _{t_1},N_{t_1}\rangle =
-\frac{1}{2}S_{M_t}+\langle \Re (a(t_1)),N_{t_1}\rangle .
\]
Calling $\psi _{t_1}=\Re \int _{z_0}^z\Psi _{t_1}-\int _0^{t_1}\Re (a(s))\, ds$, we have that the normal component
of the variational field of the deformation $t_1\mapsto \psi _{t_1}$ of $M$ equals (up to a multiplicative constant),
the (real valued) Shiffman function of $M_t$.
}
\end{remark}

\begin{theorem}
\label{thmintegrShiffman} Let $M=(\C /\langle i\rangle ,g,dz)$
be a quasiperiodic, immersed minimal surface of Riemann type {\rm
(}i.e., $g\in {\cal M}_{\mbox{\rm \footnotesize imm}}${\rm ).} Then, its
Shiffman function can be holomorphically integrated. Furthermore, if
$M$ is embedded, then its related surfaces $M_t$ are also
embedded for $|t|$ sufficiently small.
\end{theorem}
\begin{proof}
Choose a meromorphic function $y\colon \C \to \C \cup \{ \infty \} $ such
that $g(z)=y(z)^{-2}$ and let $u\colon \C \to \C \cup \{ \infty \} $ be given
by $u(z)=-y''(z)/y(z)$, which is also meromorphic.
Note that $y(z+i)=\pm y(z)$ and $u(z+i)=u(z)$ and that each of the three functions
$g,y$ and $u$ is quasiperiodic. By Corollary~\ref{sag}, $u$ is algebro-geometric.
A direct computation (using that the zeros and poles of $g$ have order two)
shows that around each pole $z_0$ of $u$, the Laurent expansion of $u$ is
\[
u(z)=\frac{-2}{(z-z_0)^2}+\mbox{\rm holomorphic}(z).
\]
Using Theorem \ref{integrar2}, we can solve the Cauchy problem
for the KdV equation with initial condition $u(z)$ and we get a meromorphic function
$u_t(z)=u(z,t)$, $z\in\C /\langle i\rangle $ and $|t|<\varepsilon$,  which is quasiperiodic
for each $t$. Moreover the poles of $u_t$ are given by holomorphic curves $t\mapsto z_0(t)$, and
\begin{equation}
\label{eq:thm7.5A}
u(z,t)=\frac{-2}{(z-z_0(t))^2}+\mbox{\rm holomorphic}(z,t)
\end{equation}
around $z_0(t)$. Consider the differential system in (\ref{PDEy}) for $n=1$ with unknown $y(z,t)$,
\begin{eqnarray}
\label{PDEy1}
\left\{ \begin{array}{l}
y''+u(z,t)y=0 \\
\rule{0cm}{.5cm}\frac{\partial y}{\partial t_1} =
{\cal P}_1(u)' y - 2 {\cal P}_1(u)y'= u'(z,t) y - 2 u(z,t) y'
\end{array}\right.
\end{eqnarray}
(The second equation in (\ref{PDEy1}) corresponds to the Shiffman flow in the
Schr\"{o}dinger hierarchy). The compatibility condition of (\ref{PDEy1}) is just
the KdV equation for $u$, see Appendix~A in Joshi~\cite{joshi1}.
By the Frobenius theorem, (\ref{PDEy1}) admits a unique solution $y=y(z,t)$
with initial condition $y(z,0)=y(z)$. Since $z\mapsto u(z,t)$ is algebro-geometric
for every $t$, part~3 of Theorem~\ref{sw} implies that $y(z,t)$ is defined on
$\C \times \D(\varepsilon)$ (for some $\ve >0$) and is meromorphic in $z$.
The uniqueness of solution of an initial value problem together with the fact that
$y(z+i)=\pm y(z)$, give that $y(z+i,t)=\pm y(z,t)$, with the same choice of signs as for $y(z)$.

By Lemma~\ref{lemma6.11} applied to $u_t(z)=u(z,t)$ and $y_t(z)=
y(z,t)$, we find that $y_t$ is holomorphic with simple zeros
outside of the poles of $u_t$, and that at a pole of $u_t$, either
$y_t$ has a simple pole or $y_t$ has an order-two zero. We claim that
this last possibility cannot occur. To see the claim,
let $D$ be a closed disk centered at a pole $z_0=z_0(0)$ of $u(z)$
and let $\ve >0$ such that $z_0(t)\in \mbox{Int}(D)$ is the unique pole of $u_t(z)$
in $D$ whenever $|t|<\ve $, see the proof of Theorem~\ref{integrar2} for details.
Note that since $g(z)$ has order-two zeros and poles, then the zeros and poles of
$y(z)$ are simple. Since $u(z)$ has a pole at $z_0\in D$,
then Lemma~\ref{lemma6.11} demonstrates that $y(z)$ has a simple pole at $z_0$.
We can also assume without loss of generality that $y(z)$ has no other zeros or poles
in $D$ and, by continuity, $y_t(z)$ has no zeros or poles in $\partial D$ for $t$
sufficiently close to zero. Arguing by contradiction, assume there exists $\widehat{t}$ with $|\widehat{t}|<\ve $,
such that $y_{\widehat{t}}$ has an order-two zero at $z_0(\widehat{t})$. Then,
\begin{equation}
\label{eq:thm7.5B}
\# (y_t^{-1}(\infty )\cap D)-\# (y_t^{-1}(0)\cap D)=
\# (y^{-1}(\infty )\cap D)-\# (y^{-1}(0)\cap D).
\end{equation}
The right-hand-side of (\ref{eq:thm7.5B}) is $1$, while the left-hand-side
for $t=\widehat{t}$ equals $-2$. This contradiction proves our claim.

Since the zeros $p_j(t)$ and poles $q_j(t)$ of $y_t(z)$ are simple, Assertion~\ref{ass6.5}
insures that $p_j(t),q_j(t)$ depend holomorphically on $t$.
Furthermore, the fact that $y(z,0)=y(z)$ implies that $p_j(0)=p_j$, $q_j(0)=q_j$,
where $(y)=\prod _jp_jq_j^{-1}$ is the principal divisor of $y(z)$. The same arguments
in the proof of Theorem~\ref{integrar2} now give that $y_t(z)$ is quasiperiodic. Finally,
define the quasiperiodic meromorphic function
\[
g(z,t)=\frac{1}{y^2(z,t)}, \qquad (z,t)\in (\C /\langle i\rangle ) \times \D (\ve ).
\]
As a function of $z$, $g_t(z)=g(z,t)$ has only order-two zeros and poles, and
$t\mapsto g_t$ is a holomorphic curve in ${\cal W}$. The complex periods of the Weierstrass pair
$((\C /\langle i\rangle )-g_t^{-1}(\{ 0,\infty \} ),g_t,dz)$ along every closed curve $\G \subset \C /\langle i\rangle $
are constant in $t$ provided that we prove that the following integrals vanish:
\begin{equation}
\label{eq:thm7.5C}
\frac{d}{dt}\int _{\G }\frac{dz}{g_t}=
-\int _{\G }\frac{\frac{\partial g_t}{\partial t}}{g_t^2}dz,
\qquad
\frac{d}{dt}\int _{\G } g_t\, dz=
\int _{\G }\frac{\partial g_t}{\partial t}\, dz.
\end{equation}
By Assertion~\ref{ass9.3} together with
equation~(\ref{eq:periodconstant}), both integrals in
(\ref{eq:thm7.5C}) vanish if we check that $\frac{\partial
g}{\partial t}=\frac{\partial g}{\partial t_1}$, where the right-hand-side in the
last equation is the flow of the Shiffman hierarchy for $n=1$. Also note that
once we know that the complex periods of $(g_t,dz)$ on $(\C /\langle i\rangle)-g_t^{-1}(\{ 0,\infty \} )$
do not depend on $t$, we can easily deduce that this pair is the
Weierstrass data of a quasiperiodic, immersed minimal surface of Riemann type
$M_t \subset \R^3$.

Next we prove that $\frac{\partial g}{\partial t}=\frac{\partial
g}{\partial t_1}$. Since $y''_t+u_ty_t=0$
 and $g_t=y_t^{-2}$, we have $u_t=-\frac{3(g'_t)^2}{4g_t^2}
+\frac{g''_t}{2g_t}$, which is equation (\ref{u}) for time $t$.
Using that $y(z,t)$ satisfies (\ref{PDEy1}) and comparing with
(\ref{PDEy}) and (\ref{sh}), we deduce the desired equality. Note
that $\frac{\partial g}{\partial t_1}$ is a constant multiple of the
(complex) Shiffman Jacobi function. Hence, we have proved that the
Shiffman function can be holomorphically integrated on $M$.

Finally, the fact that $M_t$ is embedded for $|t|$
sufficiently small provided that $M$ is embedded, follows from the previous arguments
together with the maximum principle for minimal surfaces.
\end{proof}

\section{The proofs of Theorems~\ref{mainthm} and~\ref{classthm}.}
\label{per}

The goal of this section is to prove Assertion~\ref{ass} stated in the introduction.
Recall from Section~\ref{sec3} that Assertion~\ref{ass} implies our main
Theorem~\ref{mainthm} and its consequence, Theorem~\ref{classthm}. Our strategy to prove
Assertion~\ref{ass} is as follows (we follow the notation in that assertion). First,
we will prove in Proposition~\ref{propos7.3}
that if the Shiffman function $S_M$ of a surface $M\in {\cal M}$ is linear,
then $M$ is a Riemann minimal example.
Second, we show that for every surface $M\in {\cal M}$, its Shiffman function is linear
(item 2 of Proposition~\ref{propos7.6}).

\subsection{Minimal surfaces of Riemann type whose Shiffman function is linear.}

\begin{lemma}
\label{lema7.2}
Suppose that the Shiffman function $S_M$ of a quasiperiodic, immersed minimal surface
of Riemann type $M\subset \R^3$, is linear. Then, $M$ is singly-periodic, and its smallest
orientable quotient surface $M_1$ is a torus punctured in two points with total curvature $-8\pi $.
Furthermore, $M_1$ is properly and minimally immersed in $\R^3/\langle v\rangle$, where $v\in \R^3-\{ 0\} $ is
a translation vector of $M$, and the punctures of $M_1$ correspond to planar ends of $M$.
\end{lemma}
\begin{proof}
Let $N$ be the Gauss map of $M$. Since $S_M\in L(N)$, its conjugate Jacobi function
$S_M^*$ is also linear; thus $S_M+iS_M^*\in L_{\Csmall }(N)$. By equations
(\ref{eq:L(N)}) and (\ref{eq:uShiffman}), this linearity of $S_M+iS_M^*$ implies that there exists
$a\in \C ^3$ such that
\begin{equation}
\label{eq:lema7.2A}
\frac{3}{2}\left( \frac{g'}{g}\right) ^2-\frac{g''}{g}
 -\frac{1}{1+|g|^2}\left( \frac{g'}{g}\right) ^2=\langle N,a\rangle .
\end{equation}
After writing $a=(a_1,a_2,a_3)$, $2a_1=A+B$, $2a_2=i(A-B)$ and plugging the
equation~(\ref{eq:basislinearg}) into (\ref{eq:lema7.2A}), we obtain the following ODE for $g$:
\[
\frac{3}{2}\left( \frac{g'}{g}\right) ^2-\frac{g''}{g}-\frac{1}{|g|^2+1}\left(
\frac{g'}{g}\right) ^2=
\frac{1}{|g|^2+1}\left( Ag+B\overline{g}\right) +a_3 \frac{|g|^2-1}{|g|^2+1}.
\]
An algebraic manipulation in the last expression leads to
\[
\overline{g}\left( \frac{3}{2}\frac{(g')^2}{g}-g''-B-a_3g\right) =
\frac{g''}{g}-\frac{1}{2}\left( \frac{g'}{g}\right) ^2+Ag-a_3.
\]
Since $g$ is holomorphic and not constant, we deduce that
\[
\frac{3}{2}\frac{(g')^2}{g}-g''-B-a_3g=0,\qquad
\frac{g''}{g}-\frac{1}{2}\left( \frac{g'}{g}\right) ^2+Ag-a_3=0.
\]
After elimination of $g''$ in both equations, we arrive at $(g')^2=g(-A g^2+2a_3 g+B)$.
Hence we have a (possibly branched) holomorphic covering $\pi =(g,g')$ from the cylinder
$M\cup \{ \mbox{planar ends}\} \equiv \C /\langle i\rangle $ onto the compact Riemann surface
$\Sigma =\{ (\xi ,w)\in (\C\cup \{ \infty \} )^2\ | \ w^2=\xi(-A\xi^2+2a_3\xi+
B)\} $. Clearly, $\Sigma $ is either a sphere or a torus. We claim that $\Sigma $ cannot be a sphere. Otherwise,
consider the meromorphic differential $\frac{d\xi }{w}$ on $\Sigma $, whose pullback by $\pi $ is
$\pi ^*(\frac{d\xi }{w})=\frac{dg}{g'}=dz$. Given a pole $P\in \Sigma $ of $\frac{d\xi }{w}$, choose a point
$z_0\in \C /\langle i\rangle $ such that $\pi (z_0)=P$. The residue of
$\frac{d\xi }{w}$ at $P$ can be computed as the integral of $\frac{d\xi }{w}$ along a small closed curve
$\G _P\subset \Sigma $ which winds once around $P$. After lifting $\G _P$ through $\pi $ locally around $z_0$ we
obtain a closed curve $\widetilde{\G }_P\subset \C /\langle i\rangle $ which winds a positive integer number
of times  around $z_0$, depending on the branching order of $\pi $ at $z_0$. Hence the residue of $\frac{d\xi }{w}$ at
$P$ equals a positive integer multiple of the residue of $dz$ at $z_0$, which is zero. Therefore, $\frac{d\xi }{w}$
has residue zero at all its poles, and so, it is exact on $\Sigma $. This implies that
$dz$ is also exact on $\C /\langle i\rangle $, which is a contradiction. Thus, $\Sigma $ is a torus.
Now consider the following
Weierstrass pair on $\Sigma $:
\[
\left( g_1(\xi ,w)=\xi ,dh_1=\frac{d\xi }{w}\right) .
\]
The associated metric to this pair is $\left( \frac{1}{2}(|\xi |+|\xi |^{-1})\frac{|d\xi |}{|w|}\right) ^2$,
which can be easily proven to be positive definite and complete
in $\Sigma -\{ (0,0),(\infty ,\infty )\} $. Note that $g_1\circ \pi =g$ and
$\pi ^*(\frac{d\xi }{w})=dz$. This implies that the Weierstrass pair $(g,dz)$ of $M$ can be induced on the twice punctured torus
$\Sigma -\{ (0,0),(\infty ,\infty )\} $. From here one easily deduces that $M$ is singly-periodic. Note that
since the degree of the extended Gauss map $g_1$ on $\Sigma $ is 2, then the total curvature of the quotient minimal
surface is $-8\pi $.
\end{proof}


\begin{proposition}
\label{propos7.3}
If the Shiffman function $S_M$ of an embedded surface $M\in {\cal M}$ is linear,
then $M$ is a Riemann minimal example.
\end{proposition}
\begin{proof}
Let ${\cal M}_1\subset {\cal M}$ be the subset of surfaces which are singly-periodic and
their smallest orientable quotient is a properly embedded, twice-punctured minimal torus in a quotient
of $\R^3$ by a translation. By Lemma~\ref{lema7.2}, our proposition reduces to proving that
${\cal M}_1$ coincides with the family ${\cal R}=\{ {\cal R}_t\} _t$
of Riemann minimal examples. This result is implied by the main theorem in~\cite{mpr1}, see also
Appendix 2 for a self-contained proof. Up to this reduction, the proposition is proved.
\end{proof}

\subsection{The linearity of the Shiffman function for every surface $M\in {\cal M}$.}
\label{lin} Recall that ${\cal M}$ is the space of properly
embedded, minimal planar domains $M\subset \R^3$ with two limit ends
and flux $F=F_M=(h,0,1)$, $h=h(M)>0$. ${\cal M}$ is endowed with
its natural topology of uniform convergence on compact sets.
According to
the notation in Theorem~\ref{thmcurvestim}, the
heights of the planar ends of every $M\in {\cal M}$ are
\begin{equation}
\label{eq:heightends}
\ldots <\Re (p_{-1})<\Re (q_{-1})<\Re (p_0)<\Re (q_0)<\Re (p_1)<\Re (q_1)<\ldots
\end{equation}
Recall from Section~\ref{sec4.2} that we view
${\cal M}$ as a subset of ${\cal W}$ by the map $M\mapsto g$ that associates to
each $M\in {\cal M}$ its meromorphic Gauss map $g\in {\cal W}$. For this
inclusion to make sense we must identify surfaces in ${\cal M}$ up to horizontal
translations in $\R^3$ (two elements in ${\cal W}$ that differ in a vertical translation
are considered as different elements in ${\cal W}$, and the same holds for two surfaces
in ${\cal M}$ that differ in a vertical translation). Indeed,
${\cal M}\subset {\cal M}_{\mbox{\footnotesize imm}}\subset {\cal W}$, where
${\cal M}_{\mbox{\footnotesize imm}}$ is defined by (\ref{eq:Mimm}). This identification
of ${\cal M}$ as a subset of ${\cal W}$ is consistent with the topology of both sets
(${\cal W}$ was equipped with the uniform convergence of compact sets of $\C /\langle i\rangle $),
because the convergence of surfaces in ${\cal M} $ produces convergence of the corresponding
Gauss maps in ${\cal W}$ (for this property one needs to extend uniform convergence across the planar ends
$g^{-1}(\{ 0\})=\{ p_j\} _j$, $g^{-1}(\{ \infty \})=\{ q_j\} _j$).
Also recall that the topology of ${\cal W}$ is equivalent to the product topology on
$\left[ \Pi_{ j\in\mathbb{Z}}\, (\mathbb{C}/\langle i \rangle)\right]
\times (\C -\{ 0\} )$ by the bijection $g\mapsto (p_j,q_j,g(z_0))$ defined in (\ref{eq:topol}).

Next we want to define the functions which map each surface $M\in {\cal M}$ to the relative height
of each of its planar ends with respect to the end corresponding to $p_0$.
We define the positive functions $h_j\colon {\cal M}\to \R $ for $j\in \N$ by
\[
h_1(M)=\Re (q_0-p_0),\quad h_2(M)=\Re (p_0-q_{-1}),\quad h_3(M)=\Re (p_1-p_0),\quad h_4(M)=\Re (p_0-p_{-1})\ldots
\]
\begin{proposition}
\label{propos7.6}
Given $F=(h,0,1)$ with $h>0$, let ${\cal M}_F=\{ M\in {\cal M}\ | \ F_M=F\} $.
Then:
\begin{enumerate}
\item There exists a surface $M_{\max }\in {\cal M}_F$ which  maximizes
each of the functions $h_{j+1}$ in ${\cal M}_F(j)=\{ M\in {\cal M}_F(j-1)\ | \ h_j(M)=\max _{{\cal M}_F(j-1)}h_j\} $ for all $j\geq 1$, where ${\cal M}_F(0)={\cal M}_F$. Also, there exists
a surface $M_{\min }\in {\cal M}_F$ which  minimizes
each $h_{j+1}$ in $\widetilde{\cal M}_F(j)=\{ M\in \widetilde{\cal M}_F(j-1)\ | \ h_j(M)=
\min _{\widetilde{\cal M}_F(j-1)}h_j\} $ for all $j\geq 1$. Furthermore,
the Shiffman function of every such surface $M_{\min }, M_{\max }$ is linear.
\item The Shiffman function of every surface $M\in {\cal M}$ is linear.
\end{enumerate}
\end{proposition}
\begin{proof}
By the uniform curvature estimates in
Theorem~\ref{thmcurvestim} and subsequent uniform local area estimates,
we deduce that ${\cal M}_F$ is compact where ${\cal M}_F$ is
considered as a topological subspace of $\left[ \Pi_{ j\in\mathbb{Z}}\, (\mathbb{C}/\langle i \rangle)\right]
\times (\C -\{ 0\} )$ with its metrizable product topology.
Note that $h_j$ is continuous. Thus, there exists a maximum of $h_1$ in
${\cal M}_F$. Now consider the restriction of $h_2$ to the non-empty
subset ${\cal M}_F(1)= \{ M\in {\cal M}_F \ | \ h_1(M)=\max _{{\cal
M}_F}h_1\} $. As before, we can maximize $h_2$ on ${\cal M}_F(1)$,
which implies that the space ${\cal M}_F(2)$ defined in the statement of the proposition is non-empty,
and maximize $h_3$ in ${\cal M}_F(2)$. Repeating the argument, induction
lets us maximize $h_{j+1}$ in ${\cal M}_F(j)\neq \mbox{\O}$ for each $j\in \N $.
Since the compact subsets ${\cal M}_F(j)$ satisfy ${\cal
M}_F(j)\supset {\cal M}_F(j+1)$ for all $j$, this collection of
closed sets satisfies the finite intersection property. By the
compactness of ${\cal M}_F$,  we conclude that $\bigcap _{j\in \N}
{\cal M}_F(j)\neq \mbox{\O }$. Thus there exists a surface $M_{\max
}\in {\cal M}_F$ that maximizes each of the functions $h_{j+1}$ in
${\cal M}_F(j)$ for all $j\geq 1$. In the same way, we find a
surface $M_{\min }\in {\cal M}_F$ that minimizes the function
$h_{j+1}$ on $\widetilde{\cal M}_F(j)$ for all $j\geq 1$. This proves the
first statement of item~1.

Next we prove that if $M_0\in {\cal M}_F$ maximizes each of the functions $h_{j+1}$ in
${\cal M}_F(j)$ for all $j\geq 1$, then its Shiffman function is linear (for
minimizing surfaces the argument is similar). By
Theorem~\ref{thmintegrShiffman}, the Shiffman function $S_{M_0}$ of
$M_0$ can be holomorphically integrated. Thus we find
a curve $t\in \D (\ve )\mapsto g_t\in {\cal M}_{\mbox{\rm \footnotesize imm}}
\subset {\cal W}$ whose
zeros  $p_j(t)$ and poles $q_j(t)$ depend holomorphically on $t$,
satisfying items i), ii) and iii) of
Definition~\ref{defholomorphicallyintegrated} for $M=M_0$. With the
notation of that definition, let $\psi _t\colon (\C /\langle
i\rangle )-\{ p_j(t),q_j(t)\} _j\to \R^3$ be the parametrization of
$M_t$ given by
\[
\psi _t(z)=\Re \int _{z_0}^z\left( \frac{1}{2}(\frac{1}{g_t}-g _t), \frac{i}{2}(\frac{1}{g_t}+g_t),1\right)
\, dz,
\]
where $z_0$ has been chosen in $(\C /\langle i\rangle )-\{ p_j(t),q_j(t)\ | \ j\in \Z, |t|<\ve \} $.
By equation~(\ref{eq:SMholintegr}),
the normal part of the variational field of $t\in \D (\ve )\mapsto \psi _t$ is
(up to a multiplicative constant) the complex
valued Shiffman function $S_{M_t}+iS_{M_t}^*$ of $M_t$  plus
a linear function of the Gauss map $N_t$ of $M_t$,
we conclude from Remark~\ref{rem4.4}
that $M_t\in {\cal M}_F$ for all $t$. Therefore,
the harmonic function $t\in \D (\ve  )\mapsto h_1(M_t)=\Re (q_0(t)-p_0(t))$ attains a
maximum at $t=0$, and hence it is constant. From here we conclude that
the holomorphic function $t\in \D(\ve )\mapsto q_0(t)-p_0(t)$ is also constant.
The same argument applies to each function $t\mapsto h_j(M_t)$, concluding that for any $t$, all the planar ends
$p_j(t),q_j(t)$ of $M_t$ are placed at
\[
p_j(t)=p_0(t)+p_j-p_0,\quad q_j(t)=p_0(t)+q_j-p_0.
\]
Geometrically, this means that the maps $\psi _t$ coincide with $\psi _0$ up to translations in the parameter domain and
in $\R^3$. Therefore, the normal part of the variational field of $t\mapsto \psi _t$ is linear, as desired.
This proves the second statement in item~1 of the proposition.

Finally, we prove item 2 of the proposition. Given $M\in {\cal M}$, let
$F=(h,0,1)$ be the flux vector of $M$. By item~1, there exist surfaces $M_{\max },
M_{\min }\in {\cal M}_F$ such that $M_{\max }$ maximizes
each of the functions $h_{j+1}$ in ${\cal M}_F(j)$ (resp. $M_{\min }$ minimizes
$h_{j+1}$ in $\widetilde{\cal M}_F(j)$), for all $j\geq 1$. Furthermore, the Shiffman functions of $M_{\max },
M_{\min }$ are linear. By Proposition~\ref{propos7.3}, both $M_{\max },M_{\min }$ are Riemann
minimal examples. Since the flux parameterizes the space of Riemann minimal
examples, it follows that the Riemann minimal example ${\cal R}_F$ in ${\cal M}_F$
is unique up to translation. Thus, $M_{\max }$ and $M_{\min }$ are translations of
${\cal R}_F$. On the other hand,
the vertical distance between the ends $p_0,q_0$ of $M$ (with the notation in (\ref{eq:heightends}))
is bounded above (resp. by below) by the distance between the corresponding ends of $M_{\max }$
(resp. of $M_{\min }$). Therefore, the vertical distance between the ends $p_0,q_0$ of $M$
is maximal, or equivalently, $M$ maximizes $h_1$ on ${\cal M}_F$.
Analogously, $M$ maximizes each of the functions $h_{j+1}$ in
${\cal M}_F(j)$ for all $j\geq 1$ and applying item~1,
its Shiffman function $S_M$ is linear. Now the proof of the proposition is complete.
\end{proof}

\section{Linearity of bounded Jacobi functions on Riemann minimal examples.}
\label{3dim}
We devote this section to describing the set of bounded Jacobi functions
on every Riemann minimal example, which is the goal of Theorem~\ref{teorJacobiacotadaenRiemanneslineal}
below. This result plays a central role in our proof in Section~\ref{asymp} that any
limit end of a properly embedded minimal surface with finite genus and
horizontal limit tangent plane at infinity converges
exponentially in height to a limit end of one of the Riemann minimal examples.

\begin{theorem}
\label{teorJacobiacotadaenRiemanneslineal}
Let ${\cal R}={\cal R}_t
\subset \R^3$ be a Riemann minimal example. Then, any bounded Jacobi
function on ${\cal R}$ is linear.
\end{theorem}
\begin{proof}
We first homothetically rescale ${\cal R}$ so that its middle ends are placed at integer heights.
Let $N$ be the Gauss map of ${\cal R}$ and $\G $ the zero set of the linear Jacobi function $\langle
N,e_2\rangle $, where $e_2=(0,1,0)$. $\G $ consists of the
horizontal straight lines in ${\cal R}$ plus the intersection of
${\cal R}$ with the $(x_1,x_3)$-plane (of reflective symmetry).
Viewing the parameter domain as a cylinder $\esf ^1\times \R $
punctured at integer heights, the reflection in the $(x_1,x_3)$-plane produces
a reflection symmetry of the cylinder by a plane passing through its axis, and
$\G $ is represented in the cylinder by the wider circles and straight lines in Figure~\ref{fig3} left.

\begin{figure}
\begin{center}
  \includegraphics[width=11.1cm]{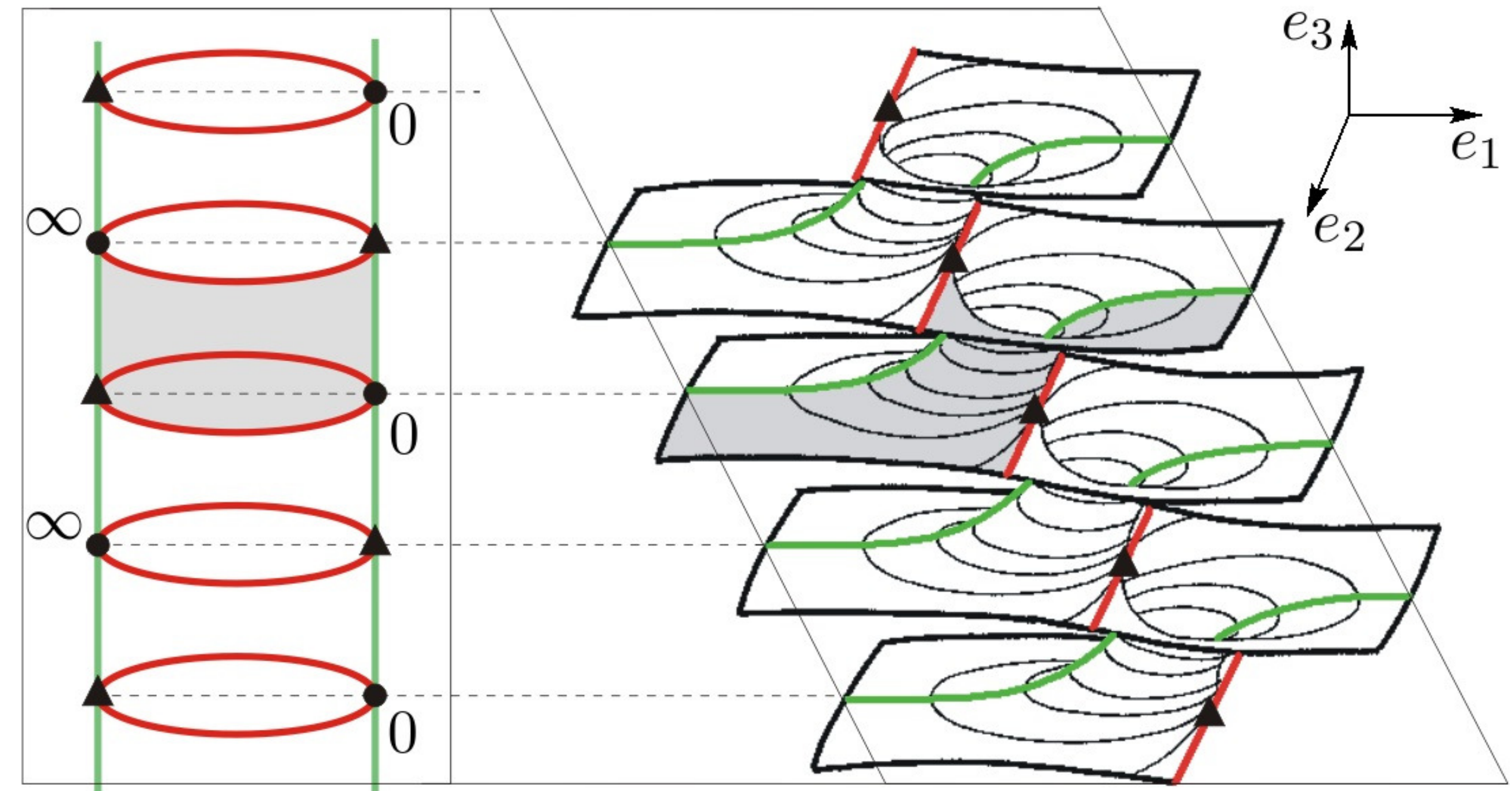}
  \caption{Left: The conformal model $\esf ^1\times \R $. Right: A Riemann minimal example.
Triangles denote finite branch points, and dots denote ends. The planar curves of symmetry together with the
straight lines form the zero set of $\langle N,e_2\rangle $, and we have shadowed one of its nodal domains.}
\label{fig3}
\end{center}
\end{figure}
$\G $ divides ${\cal R}$ into infinitely many
components $\Omega _i$, $i\in \Z $, to which we call
{\it nodal domains} (one of these nodal domains is shaded in Figure~\ref{fig3}).
The branch values of $N$ lie in the great circle $\esf ^2\cap \{ x_2=0\} $.
Thus $N$ restricts to $\Omega _i$ as a biholomorphism onto one of the open hemispheres
in which $\{ x_2=0\} $ divides $\esf^2$, for each $i\in \Z$. Furthermore,
such a biholomorphism has continuous extension to
the boundaries of these domains. Since the induced metric on $\Omega _i$ by the inner product of $\R^3$ and
the spherical metric are conformally related by $N$, we can
identify $\Omega _i$ with the hemisphere $N(\Omega _i)$ and express the
stability form associated to the Jacobi operator as the quadratic form
$\int _{\Omega _i}(|\nabla w|^2-2w^2)\, dA$ for any function in the Sobolev
space $W^{1,2}(\Omega _i)$ with the spherical metric.

\begin{assertion}
\label{ass8.2}
Let $v\in W^{1,2}(\Omega )$ be a bounded solution of $\Delta v+2v=0$ on a hemisphere $\Omega \subset \esf^2$.
Then:
\[
\int _{\partial \Omega }v\frac{\partial v}{\partial \eta }\, ds \geq 0,
\]
with equality if and only if $v$ linear, i.e., $v(x)=\langle x,a\rangle $ for some $a\in \R^3$ (here $\eta $
stands for the exterior conormal unit vector to $\Omega $ along its boundary).
\end{assertion}
{\it Proof of Assertion~\ref{ass8.2}.}
Clearly we can assume that $\Omega $ is one of the two hemispheres in $\esf^2-\{ x_2=0\} $.
Recall that the Neumann problem for the spherical Laplacian on $\Omega $ has first eigenvalue~$0$
(whose eigenfunctions are constant) and second eigenvalue $2$, with eigenfunctions being of the type
$x\in \Omega \mapsto \langle x,b\rangle $ with $b\in \R^3$ orthogonal to $e_2$. In particular, for every function
$w\in W^{1,2}(\Omega )$ with $\int _{\Omega}w\, dA=0$ we have:
\begin{equation}
\label{eq:JacobiacotadaenRiemanneslineal1}
\int _{\Omega }(|\nabla w| ^2-2w^2)\, dA\geq 0,
\end{equation}
and equality holds if and only if $w(x)=\langle x,b\rangle $ for any $b$ orthogonal to $e_2$.

Let $\varphi $ be the restriction of $x_2=\langle x,e_2\rangle $ to $\esf^2$.
The function $\varphi $ has constant non-zero sign on $\Omega $ and we can consider, for
any bounded function $v\in W^{1,2}(\Omega )$, the real number $c=\frac{\int _{\Omega }v\, dA}
{\int _{\Omega }\varphi \, dA}$. After applying (\ref{eq:JacobiacotadaenRiemanneslineal1}) to
$w=v-c\varphi $ on $\Omega $, we obtain
\eqnarray 0&\leq &{\displaystyle
\int _{\Omega }\left[ | \nabla (v-c\varphi )|^2-2(v-c\varphi )^2\right] dA}\nonumber \\
&=&{\displaystyle \int _{\Omega }\left( | \nabla v|^2-2v^2\right) dA
+c^2\int _{\Omega }\left( | \nabla \varphi |^2-2\varphi ^2\right) dA
-2c\int _{\Omega }\left( \langle \nabla v,\nabla \varphi \rangle
-2v\varphi \right) dA.\hspace{.9cm}}
\label{eq:JacobiacotadaenRiemanneslineal2}
\endeqnarray
Since $\varphi =0$ on $\partial \Omega $, integration by parts gives
\[
\int _{\Omega }\left( \langle \nabla v,\nabla \varphi \rangle -2v\varphi \right) dA
= -\int _{\Omega }\varphi \left( \Delta v+2v\right) dA=0.
\]
In the same way, the second integral in (\ref{eq:JacobiacotadaenRiemanneslineal2}) also
vanishes and the first integral in (\ref{eq:JacobiacotadaenRiemanneslineal2}) is
\begin{equation}
\label{eq:linearRiemannA}
\int _{\Omega }\left( | \nabla v|^2-2v^2\right) dA
=\int _{\Omega }\left[ \mbox{div}(v\nabla v)-v(\Delta v+2v)\right] dA
=\int _{\partial \Omega }v\frac{\partial v}{\partial \eta }\, ds,
\end{equation}
which gives the inequality in Assertion~\ref{ass8.2}. If equality
holds, then (\ref{eq:JacobiacotadaenRiemanneslineal1}) implies that
$v-c\varphi $ is linear on $\Omega $. Hence, $v$ is linear  in $\Omega $,
which proves the necessary condition in Assertion~\ref{ass8.2}. The proof of the sufficient
condition is straightforward.
\vspace{.2cm}
\par
\noindent
We next continue with the proof of Theorem~\ref{teorJacobiacotadaenRiemanneslineal}.
Let $T=(T_1,0,2)\in \R^3-\{ 0\} $ be the smallest orientation preserving
translation vector of ${\cal R}$ (recall that the planar ends are
placed at integer heights). Take a bounded Jacobi function $v$ on
${\cal R}$, and we will prove that $v$ is linear.
For $j\in \Z $ fixed, denote by $v_j$ the function on ${\cal R}$ given by
\[
v_j(p)=v(p+jT), \quad p\in {\cal R}.
\]
Since $\{ v_j\} _j$ is a sequence of bounded Jacobi functions on ${\cal R}$,
after extracting a subsequence we can assume that as $j\to +\infty $, the sequence
$\{ v_j\} _j$ converges 
 smoothly to a bounded Jacobi function
$v_{\infty }\colon {\cal R}\to \R $ on compact subsets of ${\cal R}$,
see for instance~\cite{gt1}. In fact, both $v_j,v_{\infty }$ extend smoothly through
each planar end of ${\cal R}$ and the smooth convergence $\{ v_j\} _j\to v_{\infty }$
extends to the planar ends as well.
\begin{assertion}
\label{ass8.3}
$v_{\infty }$ is linear on ${\cal R}$.
\end{assertion}
{\it Proof of Assertion~\ref{ass8.3}.}
By the convergence above, there exists a sequence of straight lines $ r_k={\cal R}\cap \{ x_3=j_k\} $
with $j_k\in \N $, $j_k\nearrow +\infty $, such that the values of $v$ and its derivatives along $r_k$
converge uniformly as $k\to +\infty $ to the values of $v_{\infty }$ and its derivatives along
the straight line $r={\cal R}\cap \{ x_3=0\} $. By Assertion~\ref{ass8.2} and equation (\ref{eq:linearRiemannA}),
we have
\begin{equation}
\label{eq:linearRiemannB}
0\leq \sum _j\int _{\Omega _j}(|\nabla v|^2-2v^2)\, dA=\int _{r_{k+1}}v\frac{\partial v}{\partial \eta }\, ds
-\int _{r_k}v\frac{\partial v}{\partial \eta }\, ds,
\end{equation}
where the sum runs in those nodal domains $\Omega _j$ between the heights of $r_k$ and $r_{k+1}$.
Taking $k\to +\infty $ in (\ref{eq:linearRiemannA}), we deduce that
\begin{equation}
\label{eq:linearRiemannC}
\lim _{j\to +\infty } \int _{\Omega _j}(|\nabla v|^2-2v^2)=0
\end{equation}
for every sequence of nodal domains $\Omega _j$ with heights going to $+\infty $.
Finally, the convergence $\{ v_j\} _j\to v_{\infty }$ and (\ref{eq:linearRiemannC}) together with
Assertion~\ref{ass8.2} imply that $v_{\infty }$ is linear on any nodal domain of ${\cal R}$.
By analyticity, $v_{\infty }$ is linear on ${\cal R}$ and Assertion~\ref{ass8.3} is proved.

The argument above can be repeated when $j\to -\infty $, from where we deduce that after passing to a subsequence,
$v_j$ converges as $j\to -\infty $ to a linear function $v_{-\infty }$ on ${\cal R}$
(possibly distinct of $v_{\infty }$).
Furthermore, there exists a sequence of straight lines $ r_k={\cal R}\cap \{ x_3=j_k\} $
with $k\in \Z \mapsto j_k\in \Z $ increasing, such that the values of $v$ and its derivatives along $r_k$
converge uniformly as $k\to +\infty $ (resp. as $k\to -\infty $) to the values of $v_{\infty }$
(resp. of $v_{-\infty }$) and its derivatives along the straight line $r={\cal R}\cap \{ x_3=0\} $.

Next consider the piece of ${\cal R}$ bounded by the straight lines $r_k,r_{-k}$ with $k\in \N$.
The same arguments above demonstrate that for the nodal domains $\Omega _j$ between the heights of
$r_k$ and $r_{-k}$, we have
\[
0\leq \sum _j\int _{\Omega _j}(|\nabla v|^2-2v^2)\, dA=\int _{r_{k}}v\frac{\partial v}{\partial \eta }\, ds
-\int _{r_{-k}}v\frac{\partial v}{\partial \eta }\, ds
\]
\[
\stackrel{(k\to +\infty )}{\longrightarrow }
\int _{r}v_{\infty }\frac{\partial v_{\infty }}{\partial \eta }\, ds
-\int _{r}v_{-\infty }\frac{\partial v_{-\infty }}{\partial \eta }\, ds.
\]
Thus it only remains to check that if $w$ is a linear function on ${\cal R}$, then
\begin{equation}
\label{eq:linearRiemannD}
\int _{r}w\frac{\partial w}{\partial \eta }\, ds
=0.
\end{equation}
To prove (\ref{eq:linearRiemannD}) we again use the spherical geometry: the straight line
$r$ corresponds via the Gauss map $N$
of ${\cal R}$ to a twice covered geodesic arc $\g \subset \esf^2\cap \{ x_2=0\} $ starting at the north or south pole
and ending at a non-vertical branch value of $N$ (when we view $r$ in $\R^3$,
this description corresponds to traveling along $r$ from one of its
ends to the finite branch point $P$ of $N$ along $r$, and from $P$ to the other end of $r$), and $w$ corresponds
to the height function $w(x)=\langle x,a\rangle $ for certain $a\in \R^3$. Then (\ref{eq:linearRiemannD}) holds since
the conormal vector $\eta $ at a point $x\in r$, viewed at one
of the two halves of $r$, is opposite to the value of $\eta $ at the same point $x$, viewed in the other half
of $r$. Now the theorem is proved.
\end{proof}

\section{Asymptotic behavior of finite genus minimal surfaces.}
\label{asymp}
In this section, we will give the following descriptive result for the asymptotic behavior
of every properly embedded minimal surface of finite genus and infinitely many ends.

\begin{theorem}[Asymptotic Limit End Property]
\label{asympthm}
Let $M$ be a properly embedded minimal surface in $\rth$ with finite genus $g$ and
an infinite number of ends. Then, $M$ has bounded curvature and after a possible
rotation and a homothety, the following statements hold.
\begin{enumerate}
\item  $M$ has two limit ends. In fact, $M$ is conformally
diffeomorphic to  $\overline{M} -{\cal E}_M$, where  $\overline{M}$ is a compact
Riemann surface of genus $g$  and ${\cal E}_M=\{e_n \mid n\in \Z\}
\cup\{e_{-\infty},\,e_{\infty}\}$ is a countable closed subset of $\overline{M}$
with exactly two limit points
 $e_{\infty}$ and $e_{-\infty}$. Furthermore,   $\lim_{n\to -\infty}
e_n =e_{-\infty} $ and $ \lim_{n\to\infty}e_n= e_{\infty} $. The set of points
${\cal E}_M$ is called the {\bf space of ends} of $M$, the point $e_{-\infty}$ is
called the {\bf bottom end},  the point $e_{\infty}$ is called the {\bf top end} and
every point $e_n$ with $n\in \Z$ is called a {\bf middle end} of $M$.
\item For each $n \in \Z$, there exists a punctured disk neighborhood $E_n
\subset M\subset \overline{M}$ of $e_n$ which is asymptotic in $\rth$ to a
horizontal plane $\Pi_n$ and which is a graph over its projection to $\Pi_n$.
Furthermore, the usual linear ordering on the index set $\Z$ respects the linear
ordering of the heights of the related planes.
The ordered set of heights $H=\{ h_n
=x_3(\Pi_n)\mid n\in \Z\}$ of these planes naturally corresponds to the set of heights
of the middle ends of $M$.
\item There exists a positive constant $C_M$ such that if $|t|>C_M$, then the
horizontal plane $\{x_3 = t\}$ intersects $M$ in a proper arc when $t\in H$,
or otherwise, $\{x_3 = t\}$ intersects $M$ in a simple closed curve.
\item Let ${\eta}$ denote the unitary outward conormal along the boundary of $M_t=M\cap \{x_3\leq t\}$.
Then the flux vector of $M$, which is defined to be
\[
F_M =\int_{\partial M_t} \eta \, ds
\]
{\rm (}here $ds$ stands for the length element\,{\rm ),} is independent of the
choice of $t$ and has the form $F_M= (h,0,1)$, for some $h>0$.
\item Let ${\cal R}_h\subset \R^3$ be the Riemann minimal example with horizontal tangent plane at infinity and
flux vector $F=(h,0,1)$ along a compact horizontal section. Then, there exists a
 translation vector $v_{\infty}\in \rth$ such that as $t\to \infty $, the function
 \[
 d_{+}(t)=\sup \left\{ \mbox{\rm dist}(p,{\cal R}_h+v_{\infty }) \ | \ p\in M\cap \{ x_3\geq t\} \right\}
 \]
is finite and decays exponentially to zero.
In a similar manner, there exists $v_{-\infty}\in \rth$ such that as $t\to -\infty$, the function
\[
d_{-}(t)=\sup \left\{ \mbox{\rm dist}(p,{\cal R}_h+v_{-\infty }) \ | \ p\in M\cap \{ x_3\leq t\} \right\}
\]
is finite and decays exponentially to zero. Furthermore,
$x_2(v_{\infty}) = x_2(v_{-\infty})$ and for $t$ large, $M\cap \{ x_3\geq t\} $ (resp.
$M\cap \{ x_3\leq -t\} $) can be expressed as a small (with arbitrarily small $C^k$-norm for any $k$)
normal graph over its projection on ${\cal R}_h+v_{\infty }$ (resp. on ${\cal R}_h+v_{-\infty }$).
%
\end{enumerate}
\end{theorem}
{\it Proof. }
Let $M\subset \R^3$ be a properly embedded minimal surface
satisfying the hypotheses of Theorem~\ref{asympthm}. $M$ has exactly
two limit ends by Theorem~1 in~\cite{mpr4}. In this situation,
Theorem~3.5 in~\cite{ckmr1} gives that  between any two middle ends of $M$, there is a
plane that intersects $M$ transversely in a compact set. It follows that all middle ends
of $M$ are planar. Now items~1, 2, 3, 4 of the theorem follow from similar arguments
as those giving items~1, 2, 3, 4 of Theorem~\ref{thmcurvestim}.

Item~3 of Theorem~\ref{asympthm} allows us to reduce the proof of the property that
$M$ has bounded curvature to proving it for the three regions $M\cap \{ |x_3|\leq C_M\} $,
$M\cap \{ x_3> C_M\} $ and $M\cap \{ x_3\leq -C_M\} $. The region $M\cap \{ |x_3|\leq C_M\} $
consists of a finite number of graphs outside of a compact set, and thus it has bounded
Gaussian curvature. The other two regions have genus zero and one can argue similarly
as we explained in the proof of Theorem~\ref{thmcurvestim}.

It remains to prove item~5 of the theorem. The already proven items 1, 2, 3, 4
of the theorem imply that
there exists a translation vector $T\in \R^3$  such that $M^+=(M-T)
\cap \{ x_3\geq 0\} $ can be conformally parameterized
 by the half-cylinder $\C ^+/\langle i\rangle $
(here $\C^+=\{ x+iy\ | \ x\geq 0\} $) punctured in an infinite discrete set of interior
points $\{ p_j,q_j\} _{j\in \Nsmall }$, which represent respectively those ends of
$M^+$ where its Gauss map points to the north and south poles.
In this setting, the proofs of items~5, 6 , 7 of Theorem~\ref{thmcurvestim}
remain valid with $M$, $\C $, $\Z $ replaced by $M^+$, $\C ^+$, $\N $,
with the only change being in the statement of
item 7, where the last two instances of $\C /\langle i\rangle $ are left unchanged,
by choosing the third coordinate of $T$ sufficiently large. We remark
that these arguments rely solely on Colding-Minicozzi theory for planar domains.

Consider the Shiffman function $S_{M^+}$ of $M^+$, which exists by
item 3 of the theorem.
Reasoning as in the last paragraph of Section~\ref{subsecShiffman}
we deduce that $S_{M^+}$ extends smoothly through the points
$p_j,q_j$ to a function on $\C ^+/\langle i\rangle $, which we also
denote by $S_{M^+}$. Note that $S_{M^+}$ is asymptotic to zero on
the end of $\C ^+/\langle i\rangle $, because any sequence of
translations of $M^+$ diverging in height, up to a horizontal
translation (which may depend upon the sequence), converges by
Theorem \ref{mainthm} to the Riemann minimal example ${\cal R}_h$
whose flux vector is $(h,0,1)$. In particular, $S_{M^+}$ is bounded
and  can considered to be a function in $\esf ^1\times [0,\infty )$.
The next lemma gives a control on the decay of $S_{M^+}\colon
\esf^1\times [0,\infty) \to \R $.

\begin{lemma}
\label{lema9.3}
There exist $C,a>0$ so that $|S_{M^+}(\theta ,t)|\leq Ce^{-at}$, for all $(\theta ,t)\in \esf^1\times [0,\infty )$.
\end{lemma}
\begin{proof} We first prove the following assertion.
\begin{assertion}
\label{assexp}
Let $f=f(t)\colon [0,\infty )\to [0,\frac{1}{2}]$ be a continuous function such that
\begin{enumerate}
\item $f(t)\to 0$ as $t\to \infty $.
\item For any $a>0$, there exists
$t(a)\geq 0$ such that $f(t(a))\geq 2^{-at(a)}$.
\end{enumerate}
  Then for each $n \in \N$, there exists $T_n \geq n$ such that $f(t)\leq
2f(T_n)$ for
$t \in [T_n - n,\infty).$
\end{assertion}
{\it Proof of Assertion~\ref{assexp}.} Fix $n\in \N$.  Let $t_n \in [0,\infty )$ be the smallest $t$
such that $f(t)=2^{-\frac{t}{n}}$
(the existence of $t_n$ follows from $f(0)\leq \frac{1}{2}$ and from
hypothesis 2). Since $f(t_n)\leq \frac12$, then $t_n \geq n$.
Let $T_n\in [t_n,\infty )$ be a point where $f$ attains its maximum
value ($T_n$ exists by hypothesis 1). We now prove that
$f(t)\leq 2f(T_n)$ for all $t\in [T_n-n,\infty )$ by discussing two
possibilities.
\par
\vspace{.2cm} \noindent \underline{\rm Assume $T_n-n\leq t_n$.} If
$t\in [t_n,\infty )$, then $f(t)\leq f(T_n)\leq 2f(T_n)$, and if
$t\in [T_n-n,t_n]$, then $f(t)\leq 2^{-t/n}\leq
2^{\frac{-(t_n-n)}{n}}=2 \cdot 2^{-t_n/n}=2f(t_n)\leq 2f(T_n)$.
\par
\vspace{.2cm} \noindent \underline{\rm Assume $T_n-n>t_n$.} Then
$[T_n-n,\infty )\subset [t_n,\infty )$, so we apply the first two
inequalities in the case above. This completes the proof of Assertion~\ref{assexp}.
{\hfill\penalty10000\raisebox{-.09em}{$\Box$}\par\medskip}
\par
\vspace{.2cm}
We now continue the proof of the lemma. Arguing by contradiction,
suppose that $S_{M^+}(\theta ,t)$ does not decay exponentially.
Choose a constant $C>0$ such that the function
\[
f(t)=C\, |\max _{\theta \in \esfsmall ^1}S_{M^+}(\theta ,t)|
\]
satisfies $f(t) \leq \frac{1}{2}$ for all $t$.
By Assertion~\ref{assexp}, there exist sequences $T_n\in [0,\infty )$, $\theta _n\in \esf^1$ such that
$T_n\geq n$  and $C|S_{M^+}(\theta ,t)|\leq 2f(T_n)$ for every $(\theta ,t)\in \esf^1\times [T_n-n,\infty )$.
For any $n\in \N$, consider the function $h_n(\theta ,t)=\frac{C}{f(T_n)}
S_{M^+}(\theta ,t+T_n)$, defined on $\esf^1\times (-n,\infty )$. By construction,
$|h_n|\leq 2$ and $|h(\theta (n),0)|=1$ for some $\theta (n)\in \esf^1$. Therefore
after extracting a subsequence, the $h_n$ converge to a bounded Jacobi function $h_{\infty }$
on the Riemann minimal example
${\cal R}_h$, considered to be a function defined on the cylinder $\esf^1\times \R $.

By Theorem~\ref{teorJacobiacotadaenRiemanneslineal}, $h_{\infty }$
is linear, and so $h_{\infty }= \langle N_{\infty
},a\rangle $ for some $a\in \R^3-\{ 0\} $, where $N_{\infty }$ is
the Gauss map of ${\cal R}_h$. 
By the Four Vertex Theorem, the Shiffman Jacobi function $S_{M^+}$ has at least four zeros
at each compact horizontal section of $M^+$, and so, the same holds for each of the functions
$h_n$, which contradicts the assertion below and completes the proof of the lemma.
\end{proof}

\begin{assertion}
\label{ass9.5}
Given a Riemann minimal example $\cal R$ with Gauss map $N$ and a vector $a\in \mathbb{R}^3-\{0\}$,
there is a horizontal circle $\Gamma ={\cal R}\cap \{x_3=t\}$ such that the linear function
$v=\langle N,a\rangle$ has at most two zeros on $\Gamma$.
Moreover these zeros are non-degenerate.
 \end{assertion}
\begin{proof}
Consider the great circle $\g _a=\esf^2\cap \{ x\in \R^3\ | \ \langle x,a\rangle =0\} $
and a horizontal line $L\subset {\cal R}$. We can assume that $L={\cal R}\cap \{x_3=0\}$
and that $N(p_0)=(0,0,1)$, where $p_0$ denotes the end of ${\cal R}$ at level $x_3=0$.
The Gaussian image $N(L)$ consists of a twice covered geodesic arc in $\esf^2\cap \{ x_2=0\} $
whose extrema are $(0,0,1)$ with $N(p)\in \esf^2$, where $\{ p\} =L\cap \{x_2=0\}$ is the unique
branch point of $N$ along  $L$.

If $a$ is horizontal, then $\g _a$ passes through the north and south poles. As the Gauss image of
any compact horizontal circle $\G $ on ${\cal R}$ winds once around the north and south poles of $\esf^2$,
it follows that $\Gamma $ intersects $\g _a$ transversely into two points. Thus, the assertion holds
in this case.

If $a$ is not horizontal, then we discuss the following cases:
\par
\vspace{.2cm}
\noindent
{\it 1. The great circle $\g _a$ and the geodesic arc $N(L)$ are disjoint.}
In this case, we have that $v$ does not vanish along $\Gamma_t={\cal R}\cap \{x_3=t\}$
for $t>0$ small enough.
\par
\vspace{.2cm}
\noindent
{\it 2. The great circle $\g _a$ meets the interior of the geodesic arc $N(L)$.}
Then $v(p)\neq 0$ , $v(p_0)\neq 0$
 and $v$ has exactly two zeros along $L$, which are non-degenerate. It follows that $v$ has just two
non-degenerate zeros along the nearby circle $\Gamma_t$, for any small positive $t$.
\par
\vspace{.2cm}
\noindent
{\it 3. The great circle $\g _a$ passes through the point $N(p)$.}
If we parameterize $L$ by $\g (s)= p + s \, e_2$ where $e_2=(0,1,0)$,
then we have that $(v\circ \g )(0)=\langle N(p),a\rangle =0$. Also note that
$(v\circ\g)'(0)=\langle (N\circ\g)'(0),a\rangle =0$ since
$p$ is a branch point of $N$. We claim that
$(v\circ\g)''(0)\neq 0$: otherwise, $(N\circ \g )''(0)=\l \, N(p)+\mu \, a\times N(p)$
for certain $\l ,\mu \in \R $, where $\times $ denotes cross product. But
$(N\circ \g )''(0)\times N(p)\neq 0$ (because
this is the tangent component to ${\cal R}$ of $(N\circ \g )''(0)$,
and $N$ has ramification order~$1$ at $p$); hence $\mu \neq 0$,
which in turn implies that $a\times N(p)$ is orthogonal to $e_2$ (because
both $(N\circ \g )''(0)$ and  $N(p)$ are orthogonal to $e_2$). Since
$e_2$ is also orthogonal to $N(p)$, we deduce that $e_2$ is parallel
to $a$, a contradiction. Therefore, $(v\circ\g)''(0)\neq 0$. From here we conclude
that, for small positive $t$, either $v$ does not vanish along $\Gamma_t$ or it has just two distinct
simple zeros along $\G _t$.
\end{proof}

Next we prove the first statement in item~5 of Theorem~\ref{asympthm}, namely
that the exponential convergence of the top end of $M$
to a translated image of the top end of the Riemann minimal example
${\cal R}={\cal R}_h$ with
flux $(h,0,1)$ equal to the flux of $M$ (the corresponding property
for bottom ends follows similarly). During this proof,
we will make clear that the graphing property in the last statement of item~5 of the theorem also holds. Since we will use the notion of surface
written as a graph over another surface, we first make precise this notion.
We will consider minimal surfaces $\Sigma $ in a horizontal slab $\{ a\leq x_3\leq b\} $, bounded
by two Jordan curves, one in each boundary plane of the slab (in particular, $\Sigma $
is transversal to the boundary of the slab). Furthermore, $\Sigma $ will have genus zero and
(possibly) finitely many horizontal planar ends. Thus, after compactification at the planar ends,
we obtain $\overline{\Sigma }$ which is conformally a cylinder $\esf^1(r)\times [a,b]$ for certain $r>0$.
We take on $\overline{\Sigma }$ a unitary, smooth, transversal vector field $\nu $ such that
\begin{itemize}
\item $\nu |_{\partial \Sigma }$ coincides with one of the two horizontal,
normal vector fields to the planar curves that bound $\Sigma $.
\item If $\Sigma $ is non-compact, then $\nu =\pm (0,0,1)$ in a neighborhood of each of the planar
ends of $\Sigma $.
\end{itemize}
Note that $\nu $ can be thought as a deformation of the Gauss map of $\Sigma $. Although
$\nu $ is not unique, we will assume that given a surface $\Sigma $ we have made a choice of
this transversal vector field.
If we have a second surface $\Sigma '$ under the same conditions as $\Sigma $, then we say that
$\Sigma '$ is {\it a graph over $\Sigma $} if it can be written as the graph over $\Sigma $
of a function $u\in C^2(\overline{\Sigma })$, in the direction of $\nu $, i.e.,
\[
p\in \Sigma \mapsto p+u(p)\nu (p)\in \Sigma '.
\]
The notation $\| \Sigma -\Sigma '\| _{C^{k,\a }}$ that appear below will stand for
$\| u\| _{C^{k,\a }}$ (taken with respect to the flat metric in the cylinder
$\overline{\Sigma }$), where we are assuming that $\Sigma '$ is the graph
of a function $u$ over $\Sigma $ in the sense above.

All the above observations can be easily translated to
minimal surfaces in a half-space above or below a horizontal plane.
We leave the details to the reader.

By quasiperiodicity of $M^+$, we can choose a sequence of positive numbers $\{ b_n\} _{n\in \N }\nearrow \infty $ with
$M^+=M\cap \{ x_3\geq b_1\} $, satisfying the following properties.
\begin{itemize}
\item  $\{ b_{n+1}-b_n\} _n$ is bounded away from zero and bounded above.
 \item The surface $M^+$ intersects transversely the horizontal plane $\{ x_3=b_n\} $ in a Jordan curve $\G _n$, whose length is bounded above independently of $n$.
\item The surface $\Sigma _n=M\cap \{ b_n\leq x_3\leq b_{n+1}\} $ is either a compact annulus or has just one planar end. Furthermore
$\Sigma _n$ has total curvature smaller than $\pi $.
\end{itemize}
As a consequence of Lemma~\ref{lema9.3} and linear elliptic theory, we have that the derivatives of
$S_{M^+}$ of any order also decay exponentially and, from the definition of the Shiffman function we deduce
that for $n$ large, the curves $\G _n$ can be exponentially approximated by horizontal circles $\G '_n\subset \{ x_3=b_n\} $, in the sense that
\[
\| \Gamma_n-\Gamma'_n\| _{C^{4,\alpha}}\leq C_1 e^{-ab_n},
\]
for certain constant $C_1>0$ independently of $n$.

As the total curvature of $\Sigma _n$ is less than $\pi $, then $\Sigma _n$ is stable and there are no bounded Jacobi functions on $\Sigma_n$
vanishing at its boundary. An application of the Implicit Function Theorem in the Banach space context (see for instance White~\cite{wh1}
for the compact case, and P\'erez and Ros~\cite{pro4} and P\'erez~\cite{perez3} for the necessary
modifications in the case $\Sigma_n$ has a planar end) implies that
there exists an embedded minimal surface ${ R}_n$, described as a graph over $\Sigma _n$ whose boundary is $\partial R_n=
\G _{n}'\cup \G _{n+1}'$, and such that
\begin{equation}
\label{eq:Ckalpha}
\| \Sigma _n-R_n\|_{C^{2,\alpha}}\leq C_2e^{-ab_n}.
\end{equation}
Recall that the notion of graph over $\Sigma _n$ depends on the choice of a transversal vector field $\nu _n$ along $\Sigma _n$. By
the quasiperiodicity of $M^+$, we can assume that both $\Sigma _n$ and $\nu _n$ have geometry uniformly bounded in $n$, and
$\nu _{n-1}=\nu _n$ along $\G _n$. In particular,
the constant $C_2$ in (\ref{eq:Ckalpha}) can be chosen independently of $n$. In the sequel, we will
find other positive constants independent of $n$, which will be denoted by $C_3,C_4,\ldots $

Also note that $R_n$ is compact when $\Sigma_n$ is compact and $R_n$ has a horizontal planar end when $\Sigma_n$ has a planar end.
By the maximum principle, $R_n\subset \{ b_n\leq x_3\leq b_{n+1}\} $.
Furthermore, the horizontal sections of $R_n$ are either closed curves
or an open arc (this last case occurs at the height of the planar end of $R_n$, if it exists).
It also follows from (\ref{eq:Ckalpha}) that the total curvature of $R_n$
is smaller that $3\pi/2$ and so, $R_n$ is strictly stable~\cite{bc1}.
Since the Shiffman function is well-defined and bounded on $R_n$ and vanishes at $\partial R_n$,
the stability of $R_n$ implies that its Shiffman function vanishes identically. Thus, $R_n$ is a piece of a
Riemann minimal example. Furthermore, (\ref{eq:Ckalpha}) implies that the flux of the Riemann minimal example which
contains $R_n$ is exponentially close to the flux $F_{M_+}=F_{\cal R}$.

As a consequence, there exists a piece $R'_n=({\cal R}+v_n)\cap \{ b_n\leq x_3\leq b_{n+1}\} $ of a translated image of the
Riemann minimal example $\cal R$ such that $\| R_n-R'_n\|_{C^{2,\alpha}}\leq C_3 e^{-a b_n}$
for $n$ large. By the triangle inequality, we have
\begin{equation}
\label{eq:SigmaR'}
\|\Sigma_n-R'_n\|_{C^{2,\alpha}}\leq C_4 e^{-a b_n}.
\end{equation}
We next explain how to conclude all statements in item~5 of Theorem~\ref{asympthm},
except the property that $x_2(v_{\infty}) = x_2(v_{-\infty})$, which will be proven later. It is enough to prove the following statement:
\begin{assertion}
There exists a vector $v_{\infty }\in \R^3$ such that if $W_n=\{ b_n\leq x_3\leq b_{n+1}\} $, then
\[
\| M^+-({\cal R}+v_{\infty })\| _{C^{2,\a }(W_n)}\leq C_5e^{-ab_n}.
\]
\end{assertion}
\begin{proof}
By (\ref{eq:SigmaR'}) applied to $\Sigma _{n-1}$ and $\Sigma _n$ which share the common boundary curve $\G _n$
(recall that the transversal vector fields $\nu _{n-1},\nu _n$ both coincide along $\G _n$), we have that
both $({\cal R}+v_n)\cap \{ x_3=b_n\} $, $({\cal R}+v_{n-1})\cap \{ x_3=b_n\} $ are exponentially close to $\G _n$ in the norm
$\| \cdot \| _{C^{2,\a }}$. Since ${\cal R}$ is a periodic surface,
we can choose $v_n$ so that  $\| v_n-v_{n-1}\| \leq C_6e^{-ab_n}$.
The triangle inequality gives
\begin{equation}
\label{eq:series}
\| v_{n+k}-v_{n}\| \leq C_6\sum _{j=n+1}^{n+k}e^{-ab_j}.
\end{equation}
The convergence of the series $\sum _{j=1}^{\infty }e^{-ab_j}$ shows that $\{ v_n\} _n$ is a
Cauchy sequence, and so it converges to a vector $v_{\infty }\in \R^3$. Finally,
\[
\| M^+-({\cal R}+v_{\infty })\| _{C^{2,\a }(W_n)}\leq \| M^+-({\cal R}+v_n)\| _{C^{2,\a }(W_n)}+
\| ({\cal R}+v_{\infty })-({\cal R}+v_n)\| _{C^{2,\a }(W_n)}
\]
\[
\leq
C_4e^{-ab_n}+C_7\| v_{\infty }-v_n\| ,
\]
where in the last equality we have used (\ref{eq:SigmaR'}) for the first
summand and the fact that ${\cal R}+v_{\infty }$
and ${\cal R}+v_n$ differ by a small translation (namely $v_{\infty }-v_n$). Finally, (\ref{eq:series})
implies
\[
\| M^+-({\cal R}+v_{\infty })\| _{C^{2,\a }(W_n)}
\leq  C_4e^{-ab_n}+C_6\sum _{j=n+1}^{\infty }e^{-ab_j}\leq C_5e^{-ab_n},
\]
which completes the proof of the assertion (and consequently,
item~5 of Theorem~\ref{asympthm} is proved except for
the property stated in the following lemma).
\end{proof}

\begin{lemma} $x_2(v_{\infty})=x_2(v_{-\infty})$.
\end{lemma}
\begin{proof}
Recall that  ${\cal R}={\cal R}_h\subset \R^3$ is the Riemann minimal example
with the same flux vector $(h,0,1)$ as $M$. Let $T$ be the smallest
orientation preserving translation vector of ${\cal R}$, with $x_3(T)>0$.
Also, assume that ${\cal R}$ is normalized by a translation so that the
$(x_1,x_2)$-plane intersects ${\cal R}$ in a circle $\g$.
By the first statement in item~5 of Theorem~\ref{asympthm},
for $n\in \N$ large, the curve $\g(n)=M\cap
\{x_3=x_3(v_{\infty}+nT)\}$ is closely approximated by the horizontal
circle $\g+v_{\infty}+nT\subset {\cal R}+v_{\infty}$. Similarly, for $n$ large,
$\g(-n)=M\cap\{x_3=x_3(v_{-\infty}-nT)\}$ is closely approximated
by the horizontal circle $\g+v_{-\infty}-nT\subset {\cal R}+v_{-\infty}$.

As $M$ is minimal, the $\R^3$-valued one-form $\a \colon
M\rightarrow \R^3$ given by $\a _p(v)=p\times v$ for all $p\in M$
and $v\in T_pM$, has divergence zero. The Divergence Theorem applied
to $\a $ on a compact subdomain $\Omega \subset M$ gives:
\begin{equation}
\label{eq:lemaverticalslab1}
\vec{0}=\int _{\partial \Omega }\a (\eta )\, ds= \int _{\partial \Omega }p\times \eta \, ds,
\end{equation}
where $\eta $ denotes the exterior conormal field to $\Omega $ along
its boundary. We now
choose a domain $\Omega $ adapted to our setting:  For $n$ large,
label by $A(n)$ the component of $M-[\g (n)\cup \g (-n)]$ whose
boundary is $\partial A(n)=\g (n)\cup \g (-n)$. The proper domain $A(n)$ contains
a finite positive number $l(n)$ of planar
ends. For each end $e_k$ in $A(n)$, choose an embedded curve $\beta
_k\subset A(n)$ around this end, the $\beta _k$ curves being
disjoint. Finally, define $\Omega (n)$ to be the
 compact subdomain of $A(n)$
bounded by $\g (n)\cup \g (-n)\cup \beta _1\cup \ldots  \cup \beta
_{l(n)}$.
 Then, (\ref{eq:lemaverticalslab1})
can be written as
\begin{equation}
\label{eq:lemaverticalslab2} \vec{0}=\int _{\g (n)}p\times \eta \,
ds+\int _{\g (-n)}p\times \eta
 \, ds+\sum _{k=1}^{l(n)}\int _{\beta _k}p\times \eta \, ds.
\end{equation}

Each integral along $\beta _k$ in the summation above is the torque
vector associated to the end~$e_k$. This is a
horizontal vector pointing to the direction of the straight line
asymptotic to the (non-compact) level section at the height of the end
$e_k$. Thus, the third term
 in (\ref{eq:lemaverticalslab2})
will disappear after taking inner products with $e_3=(0,0,1)$. Concerning the first
integral in (\ref{eq:lemaverticalslab2}), we can estimate it as the corresponding
integral over the translated Riemann minimal
example ${\cal R}+v_{\infty }$ up to an error
$\ve _n$ such that $\ve _n\to 0$ as $n\to \infty $:
\[
\begin{array}{rcl}
{\displaystyle \int _{\g (n)}p\times \eta \, ds} & = &
{\displaystyle \ve_n+ \int _{\g+v_{\infty}+nT}p({\cal R}+v_{\infty})\times
 \eta ({\cal R}+v_{\infty}) \, ds({\cal R}+v_{\infty })}
\\
& = &
{\displaystyle \rule{0cm}{.6cm}\ve_n+ \int _{\g}[p({\cal R})+v_{\infty }+nT]\times \eta ({\cal R})\, ds({\cal R}})
\\
& = &
{\displaystyle \rule{0cm}{.6cm}\ve_n+ \int _{\g}p({\cal R})\times \eta ({\cal R})\, ds({\cal R})+
[v_{\infty}+nT]\times \mbox{Flux}(M),}
\end{array}
\]
where $\eta ({\cal R}+v_{\infty })$ is the unitary conormal vector to
${\cal R}+v_{\infty }$ along $\g +v_{\infty }+nT$ (we follow a similar notation for
$p({\cal R}), ds({\cal R})$), and we have used that
$\int _{\g }\eta ({\cal R})\, ds({\cal R})=\mbox{Flux}({\cal R})=\mbox{Flux}(M)$.

If we repeat the same argument along $\g(-n)$, then we must take into account that the
exterior conormal vectors to $\Omega (n)$ along $\g (n)$ and $\g (-n)$ are almost
opposite, thus after taking limits the conormal vector $\eta ({\cal R})$ along $\g $
in the integral of the right-hand-side of the last displayed expression associated
to $\g (n)$ is opposite to the corresponding one for $\g (-n)$.
Therefore, equation (\ref{eq:lemaverticalslab2}) implies:
\[
\begin{array}{rcl}
0&=&{\displaystyle \lim _{n\rightarrow \infty}\left\langle
 \int _{\g (n)}p\times \eta \, ds+\int _{\g (n)}p\times \eta \, ds,\ e_3\right\rangle}\\
&=& \rule{0cm}{.5cm}\left\langle \left[ (v_{\infty}-v_{-\infty}) +2nT\right] \times \mbox{Flux}(M),e_3\right\rangle
\\
&=&\rule{0cm}{.5cm}\left\langle \left[ (v_{\infty}-v_{-\infty})+2nT\right] ,\mbox{Flux}(M)\times
e_3\right\rangle.
\end{array}
\]
Since $\mbox{Flux}(M)\times e_3$ is a non-zero vector pointing to
the $x_2$-axis, we deduce that $x_2(v_{\infty})=x_2(v_{-\infty}),$ which completes
the proof of the lemma (thus, the proof of Theorem~\ref{asympthm} is complete).
\end{proof}

\begin{remark}
{\rm
By Theorem~\ref{mainthm}, any properly embedded with genus zero and an infinite number of
ends is a Riemann minimal example. In the statement of this result, it
is natural to replace the hypothesis of properness  by
completeness. The authors have found a proof of this result in the complete setting under
the additional hypothesis that the surface has countably many ends~\cite{mpr9}. We remark
that every properly embedded minimal surface in $\R^3$ has a countable number of ends,
regardless of its genus~\cite{ckmr1}.
}
\end{remark}

\section {Appendix 1: Finite dimensionality of the
space of bounded Jacobi functions.}
\label{dim}
We saw in Corollary~\ref{corol4.10} that the Shiffman function $S_M$
associated to a quasiperiodic, immersed minimal surface $M$ of Riemann type
extends smoothly through the planar ends of $M$
to a bounded, smooth function on the cylinder
$\C /\langle i\rangle $. Given such an immersed minimal surface $M$
we produced in Section~\ref{secKdV2} a sequence $\{ v_n\} _n$ of
(complex valued) Jacobi functions on $M$, one of whose terms is $S_M+iS_M^*$. A key point in the holomorphic integration of $S_M$ is that the sequence $\{ v_n\} _n$ only has a finite number of linearly independent
functions. This property follows from two facts: firstly, the fact that each Jacobi function $v_n$
extends smoothly to a bounded function on $\C /\langle i\rangle $ (Theorem~\ref{thm9.2}) and secondly,
that the bounded kernel of the Jacobi operator on a quasiperiodic, immersed minimal surface
of Riemann type is finite dimensional. This finite dimensionality can be deduced from Theorem~0.5 in  Colding, de Lellis and Minicozzi~\cite{cm39}. For the sake of completeness, we provide a direct proof of this property communicated to us by to Frank Pacard.

\begin{theorem}
\label{bounded}
Let $M\subset \R^3$ be a quasiperiodic, immersed minimal surface
of Riemann type. Then, the linear space ${\cal J}_{\infty }(M)=\{ v\in {\cal J}(M)\ | \ v \mbox{ is bounded}\} $
is finite dimensional.
\end{theorem}
\begin{proof}
By definition, $M$ is conformally equivalent to $(\C /\langle
i\rangle )-g^{-1}(\{ 0,\infty \} )$ where $g\in {\cal
M}_{\mbox{\footnotesize imm}}$ is the Gauss map of $M$. Recall that
$\C /\langle i\rangle $ is isometric to $\esf^1\times \R $. Take
global coordinates $(\theta ,t)$ on $\esf^1\times \R $ and consider
the product metric $d\theta ^2\times dt^2$, which is conformal to
the metric $ds^2$ on $M$ induced by the usual inner product of
$\R^3$. If we write $N$ for the Gauss map of $M$ and $ds^2=\l
^2(d\theta ^2\times dt^2)$, then the Jacobi operator $L=\Delta
+|\sigma |^2=\Delta +|\nabla N|^2$ of $M$ is $L=\l ^{-2}L_M$, where
$L_M=(\Delta_{\esfsmall^1}+\partial^2_t)+V_M$ is a
Schr\"{o}dinger operator on $\esf^1\times \R$ with potential $V_M$
given by the square of the norm of the differential of $N=N(\theta
,t)$ (with respect to $d\theta ^2\times dt^2$). Since $M$ is
quasiperiodic, $V_M$ is globally bounded on $\esf^1\times \R $.

By elliptic regularity, any bounded Jacobi function $v$ on $M$
extends smoothly through the zeros and poles of $g$ to a function
$\widehat{v}$ in the kernel of the operator $L_M$, such that
$\widehat{v}$ is bounded at both ends of $\esf^1\times \R $.
Therefore, the space ${\cal J}_{\infty }(M)$ of bounded Jacobi
functions identifies naturally with the bounded kernel of $L_M$. Now
the theorem is a consequence of part (1) of a more general
result proved by Colding, de Lellis and Minicozzi, namely Theorem~0.5
in~\cite{cm39}. For the sake of completeness, we provide a
direct proof of the finite dimensionality of the bounded kernel of
$L_M$ in a simpler setting which is sufficient for our purposes,
see Assertion~\ref{assloMcOw2} below. We thank Frank Pacard, who
communicated to us this argument. Modulo Assertion~\ref{assloMcOw2} and Remark~\ref{remark5.4} below,
Theorem~\ref{bounded} is proved.
\end{proof}

Let $\Sigma $ be a compact manifold endowed with a Riemannian metric
$h$, and let $\l _0=0<\l _1<\l _2<\ldots $ be the  eigenvalues of
the Laplacian $-\Delta _h$. Our goal is to give a sufficient condition
under which the bounded
kernel of the Schr\"{o}dinger operator on the metric cylinder
$(\Sigma \times \R ,h\times dt^2)$ given by
$(\Delta _h+\partial ^2_t)+V$,
where $V\in L^{\infty }(\Sigma
\times \R )$, has finite dimension. Such a condition will relate the spectral gaps $\{ \l _{j+1}-\l _j\} _j$ of $-\Delta _h$ and
$\| V\| _{L^{\infty }(\Sigma \times \R )}$, see
Assertion~\ref{assloMcOw2} below.

We first study the operator
$\Delta _h+\partial ^2_t$ acting on functions belonging to the
weighted space $e^{\de t}L^2(\Sigma \times \R )$, where
$\de $ is a real number (we will assume from now on that $\Sigma
\times \R $ is endowed with the product metric $h\times dt^2$).
The following result is a refinement
of some ideas in the paper of Lockhart and McOwen~\cite{loMcOw1}.

\begin{assertion}
\label{assloMcOw1}
Assume that $\de \in \R $ is chosen so that $\de^2\neq \l _j$ for
all $j\geq 0$. Then, if $(\Delta _h+\partial ^2_t)U=F$
with $U,F\in e^{\de t}L^2(\Sigma \times \R )$, we have
\begin{equation}
\label{eq:Pacard1}
\| e^{-\de t}U\| _{L^2(\Sigma \times \R )}\leq \frac{1}{\inf _j
|\de ^2-\l _j|}\| e^{-\de t}F\| _{L^2(\Sigma \times \R )}.
\end{equation}
\end{assertion}
\begin{proof}
First observe that the functions $u=e^{-\de t}U$, $f=e^{-\de t}F$ belong to $L^2(\Sigma \times \R )$ and $f=e^{-\de t}(\Delta _h+
\partial _t^2)(e^{\de t}u)$. We perform the Fourier transform of $t\mapsto
u(y,t)$ and $t\mapsto f(y,t)$ for $y\in \Sigma $ fixed, defining the complex valued functions
\[
\widehat{u}(y,s)=\frac{1}{\sqrt{2\pi }}\int _{\R }u(y,t)e^{ist}dt,
\qquad
\widehat{f}(y,s)=\frac{1}{\sqrt{2\pi }}\int _{\R }f(y,t)e^{ist}dt,
\]
for all $(y,s)\in \Sigma \times \R $. To keep notations short, we set
$z:=\de -is\in \C $. 
It is straightforward to check that given an $s\in \R$,
the functions $\widehat{u}(\cdot,s)$, $\widehat{f}(\cdot ,s )$ satisfy (in the sense of distributions) the equation
\[
(\Delta _h+z^2)\widehat{u}(\cdot ,s)=\widehat{f}(\cdot ,
s) \quad \mbox{on }\Sigma .
\]
Given $z\in \C $, consider the linear Schr\"{o}dinger operator
$\widehat{B}_{z}=\Delta _h+z^2$,
acting on complex valued functions on $\Sigma $. Let us denote by
$E_0,E_1,E_2,\ldots $ the eigenspaces of $-\Delta _h$ corresponding
to the eigenvalues $\l _0=0<\l _1<\l _2<\ldots $, respectively.
By classical elliptic theory, if $z^2\neq \l _j$ for all $j\in \N\cup \{ 0\} $,
then there exists a bounded operator
$\widehat{R}_z:L^2(\Sigma )\to L^2(\Sigma )$
which is a right inverse of $\widehat{B}_z$, i.e.,
$\widehat{B}_z\circ \widehat{R}_z$ is the identity on $L^2(\Sigma )$,
where we keep the notation $L^2(\Sigma )$ for $L^2$-complex valued functions
on  $\Sigma $.
Also note that the condition $z^2\neq \l _j$ holds for all $j$ since
$|z^2-\l _j|\geq |\de ^2-\l _j|>0$.

Using the orthogonal eigendata decomposition $\widehat{f}(\cdot ,s)=
\sum _{j\geq 0}\sum _{\widehat{f}_h\in E_j}\widehat{f}_h$, it is
easy to check that
\[
\widehat{R}_z(\widehat{f}(\cdot ,s))=\sum _{j\geq 0}\frac{1}{z^2-\l _j}
\sum _{\widehat{f}_h\in E_j}\widehat{f}_h.
\]
Plancherel's formula then implies
\[
\| \widehat{R}_z(\widehat{f}(\cdot ,s))\| ^2_{L^2(\Sigma )}=
\sum _{j\geq 0}\frac{1}{|z^2-\l _j|^2}\sum _{\widehat{f}_h\in E_j}
\| \widehat{f}_h\| ^2_{L^2(\Sigma )}.
\]
Using the inequality $|z^2-\l _j|\geq |\de ^2-\l _j|$, we obtain
\begin{equation}
\label{eq:Pacard2}
\| \widehat{R}_z(\widehat{f}(\cdot ,s))\| ^2_{L^2(\Sigma )}
\leq
\frac{1}{\inf _j|z^2-\l _j|^2}
\sum _{j\geq 0}
\sum _{\widehat{f}_h\in E_j}
\| \widehat{f}_h\| ^2_{L^2(\Sigma )}
\leq
\frac{1}{\inf _j|\de^2-\l _j|^2}
\| \widehat{f}(\cdot ,s)\| ^2_{L^2(\Sigma )}.
\end{equation}
Note that $\widehat{R}_z(\widehat{f}(\cdot ,s))
=\widehat{u}(\cdot ,s)$ (because $\widehat{B}_z\widehat{u}=\widehat{f}$ and
$z^2\neq \l _j$ for all $j$). Since the Fourier transform is an
isometry of $L^2(\Sigma )$, one has
\[
\begin{array}{rcl}
\| u\| ^2_{L^2(\Sigma \times \R )}&=&
{\displaystyle \| \widehat{u}\| ^2_{L^2(\Sigma \times
\R )}=\int _{\R }\| \widehat{u}(\cdot ,s)\| ^2_{L^2(\Sigma )}ds=
\int _{\R }\| \widehat{R}_z(\widehat{f}(\cdot ,s))
\| ^2_{L^2(\Sigma )}ds}
\\
&\stackrel{(\ref{eq:Pacard2})}{\leq }&
{\displaystyle
\frac{1}{\inf _j|\de^2-\l _j|^2}\int _{\R }
\| \widehat{f}(\cdot ,s)\| ^2_{L^2(\Sigma )}ds
=\frac{1}{\inf _j|\de^2-\l _j|^2}\| \widehat{f}
\| ^2_{L^2(\Sigma \times \R )}}
\\
&=&
{\displaystyle
\frac{1}{\inf _j|\de^2-\l _j|^2}\| f\| ^2_{L^2(\Sigma \times
\R )},}
\end{array}
\]
from which one deduces directly the inequality (\ref{eq:Pacard1}). Hence, the assertion is proved.
\end{proof}

\begin{assertion}
\label{assloMcOw2}
Let $(\Sigma ,h)$ be a compact Riemannian manifold and $V\in L^{\infty }
(\Sigma \times \R )$. Assume that there exists $j_0\in \N$ such that
\begin{equation}
\label{eq:Pacard3}
4\| V\| _{L^{\infty }(\Sigma \times \R )}\leq \l _{j_0+1}-\l _{j_0},
\end{equation}
where $\l _0=0<\l _1<\l _2<\ldots $ is the spectrum of $-\Delta _h$
on $\Sigma $. Then, the bounded kernel of $\Delta _h+\partial ^2_t+V$
on $\Sigma \times \R $ is finite dimensional.
\end{assertion}
\begin{remark}
\label{remark5.4}
{\rm
In the case where $\Sigma =\esf^1$ with its standard metric, then
$\l _j=j^2$ and $\l _{j+1}-\l _j=2j+1$, so the hypothesis (\ref{eq:Pacard3}) is always fulfilled.
}
\end{remark}
\begin{proof}
Again the proof is essentially borrowed from the paper by
Lockhart and McOwen~\cite{loMcOw1} (see Section 2 on pages 420, 421).
First consider a
function $w$ in the bounded kernel of $\Delta _h+\partial ^2_t+V$.
Let $\chi \in C^{\infty }(\Sigma \times \R )$ be a cutoff function
only depending on $t$, equal to $0$ for $t<-1$ and equal to $1$ for
$t>1$. We choose $\de >0$ such that
\[
\de^2=\frac{1}{2}(\l _{j_0+1}+\l _{j_0}),
\]
in particular $\de^2\neq \l _j$ for all $j\geq 0$.

Applying the result in Assertion~\ref{assloMcOw1} to this value of
$\de $ and to the functions $U:=\chi w$, $
F:=(\Delta _h+\partial ^2_t)(\chi w)=-V\chi w+w\partial ^2_t\chi
+2\partial _t\chi \, \partial _tw$, we obtain
\[
\begin{array}{rcl}
{\displaystyle
\| e^{-\de t}\chi w\| _{L^2(\Sigma \times \R )}\hspace{-0.2cm}}
&\leq \hspace{-0.2cm}&
{\displaystyle
\frac{1}{\inf _j|\de^2-\l _j|}\| e^{-\de t}F\| _{L^2(\Sigma \times
\R )}
}
\\
& \leq \hspace{-0.2cm}&
{\displaystyle
\frac{1}{\inf _j|\de^2-\l _j|}\left( \| e^{-\de t}V\chi w\| _{L^2(\Sigma \times \R )}
+\right.
}
\\
& &
{\displaystyle
\hspace{2.5cm}
\left. \| e^{-\de t}w\partial ^2_t\chi \| _{L^2(\Sigma \times \R )}
+\| e^{-\de t}2\partial _t\chi \, \partial _tw\| _{L^2(\Sigma \times \R )}\right )
}
\\
& \leq \hspace{-0.2cm}&
{\displaystyle
\frac{1}{\inf _j|\de^2-\l _j|}\left(
\| V\| _{L^{\infty }(\Sigma \times \R )}\, \| e^{-\de t}
\chi w\| _{L^2(\Sigma \times \R )}
+c\| w\| _{W^{1,2}(\Sigma \times (-1,1))}\right ) ,
}
\end{array}
\]
for some constant $c>0$ which only depends on the choice of the
cutoff function $\chi $ (note that we have used that
$\partial _t\chi =0$ outside $\Sigma \times (-1,1)$ in the
last inequality).

On the other hand, (\ref{eq:Pacard3}) gives that
\[
\inf _j|\de ^2-\l _j|=\frac{1}{2}(\l _{j_0+1}-\l _{j_0})
\geq 2\| V\| _{L^{\infty }(\Sigma \times \R )}.
\]
Thus,\quad  ${\textstyle \| e^{-\de t}\chi w\| _{L^2(\Sigma \times \R )}
\leq \frac{1}{2}\| e^{-\de t}\chi w\| _{L^2(\Sigma \times
\R )}+\frac{2c}{\l _{j_0+1}-\l _{j_0}}\| w\| _{W^{1,2}(\Sigma \times (-1,1))}}$,
which implies
\begin{equation}
\label{eq:Pacard5}
\| e^{-\de t}\chi w\| _{L^2(\Sigma \times \R )}
\leq
\frac{4c}{\l _{j_0+1}-\l _{j_0}}\| w\| _{W^{1,2}(\Sigma \times (-1,1))}
\leq
C\| w\| _{W^{1,2}(\Sigma \times (-1,1))},
\end{equation}
for some constant $C>0$ that depends on $\l _{j_0},\l _{j_0+1}$ and on
$\chi $ (but these data have been fixed once for all).

Next we apply Assertion~\ref{assloMcOw1} to the functions
$U:=(1-\chi )w$,
$F:=(\Delta _h+\partial ^2_t)((1-\chi )w)=-V(1-\chi )w-w\partial ^2_t\chi
-2\partial _t\chi \, \partial _tw$ and this time we change $\de $ to
$-\de $, to get
\[
\begin{array}{rcl}
{\displaystyle
\| e^{\de t}(1-\chi )w\| _{L^2(\Sigma \times \R )}\hspace{-0.2cm}}
&\leq \hspace{-0.2cm}&
{\displaystyle
\frac{1}{\inf _j|\de^2-\l _j|}\left(
\| V\| _{L^{\infty }(\Sigma \times \R )}\| e^{\de t}(1-\chi )
w\| _{L^2(\Sigma \times \R )}\right.
}
\\
& & {\displaystyle
\left. \hspace{4cm}\rule{0cm}{.4cm}
+c\| w\| _{W^{1,2}(\Sigma \times (-1,1))}\right) }.
\end{array}
\]
Arguing as above, we conclude that
\begin{equation}
\label{eq:Pacard6}
\| e^{\de t}(1-\chi )w\| _{L^2(\Sigma \times \R )}
\leq C\| w\| _{W^{1,2}(\Sigma \times (-1,1))}.
\end{equation}
Collecting (\ref{eq:Pacard5}) and (\ref{eq:Pacard6}) and using the
triangle inequality, we deduce that
\begin{equation}
\label{eq:Pacard7}
\| e^{-\de |t|}w\| _{L^2(\Sigma \times \R )}
\leq 2C\| w\| _{W^{1,2}(\Sigma \times (-1,1))}
\end{equation}
for certain $C>0$, and this estimate holds for any function $w$ in
the bounded kernel of $\Delta _h+\partial _t^2+V$.

Another ingredient we will use in proving the assertion is a
$W^{2,2}$-estimate valid for any function in the kernel
of $\Delta _h+\partial _t^2+V$. Namely, the classical $L^p$-estimates applied to the solution $w$ of
$(\Delta _h+\partial _t^2)w=-Vw$ (see for instance Theorem~9.11
of Gilbarg and Trudinger~\cite{gt1}), imply that
\begin{equation}
\label{eq:Pacard8}
\| w\| _{W^{2,2}(\Sigma \times (-1,1))}\leq C\left(
1+\| V\| _{L^{\infty }(\Sigma \times \R )}\right)
\| w\| _{L^2(\Sigma \times (-2,2))}.
\end{equation}
The final ingredient is the compactness of the Sobolev embedding
(Rellich's theorem, see for instance~\cite{gt1} Theorem 7.26)
\begin{equation}
\label{eq:Pacard9}
W^{2,2}(\Sigma \times (-1,1))\hookrightarrow W^{1,2}(\Sigma \times (-1,1)).
\end{equation}

We finally prove the finite dimensionality of the bounded kernel $K$
of $\Delta _h+\partial ^2_t+V$. Arguing by contradiction, suppose
that $K$ were infinite dimensional. There would exist a sequence of
linearly independent functions $w_n\in K$, which can be normalized
so that
\begin{equation}
\label{eq:Pacard10}
\int _{\Sigma \times \R }w_nw_me^{-2\de |t|}\, dt\, dh=\left\{ \begin{array}{ll}
0 & \mbox{if }n\neq m,
\\
1 & \mbox{if }n=m.
\end{array}\right.
\end{equation}

Recall that $\de >0$ and hence the integrals are well-defined since
the functions $w_n$ are bounded. As $e^{-\de |t|}$ is bounded away from zero in $(-2,2)$, we conclude from (\ref{eq:Pacard10}) that the
sequence $\{ w_n\} _n$ is bounded in $L^2(\Sigma \times (-2,2))$,
and (\ref{eq:Pacard8}) then implies that it is also bounded in $W^{2,2}
(\Sigma \times (-1,1))$. Now, using the compactness of the embedding
(\ref{eq:Pacard9}), there exists a subsequence (still denoted by
$\{ w_n\} _n)$ which converges in $W^{1,2}(\Sigma \times (-1,1))$.

The convergence of $\{ w_n\} _n$ in $W^{1,2}(\Sigma \times (-1,1))$
together with (\ref{eq:Pacard7}) applied to $w_n-w_m$
implies that $\{ e^{-\de |t|}w_n\} _n$ is a Cauchy sequence in $L^2
(\Sigma \times \R )$. But, $L^2(\Sigma \times \R )$ being a Banach
space, this sequence converges in $L^2(\Sigma \times \R )$ to some
function $W\in L^2(\Sigma \times \R )$. Passing to the limit
as $n\to \infty $ in (\ref{eq:Pacard10}) (keeping $m$ fixed in the first integral),
we have
\begin{equation}
\label{eq:Pacard11}
\int _{\Sigma \times \R }Ww_me^{-\de |t|}\, dt\, dh=0,
\qquad
\int _{\Sigma \times \R }W^2\, dt\, dh=1.
\end{equation}
Finally, passing to the limit as $m\to \infty $ in the left integral
above, one obtains
\[
\int _{\Sigma \times \R }W^2\, dt\, dh=0,
\]
which contradicts the right integral in (\ref{eq:Pacard11}). This finishes the proof of the assertion.
\end{proof}

From the descriptions in Theorems~\ref{thmcurvestim} and~\ref{thm1finallimite}, it follows that every properly embedded minimal surface $M$
in $\R^3$ with finite genus $g$ and infinite topology is conformally diffeomorphic to a compact Riemann surface of genus
$g$ with a countable set of points removed, and this set of points has exactly two limit points on the compact surface.  Furthermore, $M$ has bounded curvature, each middle end  is planar and $M$ has two limit ends of Riemann type
(see Theorem~\ref{asympthm} for  an improved  description of such an $M$).  In particular, $M$ has a partial conformal compactification $\overline{M}$ by adding  its non-limit ends, and a complete metric on $\overline{M}$
with two cylindrical ends. As in the previously considered case when the genus of $M$ was zero,  the bounded Jacobi functions of $\overline{M}$  can be  identified with the bounded kernel of $\Delta_{\overline{M}}+V$, where
$V$ is a bounded potential.  In this case, the
arguments in the proof of Theorem~\ref{bounded}  can be easily modified and imply the result below. We remark
that when $M$ has $0<k<\infty $ ends, then $M$ has finite total curvature and the statement below is well-known.
We also note that Theorem~\ref{thm5.5} is a particular case of the more general result in Theorem~0.5 in
Colding, de Lellis and Minicozzi~\cite{cm39}.

\begin{theorem}
\label{thm5.5}
Let $M\subset \R^3$ be a properly embedded minimal surface with finite genus and more than one end. Then,
the linear space of bounded Jacobi functions on $M$ is finite dimensional.
\end{theorem}

\section{Appendix 2: Uniqueness of the Riemann minimal examples in the 2-ended periodic case.}
\label{app2}
Recall that $\mathcal{M}_1\subset \mathcal{M}$ is the subset of singly-periodic
surfaces that define a two-ended torus in their orientable quotient by a translation with smallest absolute total curvature. In Proposition~\ref{propos7.3} of this paper we showed that if $S_M$ is linear for a surface $M\in \mathcal{M}$, then $M\in \mathcal{M}_1$. At this point we can conclude that $M\in \mathcal{R}$ by quoting our previous characterization in~\cite{mpr1}. We give below an alternative proof of the characterization $\mathcal{M}_1=\mathcal{R}$ based on the Four Vertex Theorem.

Consider the flux map $h_1=h|_{{\cal M}_1}\colon {\cal M}_1\to (0,\infty )$ where $F_M=(h(M),0,1)$ is the
flux vector of $M\in {\cal M}_1$. We first prove that ${\cal R}$ is a connected component of ${\cal M}_1$.
Since ${\cal R}$ is a path connected closed set in ${\cal M}_1$, it remains to prove that ${\cal M}_1-{\cal R}$
is closed in ${\cal M}_1$. Otherwise, there exists a sequence of surfaces $\{\Sigma _n\} _n\subset {\cal M}_1-
{\cal R}$ which converges on compact subsets of $\R^3$ to some ${\cal R}_t\in {\cal R}$. Note that
the Shiffman functions $S_{\Sigma _n}$ of the $\Sigma _n$ are not identically zero, and after the normalization
$\widehat{S}_{\Sigma _n}=\frac{1}{\sup _{\Sigma _n}|S_{\Sigma _n}|}S_{\Sigma _n}$,
we find a bounded sequence of Jacobi functions which converges (up to extracting a subsequence) to a
periodic Jacobi function $\widehat{S}_{\infty }$ on ${\cal R}_t$. By the Four Vertex Theorem,
$\widehat{S}_{\Sigma _n}$
has at least four zeros on each compact horizontal section of ${\Sigma _n}$ (counted with
multiplicity), and the same holds for $\widehat{S}_{\infty }$ on each compact horizontal section of ${\cal R}_t$.
On the other hand, the only periodic Jacobi functions on ${\cal R}_t$ are the linear ones (this follows, for
instance, from Montiel and Ros~\cite{mro1} and also follows from our more general result in
Theorem~\ref{teorJacobiacotadaenRiemanneslineal}). This contradicts Assertion~\ref{ass9.5} and proves
that ${\cal R}$ is a connected component of ${\cal M}_1$.

Since ${\cal R}$ is a connected component of ${\cal M}_1$ and
$h|_{\cal R}\colon {\cal R}\to (0,\infty )$ is bijective, to deduce that
${\cal R}={\cal M}_1$ it suffices to prove that the following three properties.
\begin{enumerate}
\item $h_1$ is a proper map.
\item $h_1$ is an open map.
\item There exists $\ve >0$ such that if $h_1(M)<\ve $, then $M\in {\cal R}$.
\end{enumerate}
The properness of $h_1$ in point 1 above follows from the curvature estimates in Theorem~5 of~\cite{mpr3},
which in fact insures properness of $h\colon {\cal M}\to (0,\infty )$.
Both the openness point 2 and the local uniqueness point 3 above follow from arguments in~\cite{mpr1},
but we give simpler arguments below; once we prove these two points, then Proposition~\ref{propos7.3} will hold.

We first prove the openness of $h_1$ in point 2 above. Consider the space
${\cal W}_1=\{ (\Sigma ,g,[\a ])\} $ where $\Sigma $ is a compact Riemann surface of genus one,
$g\colon \Sigma \to \C \cup \{ \infty \} $ is a meromorphic function of degree two with
an order-two zero $p$ and an order-two pole $q$, and $[\a ]$ is a homology class in $H_1(\Sigma -\{ p,q\} ,\Z )$
which is non-trivial in $H_1(\Sigma ,\Z )$. We denote
the elements in ${\cal W}_1$  simply by $g$. The space ${\cal W}_1$ is a two-dimensional complex manifold, with
local charts given by $g\mapsto (a_1+a_2,a_1\cdot a_2)$, where $a_1,a_2\in \C-\{ 0\} $ are
the (possibly equal) branch values of $g\in {\cal W}_1$ close to a given element $g_0\in {\cal W}_1$
(in a chart, we can forget about the homology class associated to $g$ after identification with that of $g_0$).
Given $g\in {\cal W}_1$, we associate a unique holomorphic differential $\phi $ on $\Sigma $ by the equation
$\int _{\a }\phi =2\pi i$. Consider the {\it period map} $\mbox{Per}_1\colon {\cal W}_1\to \C^2$ given by
\[
\mbox{Per}_1(g)=\left( \int _{\a }\frac{1}{g}\phi ,\int _{\a }g\, \phi \right) .
\]
Then, the space of elements $g\in {\cal W}_1$ such that $(g,\phi )$ is the Weierstrass pair of an
immersed minimal surface are ${\cal M}_{1}^{\mbox{\footnotesize imm}}=
\mbox{Per}_1^{-1}(\{ (a,\overline{a})\ | \ a\in \C \} )$
(note that we do not need to impose any residue condition at
the ends, since $g$ has a unique zero and a unique pole and the sum of residues of a meromorphic differential
on a compact Riemann surface is zero). Since $\mbox{Per}_1$ is holomorphic, for $a\in \C $ fixed the set
${\cal M}_1^{\mbox{\footnotesize imm}}(a)=\mbox{Per}_1^{-1}(a)$ is a complex
analytic subvariety of ${\cal W}_1$. Since the limit of embedded surfaces is embedded,
we have that the subset ${\cal M}_1 \subset {\cal M}_{1}^{\mbox{\footnotesize imm}}$ of
{\it embedded} surfaces is closed in ${\cal M}_{1}^{\mbox{\footnotesize imm}}$.
An application the maximum principle at infinity~\cite{mr1} gives that ${\cal M}_1$ is also open in
${\cal M}_{1}^{\mbox{\footnotesize imm}}$. In particular, the set
${\cal M}_1(a)={\cal M}_{1}^{\mbox{\footnotesize imm}}(a)\cap {\cal M}_1$ is a complex analytic subvariety
of ${\cal W}_1$. By our uniform curvature estimates in
Theorem~5 of~\cite{mpr3} and subsequent uniform local area estimates,
 ${\cal M}_1(a)$ is compact. As the only compact, complex analytic subvarieties
of ${\cal W}_1$ are finite sets (see Lemma~4 in~\cite{mpr1}),
we deduce that ${\cal M}_1(a)$ is finite. Thus, given $M\in {\cal M}_1$,
there exists an open neighborhood $U$ of $M$ in ${\cal W}_1$ such that $U\cap {\cal M}_{1}(a)=
U\cap {\cal M}_{1}^{\mbox{\footnotesize imm}}(a)=\{ M\} $. In this setting, the openness
theorem for finite holomorphic maps (Chapter 5.2 of Griffiths and Harris~\cite{GriHar}) gives that
$\mbox{Per}_1$ is an open map locally around $M$. Finally, the relationship between the period map
$\mbox{Per}_1$ and the flux map $h_1\colon {\cal M}_1\to (0,\infty )$ gives the
desired {\bf openness} for $h_1$.

Our next statement is the local uniqueness point 3 in the list
of properties of $h_1$ (and thus, it finishes the proof of
Proposition~\ref{propos7.3}). In fact, the argument below does not use periodicity
for the surfaces in question, so it can be stated in ${\cal M}$ instead of in ${\cal M}_1$.

\begin{theorem}
\label{thmFalmostvertRiemann}
There exists $\ve >0$ such that of $M\in {\cal M}$ has flux vector $F_M=(h,0,1)$ with $0<h<\ve $,
then $M$ is a Riemann minimal example.
\end{theorem}
\begin{proof}We will present here a different proof from the one we gave in~\cite{mpr1}.
Arguing by contradiction, assume we have a sequence $\{ M_n\} _n \subset {\cal M}$ with
flux vector $F_{M_n}=(h(M_n),0,1)$ and $h(M_n)\to 0$ as $n\to \infty $, and assume that none of the $M_n$
are Riemann minimal examples. Point~5 in
Theorem~\ref{thmcurvestim} insures that there exists a uniform bound for the Gaussian curvature
of the surfaces $M_n$, $n\in \N $. A suitable modification of the arguments in the proof of Lemma~3
in~\cite{mpr1} can be used to show that as $n \rightarrow \infty$, the surfaces
$M_n$  become arbitrarily close to an infinite discrete collection of larger and
larger translated pieces of a vertical catenoid with flux $e_3=(0,0,1)$ joined by
flatter and flatter graphs containing the ends of $M_n$. For each $n$, let
$\overline{M}_n$ be the conformal cylinder obtained by attaching the middle ends to
$M_n$, and let $S_{M_n}$ be the Shiffman function of $M_n$.

By Corollary~\ref{corol4.10}, $S_{M_n}$ extends smoothly to a bounded function on $\overline{M}_n$.
Note that for fixed $n$, the function $|S_{M_n}|$ needs not attain its maximum on
$M_n$, but in that case we can exchange each $M_n$ by a limit of suitable
translations of $M_n$ (hence such limit also belongs to ${\cal M}$), so that the Shiffman function
in absolute value reaches its maximum on this limit. Since the flux of a surface in ${\cal M}$
does not change under translations, we do not lose generality by assuming that for all $n$ large,
$|S_{M_n}|$ attains its maximum at a point $p_n\in M_n$. We now define
$v_n=\frac{1}{|S_{M_n}(p_n)|}S_{M_n}$.

Take a sequence $\{ \de (n)\} _n\subset (0,1)$ converging to $1$. For $n$ large, let
$C_n\subset M_n$ be one of the connected components of $\langle N_n,e_3\rangle
^{-1}[-\de (n),\de (n)]$ which contains $p_n$ or is adjacent to a horizontal
graphical region containing $p_n$.
 By our previous arguments, $C_n$ is arbitrarily close to a
translated image of the intersection of a vertical catenoid $C_{\infty }$ of
vertical flux $e_3$ centered at the origin with a ball of arbitrarily large radius
also centered at the origin.
\begin{assertion}
\label{assSchifnormalizedgoesto0oncat} $\{ \sup _{C_n}|v_n|\} _n$ tends to
zero as $n\to \infty $.
\end{assertion}
{\it Proof of Assertion~\ref{assSchifnormalizedgoesto0oncat}.} Since
$\{ v_n|_{C_n}\} _n$ is a bounded sequence of Jacobi functions on the forming
catenoidal pieces $C_n$ and suitable
translations of the $C_n$ converge to the catenoid $C_{\infty }$, it is not
difficult to check that a subsequence of $\{ v_n|_{C_n}\} _n$ (denoted in the same
way) converges to a bounded Jacobi function on $C_{\infty }$. Since bounded Jacobi
functions on a catenoid are linear, we conclude that $\{ v_n|_{C_n}\} _n$ converges
to a linear Jacobi function $v$ on $C_{\infty }$ (or by identifying the compactification
of $C_{\infty }$ with the sphere $\esf^2$ through the Gauss map of $C_{\infty }$,
we can view $v$ as a linear function
on $\esf ^2$). We now check that $v$ is identically zero on $\esf^2$.

Arguing by contradiction, suppose $v$ is not identically zero on $\esf^2$. Recall that the
Shiffman function $(S_{M_n})|_{C_n}$ measures the derivative of the curvature of each
planar section of $C_n$ with respect to a certain parameter times a positive
function. By the Four Vertex Theorem, each horizontal section of $C_n$ contains at
least four zeros of $S_{M_n}$ and so, also at least four zeros of $v_n$. Since
horizontal sections of the $C_n$ (suitably translated) converge to horizontal
sections of $C_{\infty }$ and any non-trivial linear function on $\esf ^2$ has at
most two zeros on each horizontal section (with a possible exceptional horizontal
section if the linear function is the vertical coordinate, but this does not affect
our argument by taking a different horizontal section), we conclude that at least
two zeros of $v_n$ in a certain horizontal section must collapse into a zero of $v$;
hence the gradient of $v$ will vanish at such a collapsing zero. But the gradient of
a non-trivial linear function on $\esf^2$ never vanishes at a zero of the function.
This contradiction proves Assertion~\ref{assSchifnormalizedgoesto0oncat}.

Recall that $|v_n(p_n)|=1$ for all $n$. By
Assertion~\ref{assSchifnormalizedgoesto0oncat}, $N_n(p_n)$ must converge to
the vertical or equivalently, $p_n$ must lie in one of the graphical components of
the complement of all the catenoidal pieces in $M_n$, a non-compact minimal graph
which we will denote by $\Omega _n$. Note that $\Omega _n$ is a graph over an
unbounded domain in the plane $\{ x_3=0\} $, $\partial \Omega _n$ consists of two
almost-circular, almost-horizontal curves with $\langle N_n,e_3\rangle |_{\partial
\Omega _n}= \pm \de (n)$ and $\Omega _n$ contains exactly one end of $M_n$. Hence we
can apply Lemma~\ref{lemaprincmaxJacobi} below to the minimal surface $\Omega _n$ and to the bounded
Jacobi function $v_n|_{\Omega _n}$, contradicting that $v_n|_{\partial \Omega _n}$
converges to zero (Assertion~\ref{assSchifnormalizedgoesto0oncat}) but
$|v_n(p_n)|=1$. This contradiction finishes the proof of Theorem~\ref{thmFalmostvertRiemann}.
\end{proof}

\begin{lemma}
\label{lemaprincmaxJacobi}
 Let $\de \in (0,1)$ and let $\Omega \subset \R^3$  be a complete, non-compact minimal surface
with non-empty compact boundary and finite total curvature, whose Gauss map
$N$ satisfies $N_3=\langle N,e_3\rangle \geq 1-\de $ in $\Omega $. Then, for every
bounded Jacobi function $v$ on $\Omega $,
\[
(1-\de )\sup _{\Omega }|v|\leq \sup _{\partial \Omega }|v|.
\]
\end{lemma}
\begin{proof}
Since $\Omega $ has finite total curvature, $\Omega $ compactifies after attaching
its ends to a compact Riemann surface $\overline{\Omega }$ with boundary. As $v$
is bounded on $\Omega $, $v$ extends smoothly across the punctures to a Jacobi
function on $\overline{\Omega }$. We will let $a=\sup _{\partial \Omega }|v|$. Since
$N_3\geq 1-\de >0$ in $\Omega $ and $N_3$ is a Jacobi function, we conclude that $\Omega $
is strictly stable and so, $a>0$. Now, $v+\frac{a}{1-\de }N_3\geq 0$ on $\partial \Omega $ and
$v+\frac{a}{1-\de }N_3$ is again a Jacobi function on $\Omega $. Thus by stability,
$v+\frac{a}{1-\de }N_3\geq 0$ in $\Omega $. Analogously, $v-\frac{a}{1-\de }N_3\leq 0$ in
$\partial \Omega $, and hence $v-\frac{a}{1-\de }N_3\leq 0$ in $\Omega $. These inequalities
together with $N_3\leq 1$ give $|v|\leq \frac{a}{1-\de }$ in $\Omega $, as desired.
\end{proof}


\center{William H. Meeks, III at profmeeks@gmail.com\\
Mathematics Department, University of Massachusetts, Amherst, MA
01003}
\center{Joaqu\'\i n P\'{e}rez at jperez@ugr.es\\
Department of Geometry and Topology, University of Granada, 18071
Granada, Spain}
\center{Antonio Ros at aros@ugr.es\\
Department of Geometry and Topology, University of Granada, 18071
Granada, Spain}

\bibliographystyle{plain}

\bibliography{bill}

\end{document}